\documentclass[11pt,reqno]{amsart}
\usepackage[letterpaper,margin=1in,footskip=0.25in]{geometry}
\usepackage{mathrsfs}
\usepackage{microtype}
\usepackage{mathtools}
\usepackage{enumitem}
\usepackage[noadjust]{cite}
\renewcommand{\citepunct}{;\penalty\citemidpenalty\ }
\usepackage{longtable,booktabs,multirow}
\usepackage[referable]{threeparttablex}
\usepackage{caption}

\usepackage{tikz-cd}

\PassOptionsToPackage{pdfusetitle,colorlinks,pagebackref,linktoc=all}{hyperref}
\usepackage{bookmark}
\hypersetup{citecolor=[HTML]{2E7E2A},linkcolor=[HTML]{800006},urlcolor=[HTML]{8A0087}}
\usepackage[hyphenbreaks]{breakurl}

\numberwithin{equation}{section}

\newtheorem{theorem}[equation]{Theorem}
\newtheorem{lemma}[equation]{Lemma}
\newtheorem{proposition}[equation]{Proposition}

\newtheorem{alphthm}{Theorem}

\theoremstyle{definition}
\newtheorem{definition}[equation]{Definition}
\newtheorem{axiom}[equation]{Axiom}
\newtheorem{notation}[equation]{Notation}

\theoremstyle{remark}
\newtheorem{remark}[equation]{Remark}
\newtheorem{examples}[equation]{Examples}

\newtheoremstyle{cited}{.5\baselineskip\@plus.2\baselineskip\@minus.2\baselineskip}{.5\baselineskip\@plus.2\baselineskip\@minus.2\baselineskip}{\itshape}{}{\bfseries}{\bfseries .}{5pt plus 1pt minus 1pt}{\thmname{#1}\thmnumber{ #2}\thmnote{ \normalfont#3}}
\theoremstyle{cited}
\newtheorem{citedprop}[equation]{Proposition}

\newtheoremstyle{citeddef}{.5\baselineskip\@plus.2\baselineskip\@minus.2\baselineskip}{.5\baselineskip\@plus.2\baselineskip\@minus.2\baselineskip}{}{}{\bfseries}{\bfseries .}{5pt plus 1pt minus 1pt}{\thmname{#1}\thmnumber{ #2}\thmnote{ \normalfont#3}}
\theoremstyle{citeddef}
\newtheorem{citeddef}[equation]{Definition}
\newtheorem{citedaxiom}[equation]{Axiom}
\newtheorem{citedaxioms}[equation]{Axioms}

\newtheoremstyle{step}{.25\baselineskip\@plus.1\baselineskip\@minus.1\baselineskip}{.25\baselineskip\@plus.1\baselineskip\@minus.1\baselineskip}{\itshape}{}{\bfseries}{\bfseries .}{5pt plus 1pt minus 1pt}{\thmname{#1}\thmnumber{ #2}\thmnote{ \normalfont(#3)}}
\theoremstyle{step}
\newtheorem{step}{Step}[equation]

\DeclareMathOperator{\Ann}{Ann}
\DeclareMathOperator{\Ass}{Ass}
\DeclareMathOperator{\Ext}{Ext}
\DeclareMathOperator{\Frac}{Frac}
\DeclareMathOperator{\Hom}{Hom}
\DeclareMathOperator{\Min}{Min}
\DeclareMathOperator{\Proj}{Proj}
\DeclareMathOperator{\Spec}{Spec}
\DeclareMathOperator{\Sym}{Sym}
\DeclareMathOperator{\bigheight}{bight}
\DeclareMathOperator{\height}{ht}
\DeclareMathOperator{\im}{im}
\DeclareMathOperator{\prdiv}{div}
\DeclareMathOperator*{\ulim}{ulim}

\newcommand{\fa}{\mathfrak{a}}
\newcommand{\fm}{\mathfrak{m}}
\newcommand{\fn}{\mathfrak{n}}
\newcommand{\fp}{\mathfrak{p}}
\newcommand{\fq}{\mathfrak{q}}
\newcommand{\BB}{\mathbf{B}}
\newcommand{\CC}{\mathbf{C}}
\newcommand{\QQ}{\mathbf{Q}}
\newcommand{\RR}{\mathbf{R}}
\newcommand{\ZZ}{\mathbf{Z}}
\newcommand{\cO}{\mathcal{O}}
\newcommand{\cJ}{\mathcal{J}}
\newcommand{\cW}{\mathcal{W}}
\newcommand{\fD}{\mathfrak{D}}
\newcommand{\sL}{\mathscr{L}}
\newcommand{\cl}{\mathsf{cl}}
\newcommand{\epf}{\mathsf{epf}}
\newcommand{\ronef}{\mathsf{r1f}}
\newcommand{\wepf}{\mathsf{wepf}}

\providecommand\given{}
\newcommand\SetSymbol[1][]{\nonscript\:#1\vert\allowbreak\nonscript\:\mathopen{}}
\DeclarePairedDelimiterX\Set[1]\{\}{\renewcommand\given{\SetSymbol[\delimsize]}#1}

\begin{document}
\title{Uniform bounds on symbolic powers in regular rings}
\author{Takumi Murayama}
\address{Department of Mathematics\\Purdue University\\West Lafayette, IN
47907-2067\\USA}
\email{\href{mailto:murayama@purdue.edu}{murayama@purdue.edu}}
\urladdr{\url{https://www.math.purdue.edu/~murayama/}}

\thanks{This material is based upon work supported by the National Science
Foundation under Grant No.\ DMS-1902616}


\makeatletter
  \hypersetup{
    pdfsubject=\@subjclass,pdfkeywords=\@keywords
  }
\makeatother

\begin{abstract}
  We prove a uniform bound on the growth of symbolic powers of
  arbitrary (not necessarily radical)
  ideals in arbitrary (not necessarily excellent)
  regular rings of all characteristics.
  This gives a complete answer to a
  question of Hochster and Huneke.
  In equal characteristic, this result was proved by
  Ein, Lazarsfeld, and Smith and by Hochster and Huneke.
  For radical ideals in excellent regular rings of mixed characteristic, this
  result was proved by Ma and Schwede.
  We also prove a generalization of these results involving products of
  various symbolic powers and a uniform bound for regular local rings
  related to a conjecture of Eisenbud and Mazur, which are completely new in
  mixed characteristic.
  In equal characteristic, these results are due to
  Johnson and to Hochster--Huneke, Takagi--Yoshida, and Johnson, respectively.
\end{abstract}

\maketitle

\setcounter{tocdepth}{2}
{\hypersetup{hidelinks}\tableofcontents}

\section{Introduction}\label{sect:intro}
Let $R$ be a Noetherian ring, and let $I \subseteq R$ be an ideal.
For all integers $n \ge 1$, the \textsl{$n$-th symbolic power of $I$} is the
ideal
\begin{equation}\label{eq:hhsymbdef}
  I^{(n)} \coloneqq
  \bigcap_{\fp \in \Ass_R(R/I)} I^nR_\fp \cap R.
\end{equation}
\par Hartshorne asked when
the $\fp$-adic and $\fp$-symbolic topologies are equivalent, where $\fp
\subseteq R$ is a prime ideal and the $\fp$-symbolic topology is the topology
defined by the symbolic powers $\fp^{(n)}$ \cite[p.\ 160]{Har70}.
One motivation for Hartshorne's question was a conjecture of Grothendieck
asking whether the modules $\Hom_R(R/J,H^i_J(M))$ are finitely generated for
finitely generated $M$ \cite[Expos\'e XIII, Conjecture 1.1]{SGA2}.
Hartshorne answered this question when $\dim(R/\fp) = 1$ \cite[Proposition
7.1]{Har70}, and Schenzel gave a complete characterization for when the $I$-adic
and $I$-symbolic topologies are equivalent on a Noetherian ring $R$ for
arbitrary ideals $I \subseteq
R$ \citeleft\citen{Sch85}\citemid Theorem 1\citepunct
\citen{Sch86}\citemid Theorem 3.2\citeright\ (see also \cite[\S2.1]{Sch98}).
Schenzel's result implies for example that
if $R$ is a domain whose
local rings are analytically irreducible, then the $\fp$-adic and $\fp$-symbolic
topologies on $R$ are equivalent for every prime ideal $\fp \subseteq R$
(see \cite[p.\ 545]{HKV15}).
This latter result is originally due to Zariski \cite[Lemma 3 on p.\ 33]{Zar51}
(see also \cite[Proposition 3.11]{Ver88}).
\par In \cite[Main Theorem 3.3]{Swa00}, Swanson proved that in fact, when the
$I$-adic and $I$-symbolic topologies on $R$ are equivalent,
there exists an integer $k$ (depending on $I$) such that
\begin{equation}\label{eq:elsincl}
  I^{(kn)} \subseteq I^n
\end{equation}
for all $n \ge 1$.
When $R$ is regular of equal characteristic, Ein, Lazarsfeld, and Smith
\cite[Theorem
2.2 and Variant on p.\ 251]{ELS01} and Hochster and Huneke \cite[Theorems 2.6
and 4.4$(a)$]{HH02} proved that \eqref{eq:elsincl} holds for all ideals $I
\subseteq R$ and all $n \ge 1$, where $k$ can be taken to be
the largest
analytic spread of the ideals $IR_\fp \subseteq R_\fp$ as $\fp$ ranges over the
associated primes of $R/I$.
In particular, when $R$ is a finite-dimensional regular ring of equal
characteristic, one can take $k = \dim(R)$ by \cite[Proposition 5.1.6]{SH06},
and hence there is a uniform $k$ that does not depend on
$I$.\medskip
\par In \cite[p.\ 368]{HH02}, Hochster and Huneke asked whether the results of
Ein--Lazarsfeld--Smith and Hochster--Huneke hold in mixed characteristic.
Ma and Schwede answered their question when $R$ is excellent (or, more
generally, has geometrically reduced formal fibers) and $I$ is radical
\cite[Theorem 7.4]{MS18}.
However, the cases for (I) arbitrary ideals (i.e.\ not necessarily radical), and
(II) non-excellent rings remained open.
\subsection{Main theorem}
\par Our main result gives a complete answer to Hochser and Huneke's question by proving
the following uniform bound on the growth of
symbolic powers of arbitrary ideals in arbitrary regular rings of all
characteristics.
This resolves the remaining
cases of Hochster and Huneke's question, namely when $I$ is an arbitrary ideal
(i.e.\ not necessarily radical) and $R$ is not necessarily excellent.
\begin{alphthm}\label{thm:mainelshhms}
  Let $R$ be a regular ring, and let $I \subseteq R$ be an ideal.
  Let $h$ be the largest analytic spread of the ideals $IR_\fp \subseteq R_\fp$,
  where $\fp$ ranges over all associated primes $\fp$ of $R/I$.
  Then, for all $n \ge 1$, we have
  \[
    I^{(hn)} \subseteq I^n.
  \]
\end{alphthm}
See \cite[Definition 5.1.5]{SH06} for the definition of analytic spread.
Note that
\[
  h \le \bigheight(I) \coloneqq \max_{\fp \in
  \Ass_R(R/I)}\bigl\{\height(\fp)\bigr\} \le \dim(R)
\]
by \cite[Proposition 5.1.6]{SH06}.
Thus, Theorem
\ref{thm:mainelshhms} implies all finite-dimensional
regular rings satisfy the uniform symbolic topology property of Huneke, Katz,
and Validashti \citeleft\citen{HKV09}\citemid Question on p.\ 337\citepunct
\citen{HKV15}\citemid p.\ 543\citeright.\medskip
\par We provide two proofs of Theorem \ref{thm:mainelshhms} in this paper.
In the mixed characteristic setting, where the full statement of Theorem
\ref{thm:mainelshhms} was not known, both proof strategies use Scholze's theory of
perfectoid spaces \cite{Sch12}.
Andr\'e and Bhatt applied perfectoid techniques to commutative
algebra in mixed characteristic
to resolve Hochster's direct summand conjecture \cite{And18a,And18b}
(conjectured in \cite[p.\ 25]{Hoc73}) and de Jong's derived variant of the direct
summand conjecture \cite{Bha18}, respectively.
\par Our first proof of Theorem \ref{thm:mainelshhms}
uses axiomatic closure operations in the sense of
\cite{Die10,RG18} analogous to the theory of tight closure \cite{HH90} in
positive characteristic.
Tight closure was used by Hochster and Huneke
\cite{HH02} to prove Theorem \ref{thm:mainelshhms} in
equal characteristic.
Our closure-theoretic proof of Theorem \ref{thm:mainelshhms}
provides a short, axiomatic proof of Theorem
\ref{thm:mainelshhms} that is almost characteristic free and
unifies all known cases of Theorem
\ref{thm:mainelshhms}.
See \S\ref{sect:closureproofdesc} for more discussion.
\par Our second proof of Theorem \ref{thm:mainelshhms}
uses perfectoid/big Cohen--Macaulay test ideals.
Perfectoid test ideals were introduced by Ma and Schwede \cite{MS18} to prove
Theorem \ref{thm:mainelshhms} for radical ideals in excellent regular rings in
mixed characteristic.
The strategy in \cite{MS18} adapts the strategy of Ein, Lazarsfeld, and Smith
\cite{ELS01} using multiplier ideals on smooth complex varieties.
A crucial difference between the approach in \cite{MS18} and in this paper is
that we use a different definition for perfectoid/big Cohen--Macaulay test
ideals due to Robinson \cite{Rob}.
Robinson's definition has the advantage of working in arbitrary characteristic
and not using almost mathematics.
Moreover, we show that Robinson's test ideals can be compared to other notions
of test ideals \cite{MS,PRG,ST,HLS}, in particular to the $+$-test ideals
introduced by Hacon, Lamarche, and Schwede \cite{HLS}.
This flexibility is a key component of our test ideal-theoretic proof of
Theorem \ref{thm:mainelshhms}, since in mixed characteristic,
only the $+$-test ideals from \cite{HLS}
are known to be compatible with open embeddings (see
$(\ref{hlsideal:localization})$ below)
and to satisfy a Skoda-type theorem (see $(\ref{hlsideal:skoda})$ below).
See \S\ref{sect:intropartii} for more discussion.
\subsection{A generalization and a containment for regular local rings}
\par We in fact show the following more general version of Theorem
\ref{thm:mainelshhms}, again for
arbitrary ideals in arbitrary regular rings of all characteristics.
When the $s_i$ are not all zero, this result is completely new in mixed
characteristic, even for radical ideals in excellent regular rings.
In equal characteristic, the special case of the result below when all the $s_i$
are equal is due to Hochster and Huneke \cite[Theorems 2.6 and 4.4$(a)$]{HH02}.
The full statement in equal characteristic is due to 
Johnson \cite[Theorem 4.4]{Joh14}.
\begin{alphthm}\label{thm:mainjohnson}
  Let $R$ be a regular ring, and let $I \subseteq R$ be an ideal.
  Let $h$ be the largest analytic spread of $IR_\fp$, where $\fp$ ranges over
  all associated primes $\fp$ of $R/I$.
  Then, for all $n \ge 1$ and for all non-negative integers
  $s_1,s_2,\ldots,s_n$ with $s =
  \sum_{i=1}^n s_i$, we have
  \begin{equation}\label{eq:mainjohnsonincl}
    I^{(s+nh)} \subseteq \prod_{i=1}^n I^{(s_i+1)}.
  \end{equation}
\end{alphthm}
As a consequence, many known results on symbolic powers in regular rings of
equal characteristic can now be extended to regular rings of arbitrary
characteristic.
These include results on Grifo's stable version of Harbourne's conjecture
\cite[Conjecture 2.1]{Gri20}
and resurgence.
For example, the results in \cite[Theorem 2.5]{Gri20},
\cite[Theorems 3.2 and 3.3]{GHM20}, and
\cite[Theorem 4.5, Corollary 4.8, Theorem 4.12, and
Corollary 4.13]{DPD21}
(see \cite[Remark 4.23]{DPD21}) now hold for arbitrary regular rings.
Theorem \ref{thm:mainelshhms} follows
from Theorem \ref{thm:mainjohnson} by setting
$s_i = 0$ for all $i$.\medskip
\par We also show the following uniform bound on the growth of symbolic powers
of arbitrary ideals in arbitrary
regular local rings of all characteristics.
This result is related to a conjecture of Eisenbud and Mazur \cite[p.\
190]{EM97}, which asks whether for a regular local ring $(R,\fm)$ of equal
characteristic zero, one has $I^{(2)} \subseteq \fm I$ for every unmixed ideal
$I$.
This result is completely new in mixed characteristic, even for radical
ideals in excellent regular rings.
In equal characteristic, special cases of the result below are due to
Hochster and Huneke \cite[Theorems 3.5 and 4.2(1)]{HH07} and to Takagi and
Yoshida \cite[Theorems 3.1 and 4.1]{TY08}.
The full statement in equal characteristic is due to Johnson \cite[Theorem
4.3(2)]{Joh14}.
\begin{alphthm}\label{thm:mainjohnsonhhty}
  Let $(R,\fm)$ be a regular local ring, and let $I \subseteq R$ be an ideal.
  Let $h$ be the largest analytic spread of $IR_\fp$, where $\fp$ ranges over
  all associated primes $\fp$ of $R/I$.
  Then, for all $n \ge 1$ and for all non-negative
  integers $s_1,s_2,\ldots,s_n$ with $s =
  \sum_{i=1}^n s_i$, we have
  \begin{equation}\label{eq:mainjohnsonhhtyincl}
    I^{(s+nh+1)} \subseteq \fm \cdot \prod_{i=1}^n I^{(s_i+1)}.
  \end{equation}
\end{alphthm}
While the conjecture of Eisenbud and Mazur has counterexamples in positive
characteristic \cite[Example on p.\ 200]{EM97} and mixed characteristic
\cite[Remark 3.3$(b)$]{KR00}, Theorem \ref{thm:mainjohnsonhhty}
holds for all regular local rings.
In particular, Theorem \ref{thm:mainjohnsonhhty} implies that for ideals $I$
in regular local rings $(R,\fm)$, we have $I^{(h+1)} \subseteq \fm
I$.
See \citeleft\citen{HH07}\citemid p.\ 172\citepunct \citen{TY08}\citemid p.\
719\citeright\ for more discussion.

\subsection{Discussion of proof strategies}
\par We now describe our strategy for proving Theorems \ref{thm:mainelshhms},
\ref{thm:mainjohnson}, and \ref{thm:mainjohnsonhhty}.
We give two logically independent strategies for proving these results,
corresponding to the two parts of this paper.
Both strategies work in all characteristics and are of independent interest.
\begin{enumerate}
  \item[$(\textup{\ref{part:closure-elshhms}})$]
    A closure-theoretic proof using axiomatic closure operations in the sense of
    \cite{Die10,RG18} to replace the theory of tight closure \cite{HH90,HHchar0}
    used in \cite{HH02}.
    In mixed characteristic, this proof uses Heitmann's full extended plus
    ($\epf$) closure \cite{Hei01}, Jiang's weak $\epf$ ($\wepf$) closure
    \cite{Jia21}, and R.G.'s results on closure operations that induce big
    Cohen--Macaulay algebras \cite{RG18}.
  \item[$(\textup{\ref{part:elshhms}})$]
    A proof using various versions of perfectoid/big Cohen--Macaulay test
    ideals from \cite{MS18,MS,PRG,MSTWW,Rob,ST,HLS}
    to replace the theory of multiplier ideals used in \cite{ELS01}.
    Note that Ma and Schwede introduced their theory of perfectoid test
    ideals to prove their special case of Theorem \ref{thm:mainelshhms} in
    \cite{MS18}.
\end{enumerate}
As mentioned above, in mixed characteristic, both proof strategies use Scholze's theory of
perfectoid spaces \cite{Sch12}.\medskip
\par Removing the condition that $I$ is radical is the key obstacle to
proving Theorem \ref{thm:mainelshhms} in the most general setting.
In mixed characteristic, the closure operations and versions of test ideals
available to us require working over a complete local ring.
We therefore need to reduce to the case where we can work over a complete local
ring.
Even if $I$ is radical, it may no longer be radical after this reduction.
Thus, even if one's main interest is the special case of Theorem
\ref{thm:mainelshhms} when $I$ is radical or even prime, one must
consider non-radical ideals.\medskip
\par While our proofs of Theorems \ref{thm:mainelshhms},
\ref{thm:mainjohnson}, and \ref{thm:mainjohnsonhhty} using test ideals predate
our closure-theoretic proofs, we
have presented the closure-theoretic proofs first because the proofs are shorter and
require fewer technical prerequisites.
After reading the preliminaries in
\S\ref{sect:prelims}, the reader can proceed directly to reading
Part \ref{part:closure-elshhms} or Part \ref{part:elshhms} depending on their
interests.
In the rest of this introduction, we summarize some key aspects of each of our
proofs and how they relate to previous work in the literature.
\subsubsection{Proof via closure theory}\label{sect:closureproofdesc}
\par The proof of Theorem \ref{thm:mainelshhms} in equal characteristic
due to Hochster and Huneke \cite{HH02} uses the theory of tight closure in
equal characteristic $p > 0$ \cite{HH90} or equal characteristic zero
\cite{HHchar0}.
This proof is remarkably short, especially in equal characteristic $p > 0$
when $I$ is radical \cite[pp.\ 350--351]{HH02}.\medskip
\par In Part \ref{part:closure-elshhms} of this paper,
we prove Theorems \ref{thm:mainelshhms},
\ref{thm:mainjohnson}, and \ref{thm:mainjohnsonhhty} in mixed characteristic
using an appropriate replacement for tight closure.
In fact, our proofs apply to any context where a sufficiently well-behaved
closure operation exists, and hence gives new proofs of Theorems
\ref{thm:mainelshhms}, \ref{thm:mainjohnson}, and \ref{thm:mainjohnsonhhty}
in all characteristics.
We found these proofs after reconstructing an unpublished proof of Theorem
\ref{thm:mainelshhms} in equal characteristic $p >0$ due to Hochster and Huneke
\cite{HH02}, which uses plus closure \cite{HH92,Smi94}.
Hochster and Huneke sketch their strategy
in \cite[p.\ 353]{HH02}.
According to Hochser and Huneke, their strategy used the following two ingredients:
\begin{enumerate}
  \item Plus closure localizes \citeleft\citen{HH92}\citemid Lemma
    6.5$(b)$\citepunct \citen{Smi94}\citemid p.\ 45\citeright.
  \item A Brian\c{c}on--Skoda-type theorem holds for plus closure
    \cite[Theorem 7.1]{HH95}.
\end{enumerate}
\par Our closure-theoretic proofs of Theorems \ref{thm:mainelshhms},
\ref{thm:mainjohnson}, and \ref{thm:mainjohnsonhhty} in mixed characteristic
are possible
because sufficiently powerful replacements for tight closure and plus closure in
mixed characteristic are now available.
Heitmann introduced full extended plus ($\epf$)
closure \cite{Hei01} as one such possible replacement for tight closure and plus
closure in mixed characteristic.
Following \cite{Hei01,HM21},
for a domain $R$ such that the image of $p$ lies in the Jacobson
radical of $R$, the \textsl{full extended plus} ($\epf$) \textsl{closure} of an
ideal $I \subseteq R$ is
\[
  I^\epf \coloneqq \Set*{x \in R \given \begin{tabular}{@{}c@{}}
    there exists $c \in R - \{0\}$ such that\\
    $c^\varepsilon x \in (I,p^N)R^+$\\
    for every $\varepsilon \in \QQ_{>0}$ and every $N \in \ZZ_{>0}$
  \end{tabular}}.
\]
Here, $R^+$ denotes the \textsl{absolute integral closure} of $R$, i.e., the
integral closure of a domain $R$ in an algebraic closure of
its fraction field (see \S\ref{sect:rplus}).
We note that $\epf$ closure coincides with tight closure for complete Noetherian
local rings of equal characteristic
$p > 0$ by \cite[Theorem 3.1]{HH91} (see \cite[Proposition 2.11]{HM21}).
\par Heitmann proved a Brian\c{c}on--Skoda-type theorem for
$\epf$ closure \cite[Theorem 4.2]{Hei01}, and
used $\epf$ closure to prove Hochster's direct summand
conjecture \cite[p.\ 25]{Hoc73} in mixed characteristic for rings of dimension three
\cite{Hei02}.
More recently, using perfectoid techniques,
Heitmann and Ma showed that Heitmann's Brian\c{c}on--Skoda-type theorem
for $\epf$ closure implies the Brian\c{c}on--Skoda theorem in mixed
characteristic \cite[Theorem 3.20]{HM21} and gave a closure-theoretic proof of
Hochster's direct summand conjecture in mixed characteristic
\cite[p.\ 135]{HM21}.
The Brian\c{c}on--Skoda theorem was originally proved by Brian\c{c}on and Skoda
\cite{SB74} for $\mathbf{C}\{z_1,z_2,\ldots,z_n\}$ and by Lipman and Sathaye
\cite{LS81} in general.
The direct summand conjecture was originally proved by Hochster \cite{Hoc73}
in equal characteristic and by Heitmann \cite{Hei02}
(in dimension three) and Andr\'e \cite{And18a,And18b} (in general) in mixed
characteristic.\medskip
\par While $\epf$ closure is powerful enough to prove the direct summand
conjecture and the Brian\c{c}on--Skoda theorem in mixed characteristic, our
proofs of Theorems \ref{thm:mainelshhms}, \ref{thm:mainjohnson}, and
\ref{thm:mainjohnsonhhty} require working with other closure operations.
One reason for this is that $\epf$ closure does not localize
\cite[p.\ 817]{Hei01}, and hence Hochster and Huneke's unpublished proof of
Theorem \ref{thm:mainelshhms} using plus closure in equal characteristic $p >0$
\cite[p.\ 353]{HH02} cannot readily be adapted to the mixed characteristic case.
\par Instead, our strategy in Part \ref{part:closure-elshhms} utilizes
$\epf$ closure in conjunction with Jiang's weak $\epf$
($\wepf$) closure \cite{Jia21} and R.G.'s results on closure operations that
induce big Cohen--Macaulay algebras \cite{RG18}.
We use $\wepf$ closure and R.G.'s results to work around the fact that $\epf$
closure does not localize.
Following \cite{Jia21},
for a domain $R$ such that the image of $p$ lies in the Jacobson
radical of $R$, the \textsl{weak
$\epf$} ($\wepf$) \textsl{closure} of an ideal $I \subseteq R$ is
\[
  I^\wepf \coloneqq \bigcap_{N=1}^\infty (I,p^N)^\epf.
\]
\par The advantage of using $\wepf$ closure is that it satisfies key
additional properties not known for $\epf$ closure,
namely Dietz's axioms \cite[Axioms
1.1]{Die10} and R.G.'s algebra axiom \cite[Axiom 3.1]{RG18}.
Dietz formulated the axioms in \cite{Die10} to characterize when a closure
operation can be used to construct big Cohen--Macaulay modules.
R.G. formulated the algebra axiom in \cite{RG18} to characterize when a closure
operation can be used to construct big Cohen--Macaulay algebras.
Using perfectoid techniques,
Jiang \cite[Theorem
4.8]{Jia21} showed that $\wepf$ closure
satisfies Dietz's axioms and R.G.'s algebra axiom when
$R$ is a complete local domain of mixed characteristic with $F$-finite residue
field.
As a consequence,
we have the inclusions
\begin{equation}\label{eq:epfwepfbcm}
  I^\epf \subseteq I^\wepf \subseteq IB \cap R
\end{equation}
for some big Cohen--Macaulay algebra $B$ over $R$ by a result of R.G.
\cite[Proposition 4.1]{RG18}.
\medskip
\par To illustrate our proof strategy for Theorems \ref{thm:mainelshhms},
\ref{thm:mainjohnson}, and \ref{thm:mainjohnsonhhty}, we sketch
the proof of Theorem \ref{thm:mainelshhms} when $R$ is a complete regular 
local ring of
mixed characteristic with an $F$-finite residue field and $I$ is generated by
$h$ elements.
If $u \in I^{(hn)}$, there exists an element $c$ avoiding the associated primes
of $R/I$ such that $cu \in I^{hn}$.
We then consider a module-finite extension $R'$ of $R$ that is complete local
and contains an $n$-th root $u^{1/n}$ of $u$.
We have
\begin{align*}
  (cu^{1/n})^n &= c^nu \in (I^hR')^n,
\intertext{and hence $cu^{1/n} \in \overline{I^hR'}$.
By Heitmann's Brian\c{c}on--Skoda-type theorem for $\epf$ closure
\cite[Theorem 4.2]{Hei01} and using \eqref{eq:epfwepfbcm}, we know that}
  cu^{1/n} \in (IR')^\epf &\subseteq (IR')^\wepf \subseteq IB \cap R',
\end{align*}
where $B$ is a big Cohen--Macaulay algebra over $R'$.
Then, $B$ is also a big Cohen--Macaulay algebra over $R$ and
$R \to B$ is faithfully flat.
We therefore see that $c$
is a nonzerodivisor on $B/IB$, and hence $u^{1/n} \in IB \cap R'$.
Taking $n$-th powers, we have $u \in I^nB \cap R = I^n$.\medskip
\par We highlight three key aspects of our proof of Theorems \ref{thm:mainelshhms},
\ref{thm:mainjohnson}, and \ref{thm:mainjohnsonhhty} in Part \ref{part:closure-elshhms}
that already appear in the proof sketch above.
\begin{enumerate}[ref=\arabic*]
  \item\label{keyaspect:i1}
    Instead of passing to a large, possibly non-Noetherian extension of $R$
    (such as $R^+$ or a big Cohen--Macaulay algebra), we pass to the
    \emph{smallest} integral extension $R'$ of $R$ where $u^{1/n}$ makes sense.
    Since $R'$ is module-finite over $R$, it is Noetherian, and hence we can use
    closure operations on $R'$ as usual.
  \item\label{keyaspect:i2}
    After applying the Brian\c{c}on--Skoda theorem for $\epf$ closure on $R'$,
    we can apply R.G.'s result on $R'$ as in \eqref{eq:epfwepfbcm} to pass to a
    big Cohen--Macaulay algebra $B$ that captures $\wepf$ closure.
    We pass to $IB \cap R'$ instead of working with
    $(IR')^\epf$ or $(IR')^\wepf$ because
there are no elements $c^\varepsilon$ or $p$-powers involved in the definition.
  \item\label{keyaspect:i3}
    Since $B$ is flat over $R$, we can detect when nonzerodivisors in $R/I$ map
    to nonzerodivisors in $B/IB$.
    Thus, we can
    get rid of the element $c$ multiplying $u^{1/n}$ into $IB$.
\end{enumerate}
Note that $(\ref{keyaspect:i2})$ and $(\ref{keyaspect:i3})$ require the
extension ring $R'$ in $(\ref{keyaspect:i1})$ to be Noetherian.
Working with the big Cohen--Macaulay algebra $B$ in
$(\ref{keyaspect:i2})$ and $(\ref{keyaspect:i3})$ allows us
to avoid localizations.\medskip
\par Our proof strategy using closure operations also applies to any
closure operation satisfying Dietz's axioms, R.G.'s algebra axiom, and a
Brian\c{c}on--Skoda-type theorem.
We therefore obtain new proofs of Theorems \ref{thm:mainelshhms},
\ref{thm:mainjohnson}, and \ref{thm:mainjohnsonhhty} in all characteristics, since
tight closure \cite{HH90} and plus closure \cite{HH92,Smi94} satisfy these
properties in equal characteristic $p > 0$, and
$\mathfrak{B}$-closure \cite{AS07} satisfies these properties in equal
characteristic zero.
In mixed characteristic, one can alternatively use Heitmann's full rank $1$
($\ronef$) closure \cite{Hei01} instead of $\wepf$ closure
(still using results from \cite{RG18,Jia21}).
See Table \ref{tab:dietzclosures} for references to proofs of these axioms for
these closure operations.
We can also adapt our proofs to use small or big equational tight clossure in
equal characteristic zero \cite{HHchar0}.
See Remark \ref{rem:weakalgaxiom}.

\subsubsection{Proof via multiplier/test ideals}\label{sect:intropartii}
The proof of Theorem \ref{thm:mainelshhms} for smooth complex varieties due to
Ein, Lazarsfeld, and Smith \cite{ELS01} uses the theory of \textsl{multiplier
ideals} $\cJ(X,\fa^t)$ (see \cite[Part Three]{Laz04b}).
Here, $X$ denotes a smooth complex variety, $\fa \subseteq \cO_X$ is a coherent
ideal sheaf, and $t$ is a non-negative real number.
One version of their proof (see \citeleft\citen{ST12}\citemid Theorem
6.23\citepunct \citen{CRS20}\citemid p.\ 27\citeright) proceeds by showing
the following sequence of inclusions and equalities:
\begin{equation}\label{eq:crssketch}
  I^{(hn)}
  \underset{(\text{\scriptsize$\ref{multiplierideal:nottoosmall}$})}{\subseteq}
  \cJ\Bigl(X,\bigl(I^{(hn)}\bigr)^1\Bigr)
  \underset{(\text{\scriptsize$\ref{multiplierideal:unambiguity}$})}{=}
  \cJ\Bigl(X,\bigl(I^{(hn)}\bigr)^{n/n}\Bigr)
  \underset{(\text{\scriptsize$\ref{multiplierideal:subadditivity}$})}{\subseteq}
  \cJ\Bigl(X,\bigl(I^{(hn)}\bigr)^{1/n}\Bigr)^n
  \underset{(\text{\scriptsize$\ref{multiplierideal:nottoobig}$})}{\subseteq}
  I^n.
\end{equation}
These inclusions and equalities
follow from the following formal properties of multiplier
ideals (using terminology from \cite[p.\ 915]{MS18}):
\begin{enumerate}[ref=\arabic*]
  \item\label{multiplierideal:nottoosmall}
    (Not too small) $\fa \subseteq \cJ(X,\fa^1)$.
  \item\label{multiplierideal:unambiguity}
    (Unambiguity of exponent) For every integer $n > 0$, we have
    $\cJ(X,\fa^{tn}) = \cJ(X,(\fa^n)^t)$.
  \item\label{multiplierideal:subadditivity}
    (Subadditivity \cite[Variant 2.5]{DEL00}) $\cJ(X,\fa^{tn}) \subseteq
    \cJ(X,\fa^t)^n$.
  \item\label{multiplierideal:nottoobig}
    (Not too big \citeleft\citen{ELS01}\citemid Proof of Variant on p.\
    251\citepunct \citen{TY08}\citemid Proof of Theorem 4.1\citeright)
    If $h$ is the largest analytic spread of $\fa R_\fp$ where
    $\fp$ ranges over all associated primes of $R/\fa$, then
    \[
      \cJ\Bigl(X,\bigl(\fa^{(hn)}\bigr)^{1/n}\Bigr) \subseteq \fa.
    \]
\end{enumerate}
Hara \cite[Theorem 2.12]{Har05} adapted this approach to give an alternative
proof of Theorem \ref{thm:mainelshhms} in equal characteristic $p > 0$ using
the theory of \textsl{test ideals} $\tau(R,\fa^t)$ introduced in
\citeleft\citen{HH90}\citemid \S8\citepunct \citen{HY03}\citeright.\medskip
\par In mixed characteristic, Ma and Schwede \cite{MS18} used perfectoid techniques
to define and develop the theory for a mixed characteristic analogue of
multiplier/test ideals, the
\textsl{perfectoid test ideals}
\[
  \tau\bigl(R,[\underline{f}]^t\bigr) \qquad \text{and} \qquad
  \tau\bigl(R,\fa^t\bigr).
\]
Here, $R$ is a complete regular local ring of mixed characteristic,
$[\underline{f}]$ represents a choice of elements $f_1,f_2,\ldots,f_r \in R$,
and $\fa$ is the ideal $(f_1,f_2,\ldots,f_r)$.
These two ideals are related in the following manner:
\[
  \tau\bigl(R,[\underline{f}]^t\bigr) \subseteq \tau\bigl(R,\fa^t\bigr).
\]
\par To prove the special case of Theorem \ref{thm:mainelshhms} when $R$ is a
regular ring of
mixed characteristic with geometrically reduced formal fibers
and $I$ is radical,
Ma and
Schwede first reduce to the complete local case (the assumption on the formal
fibers of $R$ is needed to preserve the radicalness of $I$).
Ma and Schwede then prove and apply analogues of properties
$(\ref{multiplierideal:nottoosmall})$--$(\ref{multiplierideal:subadditivity})$
for the ideal $\tau(R,[\underline{f}]^t)$ \cite[Proposition 3.8, Proposition
3.9, and Theorem 4.4]{MS18}.
The assumption that $I$ is radical is used to prove the special case of
$(\ref{multiplierideal:nottoobig})$ for the ideal $\tau(R,\fa^t)$ when $\fa$ is
radical
\cite[Theorem 5.11]{MS18}.\medskip
\par As mentioned above, removing the condition that $I$ is radical is the key
obstacle to proving Theorem \ref{thm:mainelshhms} in the most general setting.
To apply \cite[Theorem 5.11]{MS18}, it is necessary in \cite{MS18}
to reduce to the case when
$I$ is a radical ideal in a complete regular local ring $R$.
Even if one starts with a radical ideal $I$, after reducing to the complete local
case, it may no longer be radical.
In addition, our strategy for Theorems \ref{thm:mainjohnson} and
\ref{thm:mainjohnsonhhty} in Part \ref{part:elshhms} requires
stronger forms of Ma and Schwede's unambiguity
statement for exponents $(\ref{multiplierideal:unambiguity})$
\cite[Proposition 3.8]{MS18} and subadditivity
$(\ref{multiplierideal:subadditivity})$ \cite[Theorem 4.4]{MS18}.\medskip
\par In Part \ref{part:elshhms} of this paper, we instead use the test
ideals
\[
  \tau_B\bigl(R,\Delta,[\underline{f}]^t\bigr) \qquad \text{and} \qquad
  \tau_B\bigl(R,\Delta,\fa^t\bigr)
\]
which were defined by Robinson \cite{Rob}
for normal complete Noetherian local rings
of arbitrary characteristic such that $K_R+\Delta$ is $\QQ$-Cartier
(see Definition \ref{def:ms18like}).
As with Ma and Schwede's definition, $[\underline{f}]$ represents a choice of
generators for the ideal $\fa$, and the two ideals are related in the following
manner:
\[
  \tau_B\bigl(R,\Delta,[\underline{f}]^t\bigr) \subseteq
  \tau_B\bigl(R,\Delta,\fa^t\bigr).
\]
\par Robinson's definition
combines features of the perfectoid test ideals of \cite{MS18}
and of the \textsl{big Cohen--Macaulay test ideals} $\tau_B(R,\Delta)$.
Ma and Schwede \cite{MS} and P\'erez and R.G. \cite{PRG} introduced big
Cohen--Macaulay test
ideals $\tau_B(R,\Delta)$ of pairs $(R,\Delta)$.
Robinson \cite{Rob} extended their definition
to triples of the form $(R,\Delta,[\underline{f}]^t)$ and
$(R,\Delta,\fa^t)$ (see also \cite{ST}).
These definitions rely on the existence of $R^+$-algebras $B$ that are big
Cohen--Macaulay over $R$, where $R^+$ is the absolute integral closure
of $R$.
Big Cohen--Macaulay algebras were introduced by Hochster
\cite{Hoc75queens} and were shown to exist in \cite{HH92,And18b} (see also
\cite{Hoc94,HH95,Shi18,DRG19,And20} and Remark \ref{rem:bigcmalgexist}).
\par Compared to the perfectoid test ideals of \cite{MS18}, Robinson's
definition works
in arbitrary characteristic.
It also does not use almost mathematics, and there is no perturbation present in
the definition
in contrast to \cite[Definition 3.5]{MS18}.
These differences in the definition allow us in this paper to
prove versions of the unambiguity statement for
exponents $(\ref{multiplierideal:unambiguity})$
(Proposition \ref{prop:notambig}) and subadditivity
$(\ref{multiplierideal:subadditivity})$
(Theorem \ref{thm:subadditivity}) that are
stronger than those proved in \cite{MS18} for
perfectoid test ideals.
These results are essential for our proofs of Theorems
\ref{thm:mainjohnson} and \ref{thm:mainjohnsonhhty}.\medskip
\par We are able to prove Theorems \ref{thm:mainelshhms}, \ref{thm:mainjohnson},
and \ref{thm:mainjohnsonhhty} for arbitrary (not necessarily radical) ideals in
arbitrary (not necessarily excellent) regular rings
by proving the following version of
$(\ref{multiplierideal:nottoobig})$
that holds for arbitrary (not necessarily radical) ideals $I$.
\begin{alphthm}\label{thm:ty23like}
  Let $R$ be a normal complete Noetherian local ring such that $K_R$ is
  $\QQ$-Cartier.
  Consider an ideal $I \subseteq R$, and assume that 
  the localizations of $R$ at the associated primes of $R/I$ have
  infinite residue fields.
  \par Let $h$ be the largest analytic spread of $IR_\fp$, where $\fp$ ranges
  over all associated primes of $R/I$.
  For every integer $M > 0$
  and every finite set of non-negative integers $s_1,s_2,\ldots,s_n$, 
  there exists an $R^+$-algebra $B$ that is big Cohen--Macaulay over $R$
  such that
  \begin{equation}\label{eq:lastinclusionwithtaub}
    \tau_B\Bigl(R,\bigl(I^{(M)}\bigr)^{\frac{s_i+h}{M}}\Bigr)
    \subseteq I^{(s_i+1)}
  \end{equation}
  for every $i$.
  If $R$ is of residue characteristic $p > 0$, then setting $B = \widehat{R^+}$
  suffices.
\end{alphthm}
Here, $\widehat{R^+}$ denotes the $p$-adic completion of the absolute integral
closure of $R$, which is a
big Cohen--Macaulay algebra for complete Noetherian local
domains of residue characteristic $p > 0$.
This is due to Hochster and Huneke \cite[Main Theorem 5.15]{HH92} in equal
characteristic $p > 0$ and to Bhatt \cite[Corollary 5.17]{Bha} in mixed
characteristic.
\par Theorem \ref{thm:ty23like} can be viewed as a version of the equal
characteristic statements in \cite[Proof of Variant on p.\ 251]{ELS01},
\cite[Proof of Theorem 2.12]{Har05}, and \cite[Proposition
2.3 and Proof of Theorem 4.1]{TY08} that applies in all
characteristics.
Using Robinson's test ideals and Theorem
\ref{thm:ty23like}, we then prove Theorems \ref{thm:mainelshhms},
\ref{thm:mainjohnson}, and \ref{thm:mainjohnsonhhty} for arbitrary ideals in
arbitrary regular rings of all characteristics simultaneously.
The idea is to use our results
for the test ideals $\tau_B(R,\Delta,[\underline{f}]^t)$ and
$\tau_B(R,\Delta,\fa^t)$ mentioned
above to construct a sequence of inclusions as in \eqref{eq:crssketch},
instead of those for multiplier ideals or for existing versions of test
ideals used in \cite{ELS01,Har05,TY08,ST12,MS18} to prove previously known cases
of Theorem \ref{thm:mainelshhms}.\medskip
\par A key difficulty in proving Theorem \ref{thm:ty23like} is that no single
version of perfectoid/big
Cohen--Macaulay test ideals satisfies all the properties necessary to prove
Theorem \ref{thm:ty23like}.
To prove the analogous result $(\ref{multiplierideal:nottoobig})$ for rings
essentially of finite type over a field of characteristic zero,
Takagi and Yoshida \cite[Proof of Theorem 4.1]{TY08}
proceed by localizing at each associated
prime $\fp$ of $R/I$ and show the following sequence of inclusions and
equalities (see also \cite[Proof of Variant on p.\ 251]{ELS01} and Proposition
\ref{prop:ty23likechar0}):
\begin{equation}\label{eq:tysketch}
  \begin{alignedat}{3}
    \cJ\Bigl(X,\bigl(I^{(hn)}\bigr)^{1/n}\Bigr) \cdot R_\fp
    &\omit\hfil$\mathrel{\underset{(\text{\scriptsize$\ref{multiplierideal:localization}$})}{=}}$\hfil\ignorespaces&&
    \cJ\Bigl(\Spec(R_\fp),\bigl(I^{hn}R_\fp\bigr)^{1/n}\Bigr)\\
    &\omit\hfil$\mathrel{\underset{(\text{\scriptsize$\ref{multiplierideal:unambiguity}$})}{=}}$\hfil\ignorespaces&&
    \cJ\bigl(\Spec(R_\fp),(IR_\fp)^h\bigr)\\
    &\omit\hfil$\mathrel{\underset{(\text{\scriptsize$\ref{multiplierideal:skoda}$})}{=}}$\hfil\ignorespaces&&
    I \cdot \cJ\bigl(\Spec(R_\fp),(IR_\fp)^{h-1}\bigr)\\
    &\omit\hfil$\mathrel{\subseteq}$\hfil\ignorespaces&&
    IR_\fp.
  \end{alignedat}
\end{equation}
These inclusions and equalities follow from the following formal properties of
multiplier ideals which hold whenever $X$ is a normal scheme
essentially of finite type over the complex numbers such that $K_X$ is $\QQ$-Cartier:
\begin{enumerate}[label=$(\ref*{multiplierideal:nottoobig}\alph*)$,
    ref=\ref*{multiplierideal:nottoobig}\alph*]
  \item\label{multiplierideal:localization} (Localization) The formation of
    $\cJ(X,\fa^t)$ is compatible with arbitrary localization.
  \item\label{multiplierideal:skoda} (Skoda-type theorem \cite[Theorem
    9.6.36]{Laz04b}) Suppose that $X = \Spec(R)$ for a local ring $R$ and that
    $\fa$ has analytic spread $h$.
    For all $t \ge h$, we have
    $\cJ(X,\fa^t) = \fa \cdot \cJ(X,\fa^{t-1})$.
\end{enumerate}
Hara \cite[Proof of Theorem 2.12]{Har05} and Takagi--Yoshida \cite[Propostion
2.3]{TY08} used this approach for test ideals in equal characteristic
$p > 0$.\medskip
\par In mixed characteristic, no version of perfectoid/big
Cohen--Macaulay test ideals is currently known to be compatible with arbitrary
localizations.
However, the \textsl{$+$-test ideals} of Hacon, Lamarche, and Schwede \cite{HLS}
satisfy a weaker version of $(\ref{multiplierideal:localization})$ and also
satisfy $(\ref{multiplierideal:skoda})$.
The definition of $+$-test ideals relies on the result that the $p$-adic
completion $\widehat{R^+}$ of the absolute integral closure of $R$ is a
big Cohen--Macaulay algebra for complete Noetherian local
domains of residue characteristic $p > 0$ mentioned above
\citeleft\citen{HH92}\citemid Main Theorem 5.15\citepunct \citen{Bha}\citemid
Corollary 5.17\citeright.
Building
on these results
and subsequent developments
due to
Takamatsu and Yoshikawa \cite{TY} and to
Bhatt, Ma, Patakfalvi, Schwede, Tucker, Waldron, and Witaszek \cite{BMPSTWW},
Hacon, Lamarche, and Schwede \cite{HLS} introduced
the $+$-test ideals
\[
  \tau_+\bigl(\cO_X,\Delta\bigr) \qquad \text{and} \qquad
  \tau_+\bigl(\cO_X,\fa^t\bigr)
\]
for divisor pairs $(X,\Delta)$ and ideal pairs $(X,\fa^t)$, respectively,
where $X$ is a normal integral quasi-projective
scheme over a complete Noetherian local ring $(R,\fm)$ of residue
characteristic $p > 0$ such that $K_X+\Delta$ is $\QQ$-Cartier.
\par These $+$-test ideals
satisfy the following versions of $(\ref{multiplierideal:localization})$ and
$(\ref{multiplierideal:skoda})$, where $X$ and $R$ are as in the previous
paragraph:
\begin{enumerate}
  \item[{$(\ref{multiplierideal:localization}')$}]
  \makeatletter
  \protected@edef\@currentlabel{\ref*{multiplierideal:localization}'}
  \phantomsection
  \label{hlsideal:localization}
  \makeatother
    (Compatibility with open embeddings \cite[Corollary 5.8]{HLS})
    If $U \hookrightarrow X$ is an open
    embedding, then $\tau_+(\cO_X,\Delta)\rvert_U = \tau_+(\cO_U,\Delta\rvert_U)$.
  \item[{$(\ref{multiplierideal:skoda}')$}]
  \makeatletter
  \protected@edef\@currentlabel{\ref*{multiplierideal:skoda}'}
  \phantomsection
  \label{hlsideal:skoda}
  \makeatother
    (Skoda-type theorem \cite[Theorem 6.6]{HLS})
    Suppose that $X = \Spec(R)$ and that $\fa$ has analytic spread $h$.
    For all $t \ge h$, we have
    $\tau_+(\cO_X,\fa^t) = \fa \cdot \tau_+(\cO_X,\fa^{t-1})$.
\end{enumerate}\medskip
\par We highlight six key aspects of our proof of Theorems \ref{thm:mainelshhms},
\ref{thm:mainjohnson}, and \ref{thm:mainjohnsonhhty} in
Part \ref{part:elshhms}.
\begin{enumerate}
  \item Because Robinson's big Cohen--Macaulay test ideals
    $\tau_B(R,\Delta,[\underline{f}]^t)$ do not use almost mathematics and do
    not incorporate perturbations, we are able to prove stronger versions of the
    unambiguity statement for exponents $(\ref{multiplierideal:unambiguity})$
    (Proposition \ref{prop:notambig}) and subadditivity
    $(\ref{multiplierideal:subadditivity})$
    (Theorem \ref{thm:subadditivity}) compared to \cite{MS18}.
    This extra flexibility is essential in our proofs of Theorems
    \ref{thm:mainjohnson} and \ref{thm:mainjohnsonhhty}.
    Our subadditivity theorem (Theorem \ref{thm:subadditivity}) requires working
    with a fixed set of generators $f_1,f_2,\ldots,f_r$ as was also the case in
    \cite{MS18}.
\end{enumerate}
To prove Theorem \ref{thm:ty23like} in residue characteristic $p > 0$, we
proceed as follows.
\begin{enumerate}[resume]
  \item To make properties $(\ref{hlsideal:localization})$ and
    $(\ref{hlsideal:skoda})$ satisfied by Hacon--Lamarche--Schwede's $+$-test
    ideals available to us, we prove the following comparisons between the
    various versions of test ideals in \cite{MS,PRG,Rob,ST,HLS}.
For normal complete Noetherian local rings $(R,\fm)$ of residue
characteristic $p > 0$ whose canonical
divisors $K_R$ are $\QQ$-Cartier, and for all ideals $\fa \subseteq R$ generated
by elements $f_1,f_2,\ldots,f_r$, the special instances of Robinson's test
ideals
when $B = \widehat{R^+}$ and $\Delta = 0$
are related to the definitions in \cite{MS,PRG,ST,HLS}
in the following manner for rational numbers $t \ge 0$:
\begin{equation}\label{eq:keycomparisons}
  \begin{aligned}
    \tau_{\widehat{R^+}}
    \bigl(R,[\underline{f}]^t\bigr) \subseteq \tau_{\widehat{R^+}}
    \bigl(R,\fa^t\bigr)
    &= \sum_{m=1}^\infty \sum_{g \in \fa^m}
    \tau_{\widehat{R^+}}\biggl(R,\frac{t}{m} \prdiv_R(g)\biggr)\\
    &= \sum_{m=1}^\infty \sum_{g \in \fa^m}
    \tau_+\biggl(\cO_{\Spec(R)},\frac{t}{m} \prdiv_R(g)\biggr)\\
    &\subseteq \tau_+\bigl(\cO_{\Spec(R)},\fa^t\bigr).
  \end{aligned}
\end{equation}
See Proposition \ref{prop:ms18comparison},
Lemma \ref{lem:testidealasasum}, Remark \ref{rem:hls512}, and Lemma
\ref{lem:hls64}.
  \item Instead of localizing at prime ideals as in \eqref{eq:tysketch}, it is
    enough to find a suitable \emph{principal} localization where we can show
    that the sequence of inclusions and equalities in \eqref{eq:tysketch} hold.
    Since we need to apply $(\ref{hlsideal:localization})$, which holds for
    $+$-test ideals of
    divisor pairs $(X,\Delta)$, we use the penultimate ideal in
    \eqref{eq:keycomparisons} together with properties of reductions of ideals
    (see \cite[Chapter 8]{SH06}).
  \item We use the last inclusion in \eqref{eq:keycomparisons} to pass
    to a $+$-test ideal of an ideal pair $(X,\fa^t)$ and apply the Skoda-type
    theorem $(\ref{hlsideal:skoda})$.
    This step requires using the fact that $+$-test ideals are compatible with
    integral closures of ideals \cite[Remark 6.4]{HLS}.
\end{enumerate}
Finally, to prove Theorem \ref{thm:ty23like} in equal characteristic zero, we
proceed as follows.
\begin{enumerate}[resume]
  \item We prove a comparison between Robinson's big Cohen--Macaulay test ideal
    and multiplier ideals (Theorem \ref{thm:compmultideals}) that strengthens
    \citeleft\citen{MS18}\citemid Theorem 6.3\citepunct
    \citen{MS}\citemid Proposition 3.7 and Theorem 6.21\citepunct
    \citen{MSTWW}\citemid Theorem 5.1\citepunct
    \citen{Rob}\citemid Theorem 3.9\citeright.
    This requires proving that weakly functorial big Cohen--Macaulay
    $R^+$-algebras exist in equal characteristic zero
    (\ref{thm:weaklyfunctorialequalchar0}).
    This existence result uses ultraproducts, Lefschetz hulls, and the theory of
    seed algebras in equal characteristic zero
    \cite{AS07,Sch10,DRG19} and is of independent interest.
  \item We then apply the theory of multiplier ideals
    on excellent schemes of equal characteristic zero, which was developed in
    \cite{dFM09,JM12,ST}.
    An essential new ingredient is
    our recent generalization of the Kawamata--Viehweg vanishing
    theorem for proper morphisms of schemes of equal characteristic zero
    \cite[Theorem A]{Mur}, which is used to verify the Skoda-type theorem
    $(\ref{multiplierideal:skoda})$ when $X$ is no longer essentially of finite
    type over a field.
\end{enumerate}
As a result, our proof strategy using multiplier/test ideals applies to all
characteristics.
We also obtain new proofs of Theorems \ref{thm:mainjohnson} and
\ref{thm:mainjohnsonhhty} in equal characteristic that use multiplier/test
ideals.
As far as we are aware, the only proofs of Theorems \ref{thm:mainjohnson} and
\ref{thm:mainjohnsonhhty} in full generality in equal characteristic that exist
in the literature \cite{Joh14} use tight closure (see \cite[Theorems 3.1 and
4.1]{TY08} for a proof of a special case of Theorem \ref{thm:mainjohnsonhhty}
using multiplier/test ideals).\medskip
\par Lastly, 
in equal characteristic zero,
we also give an independent proof of Theorems
\ref{thm:mainelshhms}, \ref{thm:mainjohnson}, and \ref{thm:mainjohnsonhhty}
using only multiplier ideals.
This new proof of Theorem
\ref{thm:mainelshhms} in equal characteristic zero does not rely on the
N\'eron-type desingularization theorem due to Artin and Rotthaus
\cite[Theorem 1]{AR88}.
This proof therefore answers a question of Schoutens \cite[p.\ 179
and p.\ 187]{Sch03}, who asked whether one could show Theorem
\ref{thm:mainelshhms} in equal characteristic zero without using
\cite[Theorem 1]{AR88}.
Note that the proof described above using
Robinson's version of test ideals and Theorem
\ref{thm:ty23like} still depends on the theorem of Artin and Rotthaus
\cite[Theorem 1]{AR88} because as far as we are aware, all proofs
for the existence of big Cohen--Macaulay $R^+$-algebras in equal
characteristic zero rely on \cite[Theorem 1]{AR88} or stronger results.
\begin{remark}
  After writing this paper, we were informed by Wenliang Zhang that Zhang had
  independently shown a version of Theorem \ref{thm:mainjohnson} when
  $I$ is radical and $R$ has geometrically reduced formal
  fibers using the methods in \cite{MS18}.
\end{remark}
\subsection{Outline}
This paper is structured as follows.\medskip
\par We review some preliminaries in \S\ref{sect:prelims}.
We show that the definition of symbolic
powers in \eqref{eq:hhsymbdef} matches the definition used in \cite{HH02}
(Lemma \ref{lem:symbolicdefs}).
We also prove a reduction step for our main theorems (Lemma
\ref{lem:hh02reductions}) that will be used in both parts of this paper.
We then review the notions of absolute integral closures originating from
\cite{Art71} and of (balanced) big Cohen--Macaulay algebras from
\cite{Hoc75queens,Sha81}.
We then prove that weakly functorial big
Cohen--Macaulay $R^+$-algebras exist in equal characteristic zero (Theorem
\ref{thm:weaklyfunctorialequalchar0}) using the methods developed in
\cite{AS07,DRG19}.
This last result is of independent interest.\medskip
\par The remainder of the paper is divided into two parts, corresponding to the
two proof strategies for our main theorems using closure theory and
multiplier/test ideals, respectively.\medskip
\par In Part \ref{part:closure-elshhms}, we prove Theorems \ref{thm:mainelshhms},
\ref{thm:mainjohnson}, and \ref{thm:mainjohnsonhhty} via closure theory.
In \S\ref{sect:closureoperations}, we review the necessary background material
we need from closure theory.
The closure operations we use in mixed characteristic, namely
full extended plus ($\epf$) closure \cite{Hei01} and weak $\epf$
($\wepf$) closure \cite{Jia21}, are defined in \S\ref{sect:clmixedchar}.
We then state Dietz's axioms for closure operations \cite{Die10} and R.G.'s
algebra axiom \cite{RG18}, and collect references for proofs of these axioms and
for Brian\c{c}on--Skoda-type theorems for various closure operations in Table
\ref{tab:dietzclosures}.
We also state a result of R.G. from \cite{RG18} stating that Dietz
closures satisfying R.G.'s algebra axiom are contained in big Cohen--Macaulay
algebra closures (Proposition \ref{prop:rg41}).
Finally, we prove Theorems \ref{thm:mainelshhms}, \ref{thm:mainjohnson}, and
\ref{thm:mainjohnsonhhty} in \S\ref{sect:proofs}.\medskip
\par In Part \ref{part:elshhms}, we prove Theorems
\ref{thm:mainelshhms},
\ref{thm:mainjohnson}, \ref{thm:mainjohnsonhhty}, and \ref{thm:ty23like} via
multiplier/test ideals.
In \S\ref{sect:hls}, we review some definitions and preliminaries on
multiplier/test ideals.
In particular, we define the big Cohen--Macaulay test ideals from
\cite{MS}, the $+$-test ideals from \cite{HLS}, and the
multiplier ideals from \cite{dFM09,ST}.
In \S\ref{sect:bcmtestfixedgens}, we state the definition of
Robinson's version of perfectoid/big Cohen--Macaulay test ideals from
\cite{Rob}.
We then prove that they are related to
the test ideals from \S\ref{sect:hls} (Lemma
\ref{lem:testidealasasum}).
In the remainder of \S\ref{sect:bcmtestfixedgens}, we prove
foundational results on Robinson's test ideals necessary for our proofs of our
main theorems, in particular an
unambiguity statement for exponents (Proposition \ref{prop:notambig}),
the subadditivity theorem (Theorem \ref{thm:subadditivity}), and a comparison
with multiplier ideals (Theorem \ref{thm:compmultideals}).
Finally, in \S\ref{sect:mainproofs}, we prove Theorems
\ref{thm:mainelshhms},
\ref{thm:mainjohnson}, \ref{thm:mainjohnsonhhty}, and \ref{thm:ty23like}
using the various versions of test ideals in \cite{MS,PRG,Rob,ST,HLS}.
In equal characteristic zero, we also use multiplier ideals.

\addtocontents{toc}{\protect\setcounter{tocdepth}{1}}
\subsection*{Conventions}
All rings are commutative with identity, and all ring maps are unital.

\subsection*{Acknowledgments}
We are grateful to
Hailong Dao,
Rankeya Datta,
Neil Epstein,
Elo\'{i}sa Grifo,
Jack Jeffries,
Zhan Jiang,
J\'anos Koll\'ar,
Zhenqian Li,
Mircea Musta\c{t}\u{a},
Rebecca R.G.,
Daniel Smolkin,
Kevin Tucker,
and
Farrah Yhee
for helpful conversations.
We would especially like to thank
Melvin Hochster, 
Linquan Ma,
Karl Schwede,
Irena Swanson,
Jugal Verma,
and
Wenliang Zhang
for discussing their results with us.
We are also grateful to Karl Schwede for pointing out the reference \cite{Rob}.
Finally, we would like to thank Farrah Yhee for helpful edits
on multiple drafts of the introduction to this manuscript.
\addtocontents{toc}{\protect\setcounter{tocdepth}{2}}

\addtocontents{toc}{\protect\medskip}
\section{Preliminaries}\label{sect:prelims}
\subsection{Symbolic powers}
For completeness,
we show that the definition of symbolic powers in \eqref{eq:hhsymbdef}
matches the definition used in \cite{HH02}.
The former definition is the one used in the survey \cite[p.\
388]{DDSGHNB18}.
A similar argument will appear in the proof of Proposition
\ref{prop:elshhmsmd}$(\ref{prop:elshhmsmdbcm})$.
\begin{lemma}\label{lem:symbolicdefs}
  Let $R$ be a Noetherian ring, and let $I \subseteq R$ be an ideal.
  Denote by $R_W$ the localization of $R$ with respect to the multiplicative set
  $W = \bigcup_{\fp \in \Ass_R(R/I)} \fp$.
  Then, we have
  \begin{equation}\label{eq:diffdefs}
    I^nR_W \cap R = \bigcap_{\fp \in \Ass_R(R/I)} I^nR_\fp \cap R.
  \end{equation}
\end{lemma}
\begin{proof}
  The natural maps $R \to R_\fp$ factor through $R_W$ by the universal property
  of localization.
  We therefore have inclusions $I^nR_W \cap R \subseteq I^nR_\fp \cap R$ for
  every $\fp \in \Ass_R(R/I)$, which yields the inclusion ``$\subseteq$'' in
  \eqref{eq:diffdefs}.
  Thus, it suffices to show the inclusion ``$\supseteq$'' in
  \eqref{eq:diffdefs}.
  \par Let $u \in \bigcap_{\fp \in \Ass_R(R/I)} I^nR_\fp \cap R$.
  Denote by $\{\fp_\ell\}$ the subset of $\Ass_R(R/I)$ consisting of the
  associated primes that are maximal in $\Ass_R(R/I)$ with respect to inclusion.
  For each $\ell$, there exists an element $c_\ell \notin
  \fp_\ell$ such that $c_\ell u \in I^n$.
  For each $\ell$, we can also choose $d_\ell \in \fp_\ell - \bigcup_{j \ne
  \ell} \fp_j$ by prime avoidance.
  Now consider the element
  \begin{align*}
    c &= \sum_\ell \biggl( c_\ell \prod_{j \ne \ell} d_j \biggr).
  \intertext{Note that $c \notin \bigcup_\ell \fp_\ell$, for
  otherwise if $c \in \fp_{\ell_0}$ for some $\ell_0$, we have}
    c_{\ell_0} \prod_{j \ne \ell_0} d_j &=
    c - \sum_{\ell \ne \ell_0} \biggl( c_\ell \prod_{j \ne \ell} d_j \biggr) \in
    \fp_{\ell_0},
    \intertext{which contradicts the fact that $c_{\ell_0} \notin \fp_{\ell_0}$
    and $d_j \notin \fp_{\ell_0}$ for every $j \ne \ell_0$.
    Since $c_\ell u \in I^n$ by assumption, we see that}
    cu &= \sum_\ell \biggl( c_\ell u \prod_{j \ne \ell} d_j \biggr) \in I^n.
  \end{align*}
  Finally, since $c \in W = R - \bigcup_\ell \fp_\ell$,
  we have $u \in I^nR_W \cap R$.
\end{proof}
\begin{remark}
  There is a different notion of symbolic powers in the literature due to
  Verma \cite[p.\ 205]{Ver87} obtained by intersecting over minimal primes
  instead of associated primes on the right-hand side of \eqref{eq:diffdefs}.
  Using notation from \cite[p.\ 691]{HJKN23}, this version of the $n$-th
  symbolic power is defined as follows:
  \[
    \prescript{\mathsf{m}}{}{I}^{(n)} \coloneqq \bigcap_{\fp \in \Min_R(R/I)}
    I^nR_\fp \cap R.
  \]
  \par The analogue of Theorem \ref{thm:mainelshhms} using this notion of
  symbolic powers is not true.
  For example, let $K$ be a field and let $R = K[x,y]$.
  For every integer $k > 0$, the ideal
  \[
    I_k = (x)(x,y)^{k} = (x^{k+1},x^ky,\ldots,xy^k) = (x) \cap (x,y)^{k+1} \subseteq R
  \]
  satisfies
  \[
    \prescript{\mathsf{m}}{}{(I_k)}^{(kn)} = (x)^{kn} \not\subseteq
    (x)^n(x,y)^{kn} = (I_k)^n.
  \]
\end{remark}
\subsection{Reductions for main theorems}
\par We prove some initial reductions for
Theorems \ref{thm:mainjohnson} and
\ref{thm:mainjohnsonhhty}
that will be used in both Parts \ref{part:closure-elshhms} and
\ref{part:elshhms}.
\begin{lemma}[{cf.\ \citeleft\citen{HH02}\citemid Lemma 2.4$(b)$ and Theorem
  4.4\citepunct \citen{TY08}\citemid Theorem 3.1\citeright}]
  \label{lem:hh02reductions}
  It suffices to show Theorems \ref{thm:mainjohnson} and
  \ref{thm:mainjohnsonhhty} under the additional assumptions that
  $R$ is complete local ring $(R,\fm)$, and that
  the localizations of $R$ at the associated primes of $R/I$ have
  infinite residue fields.
  Moreover, after this reduction,
  we may further replace the residue field $k$ of $(R,\fm)$ with any field
  $k'$ containing $k$.
\end{lemma}
\begin{proof}
  We proceed in four steps.
  \begin{step}\label{step:infiniteresidues}
    It suffices to consider the case when the localizations of $R$ at the
    associated primes of $R/I$ have infinite residue fields.
  \end{step}
  Let $X$ be an indeterminate.
  Since $R \to R[X]$ is faithfully flat, it suffices to show
  \eqref{eq:mainjohnsonincl} and \eqref{eq:mainjohnsonhhtyincl}
  after extending scalars to $R[X]$.
  By \cite[Discussion 2.3$(b)$]{HH02}, the associated primes of $IR[X]$ are of
  the form $\fp R[X]$ where $\fp \in \Ass_R(R/I)$, the analytic spreads of
  $IR[X]_{\fp R[X]}$ are the same as those for $IR_\fp$, and the
  symbolic powers of $IR[X]$ are the extensions of the symbolic powers of $I$.
  We may therefore replace $R$ by $R[X]$ and $I$ by $I[X]$ to assume that the
  localizations of $R$ at the associated primes of $R/I$ have infinite residue
  fields.
  For Theorem \ref{thm:mainjohnsonhhty}, we furthermore replace $R[X]$ by
  $R[X]_{\fm R[X]}$ to remain in the local case.
  \begin{step}\label{step:local}
    It suffices to consider the case when $R$ is local and the
    localizations of $R$ at the associated primes of $R/I$ have infinite
    residue fields.
  \end{step}
  For Theorem \ref{thm:mainjohnson}, the inclusion \eqref{eq:mainjohnsonincl}
  can be checked after localizing at every prime ideal of $R$, and hence we may
  assume that $R$ is local.
  For Theorem \ref{thm:mainjohnsonhhty}, we are already in the local case.
  \begin{step}\label{step:completelocal}
    For $R$ as in Step \ref{step:local},
    it suffices to show Theorems \ref{thm:mainjohnson} and
    \ref{thm:mainjohnsonhhty} after passing to a faithfully flat local extension
    $R \hookrightarrow S$.
  \end{step}
  \par By \cite[Discussion 2.3$(c)$]{HH02},
  the analytic spread of $IS_\fq$ for every $\fq \in
  \Ass_{S}(S/IS)$ is at most the analytic spread
  of $IR_{\fq \cap R}$.
  Moreover, we have
  \[
    (IS)^{(n)} \cap R = \bigcap_{\fq \in
    \Ass_{S}(S/IS)} I^n S_\fq \cap R
    = \bigcap_{\fp \in \Ass_R(R/I)} I^n R_\fp \cap R = I^{(n)}
  \]
  for all $n \ge 1$, where the middle equality follows from
  \cite[Chapter IV, \S2, no.\ 6, Theorem 2 and Corollary
  1]{Bou72}
  and the fact that $R_\fp \to S_\fq$
  is faithfully flat when $\fp = \fq \cap R$.
  We may therefore replace $R$ by $S$.
  Note that for every $\fq \in \Ass_{S}(S/IS)$,
  the residue field of $S_\fq$ contains the residue field
  of $R_{\fq \cap R}$, and hence is still infinite.
  \begin{step}
    For $R$ as in Step \ref{step:local}, we can find a faithfully flat local
    extension $(R,\fm,k) \hookrightarrow (S,\fn,l)$ such that $l = k'$ and $S$
    is complete.
  \end{step}
  By \cite[Chapitre IX, Appendice, n\textsuperscript{o} 2, Corollaire au
  Th\'eor\`eme 1]{BouAC89}, there exists a gonflement $(R,\fm) \to
  (R',\fm')$ where $R'/\fm' = k'$.
  This map is a flat local map such that $R'$ is regular
  \cite[Chapitre IX, Appendice, n\textsuperscript{o} 2, Proposition 2]{BouAC89}.
  We now consider the composition $R \to R' \to \widehat{R'}$,
  where the second map is the $\fm'$-adic completion map.
  This map is faithfully flat, and Step \ref{step:completelocal}
  shows that we may replace $R$ by $S = \widehat{R'}$.
\end{proof}

\subsection{Absolute integral closures}\label{sect:rplus}
\par We now define the absolute integral closure of an
integral scheme following \cite{HLS}.
This definition is modeled after Artin's notion of the absolute integral closure
of an integral domain \cite[p.\ 283]{Art71}.
\begin{definition}[see {\cite[\S2.2]{HLS}}]\label{def:xplus}
  Let $X$ be an integral scheme.
  Fix an algebraic closure $\overline{K(X)}$ of the function field $K(X)$ of
  $X$, and consider the collection $\{f\colon Y \to X\}$
  of all finite surjective morphisms from integral schemes
  fitting into the commutative diagram
  \[
    \begin{tikzcd}
      \Spec\bigl(\overline{K(X)}\bigr) \rar \arrow{dr} & Y\dar{f}\\
      & X\mathrlap{.}
    \end{tikzcd}
  \]
  The \textsl{absolute integral closure of $X$} is the inverse limit
  \[
    X^+ \coloneqq \varprojlim_{f\colon Y \to X} Y
  \]
  in the category of schemes, which exists by \cite[Proposition 8.2.3]{EGAIV3}.
  We denote the canonical projection morphism by $\nu\colon X^+ \to X$.
  \par Now suppose that $X = \Spec(R)$ for an integral domain $R$.
  The \textsl{absolute integral closure of $R$} is the ring $R^+$ of global
  sections of $X^+$, which is an affine scheme by \cite[(8.2.2)]{EGAIV3}.
  The absolute integral closure $R^+$ can also be described as the integral
  closure of $R$ in an algebraic closure $\overline{\Frac(R)}$ of its fraction
  field.
\end{definition}
\subsection{Big Cohen--Macaulay algebras}
Next, we define (balanced) big Cohen--Macaulay algebras.
\begin{citeddef}[{\citeleft\citen{Hoc75queens}\citemid pp.\ 110--111\citepunct
  \citen{Sha81}\citemid Definition 1.4\citeright}]
  Let $(R,\fm)$ be a Noetherian local ring of dimension $d$, and let
  $x_1,x_2,\ldots,x_d$ be a system of parameters of $R$.
  An $R$-algebra $B$ is \textsl{big Cohen--Macaulay over $R$ with respect to
  $x_1,x_2,\ldots,x_d$} if $x_1,x_2,\ldots,x_d$ is a regular sequence on $B$,
  and is a \textsl{(balanced) big Cohen--Macaulay algebra over $R$} if it is big
  Cohen--Macaulay over $R$ with respect to every system of parameters
  $x_1,x_2,\ldots,x_d$.
\end{citeddef}
\begin{remark}\label{rem:bigcmalgexist}
  Let $(R,\fm)$ be a Noetherian local ring.
  \begin{enumerate}[label=$(\roman*)$,ref=\roman*]
    \item\label{rem:bigcmalgexistequalchar}
      If $R$ is of equal characteristic, then there is a big Cohen--Macaulay
      algebra over $R$ \cite[Theorem 8.1]{HH92} (see also \cite[Theorems 11.1
      and 11.4]{Hoc94}).
      This big Cohen--Macaulay algebra can be taken to be an $R^+$-algebra
      if $R$ is a domain.
      In equal characteristic $p > 0$, this follows from the construction in
      \cite[Theorem 8.1]{HH92} and applying \cite[Proposition 1.2]{HH95} (see
      also \citeleft\citen{Hoc94}\citemid Theorem 11.1\citepunct
      \citen{Die07}\citemid Theorem 6.9\citeright).
      In equal characteristic zero, this was shown in \cite[Corollary
      5.2]{DRG19},
      and can also be shown using the methods in \cite{HH95,HHchar0}, as pointed
      out to us by Hochster (cf.\ \cite[Theorem 11.4]{Hoc94}).
    \item If $R$ is an excellent biequidimensional domain
      of equal characteristic $p > 0$ or a domain of equal
      characteristic $p > 0$ that is a homomorphic image of a Gorenstein local
      ring, then the absolute integral closure $R^+$ of
      $R$ is a big Cohen--Macaulay algebra over $R$
      \citeleft\citen{HH92}\citemid Main Theorem 5.15\citepunct
      \citen{HL07}\citemid Corollary 2.3$(b)$\citeright.
    \item If $R$ is of mixed characteristic, then $R$ has a big Cohen--Macaulay
      algebra 
      \cite[Th\'eo\-r\`eme 0.7.1]{And18b}.
      This big Cohen--Macaulay algebra can be taken to be an $R^+$-algebra if
      $R$ is a domain
      \citeleft\citen{Shi18}\citemid Corollary 6.5 and Remark 6.6\citepunct
      \citen{And20}\citemid Theorem 3.1.1\citeright.
    \item\label{rem:bigcmalgexistrplusmixedchar}
      If $R$ is an excellent domain of mixed characteristic, then
      the $p$-adic completion $\widehat{R^+}$ of the absolute integral closure
      $R^+$ of $R$ is a big Cohen--Macaulay algebra over $R$
      \citeleft\citen{Bha}\citemid Corollary 5.17\citepunct
      \citen{BMPSTWW}\citemid Corollary 2.10\citeright.
  \end{enumerate}
\end{remark}
\subsection{Weakly functorial big Cohen--Macaulay
\texorpdfstring{\for{toc}{$R^+$}\except{toc}{\emph{R}\textsuperscript{+}}}{R\textasciicircum+}-algebras
in equal characteristic zero}
We prove the following result showing the existence of weakly functorial big
Cohen--Macaulay $R^+$-algebras in equal characteristic zero using the work of
Aschenbrenner and Schoutens \cite{AS07} and
Dietz and R.G. \cite{DRG19}.
This result is of independent interest and will only be used in Part
\ref{part:elshhms} of this paper when proving Theorem \ref{thm:ty23like} in
equal characteristic zero.
\par The formulation of this existence result
is based on the mixed characteristic statements in
\citeleft\citen{And20}\citemid Theorems 1.2.1 and
4.1.1\citepunct \citen{MSTWW}\citemid Theorem A.5\citeright.
See \cite[\S2.1]{Sch10} for definitions of ultraproducts and ultrarings,
see \cite[\S4]{AS07} for the definition of a Lefschetz hull and \cite[p.\
261]{AS07} for the definition of an absolutely normalizing Lefschetz hull,
and see \cite[Definition 2.3 and Lemma 2.5]{DRG19} for the definition of a
rational seed.
\begin{theorem}\label{thm:weaklyfunctorialequalchar0}
  Let $f\colon (R,\fm) \to (S,\fn)$ be a local map of Noetherian local domains
  of equal characteristic zero.
  Fix faithfully flat, absolutely normalizing Lefschetz hulls
  \[
    \fD(R) = \ulim_{w \in W} R_w \qquad \text{and} \qquad
    \fD(S) = \ulim_{w \in W} S_w
  \]
  for $R$ and $S$ with characteristic $p > 0$ approximations $R_w$ and $S_w$,
  where $W$ is an infinite set with a non-principal ultrafilter
  $\cW$, such that the diagram
  \[
    \begin{tikzcd}
      \fD(R) \rar & \fD(S)\\
      R \rar\uar & S\uar
    \end{tikzcd}
  \]
  commutes.
  \begin{enumerate}[label=$(\roman*)$,ref=\roman*]
    \item\label{thm:weakfunchar0exists}
      There exists an ultraring $B_\natural$ with respect to $\cW$ that
      is an $R^+$-algebra and is a big Cohen--Macaulay algebra over $R$.
    \item\label{thm:weaklyfunctorialequalchar0funct}
      For every choice of $B_\natural$ as in $(\ref{thm:weakfunchar0exists})$,
      there exists a commutative diagram
      \begin{equation}\label{eq:weakfunchar0}
        \begin{tikzcd}
          B_\natural \rar & C_\natural\\
          R^+ \rar{f^+}\uar & S^+\uar\\
          R \rar{f} \uar & S \uar
        \end{tikzcd}
      \end{equation}
      where $C_\natural$ is a big Cohen--Macaulay algebra over $S$ that is an
      ultraring with respect to $\cW$.
      Moreover, $f^+$ can be given in advance.
  \end{enumerate}
\end{theorem}
\begin{proof}
  The first diagram exists by \cite[Remark 4.26]{AS07}, where we note that the
  choice of Lefschetz hull depends on the choice of a Lefschetz field of
  sufficiently large cardinality (e.g.\ larger than $2^{\lvert R \rvert}$ and
  $2^{\lvert S \rvert}$).
  We use the construction in \cite[Remark 4.26]{AS07} to moreover choose
  $\fD(R)$ and $\fD(S)$ to be absolutely normalizing.
  \par We next show $(\ref{thm:weakfunchar0exists})$.
  Since $\fD(R)$ was chosen to be absolutely normalizing, the construction in
  \cite[(7.7)]{AS07}
  shows that $R$ is a rational seed in the sense of \cite[Definition
  2.3]{DRG19} (see also \cite[Theorem 2.4]{DRG19}).
  By \cite[Theorem 5.1]{DRG19}, the absolute integral closure
  $R^+$ is also a rational seed over $R$ (see \cite[Corollary 5.2]{DRG19}),
  and hence $B_\natural$ exists.
  \par For $(\ref{thm:weaklyfunctorialequalchar0funct})$, if $f^+$ is not given
  in advance, then we can construct
  the bottom square in the diagram by \cite[Proposition
  1.2]{HH95}.
  Now given such a square,
  since $B_\natural \otimes_R S$ and $S^+$ are rational
  seeds over $S$ by \cite[Theorem 3.5 and Corollary 5.2]{DRG19}, the
  tensor product
  \[
    B_\natural \otimes_R S^+ \cong
    B_\natural \otimes_R S \otimes_S S^+
  \]
  is a rational seed over $S$ by \cite[Theorem 3.3]{DRG19}.
  Thus, there is a map $B_\natural \otimes_R S^+ \to C_\natural$ to a big
  Cohen--Macaulay algebra over $S$ that is an ultraring with respect to $\cW$
  and makes the diagram commute.
\end{proof}

\begingroup
\makeatletter
\renewcommand{\@secnumfont}{\bfseries}
\part{Proof via closure theory}\label{part:closure-elshhms}
\makeatother
\endgroup
\section{Closure operations}\label{sect:closureoperations}
In this section, we review some preliminaries on closure operations.
We first
define full extended plus ($\epf$) closure \cite{Hei01} and weak $\epf$
($\wepf$) closure \cite{Jia21}.
Next,
we state Dietz's axioms for closure operations \cite{Die10} and R.G.'s
algebra axiom \cite{RG18}, and collect references for proofs of these axioms and
for Brian\c{c}on--Skoda-type theorems for various closure operations in Table
\ref{tab:dietzclosures}.
We also state a result of R.G. from \cite{RG18} stating that Dietz
closures satisfying R.G.'s algebra axiom are contained in big Cohen--Macaulay
algebra closures (Proposition \ref{prop:rg41}), which is a key ingredient in our
closure-theoretic proof of Theorems \ref{thm:mainelshhms},
\ref{thm:mainjohnson}, and \ref{thm:mainjohnsonhhty}.
\subsection{Closure operations in mixed characteristic}\label{sect:clmixedchar}
We recall the definition of Heitmann's full extended plus ($\epf$) closure.
\begin{citeddef}[{\citeleft\citen{Hei01}\citemid Definition on pp.\
  804--805\citepunct \citen{RG16}\citemid Definition 7.1\citepunct
  \citen{HM21}\citemid Definition 2.3\citepunct
  \citen{Jia21}\citemid Definition 2.3\citeright}]
  \label{def:epf}
  Let $p > 0$ be a prime number.
  Let $R$ be a domain such that the image of $p$ lies in the Jacobson
  radical of $R$.
  For an inclusion $Q \subseteq M$ of finitely generated $R$-modules,
  the \textsl{full extended plus} ($\epf$) \textsl{closure} of $Q$ in
  $M$ is
  \[
    Q^\epf_M \coloneqq \Set*{u \in M \given \begin{tabular}{@{}c@{}}
      there exists $c \in R - \{0\}$ such that\\
      $c^\varepsilon \otimes u \in \im(R^+ \otimes_R Q \to R^+ \otimes_R M) +
      p^N(R^+ \otimes_R M)$\\
      for every $\varepsilon \in \QQ_{>0}$ and every $N \in \ZZ_{>0}$
    \end{tabular}}.
  \]
\end{citeddef}
We also recall the definition of Jiang's weak $\epf$ ($\wepf$) closure.
\begin{citeddef}[{\cite[Definition 4.1]{Jia21}}]\label{def:wepf}
  Let $p > 0$ be a prime number.
  Let $R$ be a domain such that the image of $p$ lies in the Jacobson
  radical of $R$.
  For an inclusion $Q \subseteq M$ of finitely generated $R$-modules,
  the \textsl{weak $\epf$} ($\wepf$) \textsl{closure} of $Q$ in $M$ is
  \[
    Q^\wepf_M \coloneqq \bigcap_{N = 1}^\infty (Q + p^NM)^\epf_M.
  \]
  By definition, we have the inclusion $Q^\epf_M \subseteq Q^\wepf_M$
  (see \cite[Remark 4.2]{Jia21}).
\end{citeddef}

\subsection{Axioms for closure operations}
In this subsection, we fix the following notation.
\begin{notation}\label{notation:closure}
  We denote by $R$ a ring, and by
  $Q$, $M$, and $W$ arbitrary finitely generated
  $R$-modules such that $Q \subseteq M$.
  We consider an operation $\cl$ sending submodules $Q \subseteq M$ to an
  $R$-module $Q_M^\cl$.
  We denote $I^\cl \coloneqq I^\cl_R$ for ideals $I \subseteq R$.
\end{notation}
We state Dietz's axioms for closure operations from \cite{Die10} with
some conventions from \citeleft\citen{Eps12}\citemid Definition 2.1.1\citepunct
\citen{Die18}\citemid Definition 1.1\citeright.
\begin{citedaxioms}[{\citeleft\citen{Die10}\citemid Axioms 1.1\citeright}]
  \label{axioms:dietzrg}
  Fix notation as in Notation \ref{notation:closure}.
  We say that the operation $\cl$ is a \textsl{closure operation}
  if it satisfies the following three axioms:
  \begin{enumerate}[label=$(\arabic*)$,ref=\arabic*]
    \item\label{axioms:dietz1}
      (Extension) $Q_M^\cl$ is a submodule of $M$ containing $Q$.
    \item (Idempotence) $(Q_M^\cl)^\cl = Q_M^\cl$.
    \item\label{axioms:dietz3}
      (Order-preservation) If $Q \subseteq M \subseteq W$, then $Q_W^\cl
      \subseteq M_W^\cl$.
  \end{enumerate}
  We also consider the following axioms:
  \begin{enumerate}[resume,label=$(\arabic*)$,ref=\arabic*]
    \item (Functoriality) Let $f\colon M \to W$ be a map of $R$-modules.
      Then, $f(Q_M^\cl) \subseteq f(Q)^\cl_W$.
    \item\label{axioms:dietz5}
      (Semi-residuality) If $Q^\cl_M = Q$, then $0^\cl_{M/Q} = 0$.
  \end{enumerate}
  Now suppose that $R$ is a Noetherian local domain $(R,\fm)$.
  We say that $\cl$ is a \textsl{Dietz closure} if it satisfies axioms
  $(\ref{axioms:dietz1})$--$(\ref{axioms:dietz5})$ above, and the following two
  additional axioms:
  \begin{enumerate}[resume,label=$(\arabic*)$,ref=\arabic*]
    \item The maximal ideal $\fm$ and the zero ideal $0$ are $\cl$-closed in
      $R$, i.e., $\fm^\cl = \fm$ and $0^\cl = 0$.
    \item (Generalized colon-capturing) Let $x_1,x_2,\ldots,x_{k+1}$ be a
      partial system of parameters for $R$ and let $J = (x_1,x_2,\ldots,x_k)$.
      Suppose there exists a surjective map $f\colon M \to R/J$ of $R$-modules,
      and let $v \in M$ be an arbitrary element
      such that $f(v) = x_{k+1} + J$.
      Then,
      \[
        (Rv)^\cl_M \cap \ker(f) \subseteq (Jv)^\cl_M.
      \]
  \end{enumerate}
\end{citedaxioms}
To state R.G.'s algebra axiom, we first define the notion of a $\cl$-phantom
extension.
\begin{citeddef}[{\cite[Definition 2.2]{Die10}}]\label{def:clphantom}
  Fix notation as in Notation \ref{notation:closure}.
  Suppose $\cl$ satisfies axioms
  $(\ref{axioms:dietz1})$--$(\ref{axioms:dietz5})$ above.
  Let $M$ be a finitely generated $R$-module and let $\alpha\colon R \to M$ be
  an injective map of $R$-modules.
  Consider the short exact sequence
  \[
    0 \longrightarrow R \overset{\alpha}{\longrightarrow} M \longrightarrow Q
    \longrightarrow 0,
  \]
  and let $\epsilon \in \Ext^1_R(Q,R)$ be the element corresponding to this
  short exact sequence via the Yoneda correspondence.
  \par Fix a projective resolution $P_\bullet$ of $Q$ consisting of finitely
  generated projective $R$-modules $P_i$.
  We say that $\epsilon$ is \textsl{$\cl$-phantom} if
  \[
    \epsilon \in \bigl(\im\bigl(\Hom_R(P_0,R) \longrightarrow
    \Hom_R(P_1,R)\bigr)\bigr)^\cl_{\Hom_R(P_1,R)}.
  \]
  This definition does not depend on the choice of
  $P_\bullet$ by \cite[Discussion 2.3]{Die10}.
  \par We say that $\alpha$ is a \textsl{$\cl$-phantom extension} if $\epsilon$
  is $\cl$-phantom.
\end{citeddef}
We now state R.G.'s algebra axiom.
\begin{citedaxiom}[{\citeleft\citen{RG18}\citemid Axiom 3.1\citeright}]
  Fix notation as in Notation \ref{notation:closure}.
  If $\cl$ satisfies axioms $(\ref{axioms:dietz1})$--$(\ref{axioms:dietz5})$
  above, we consider the following axiom:
  \begin{enumerate}[start=8,label=$(\arabic*)$,ref=\arabic*]
    \item\label{axiom:rgalg}
      (Algebra axiom) Let $\alpha \colon R \to M$ be a map of $R$-modules. 
      If $\alpha$ is an $\cl$-phantom extension, then the map $\alpha'\colon R
      \to \Sym^2_R(M)$ where $1 \mapsto \alpha(1) \otimes \alpha(1)$
      is a $\cl$-phantom extension.
  \end{enumerate}
\end{citedaxiom}
\begin{table}[t]
  \begin{ThreePartTable}
    \begin{TableNotes}
      \item[$\star$]\label{tn:char0weakalg}
        As far as we are aware, it is open whether R.G.'s algebra axiom holds
        for these closure operations.
        However, a suitable replacement that works for our proofs does hold. See
        Remark \ref{rem:weakalgaxiom}.
      \item[$F$-finite]\label{tn:ffinite} Jiang's results hold when $R/\fm$
        is $F$-finite.
    \end{TableNotes}
    {\scriptsize
    \begin{longtable}[c]{ccccc}
      \toprule
      characteristic & closure ($\cl$)
      & Dietz closure & R.G.'s algebra axiom & Brian\c{c}on--Skoda\\
      \cmidrule(lr){1-2} \cmidrule(lr){3-5}
      \multirow{2}{*}{char.\ $p > 0$}
      & tight ($*$)
      & \cite[Ex.\ 5.4]{Die10} & \cite[Prop.\ 3.6]{RG18}
      & \cite[Thm.\ 8.1]{HH94}\\
      & plus ($+$)
      & \cite[Ex.\ 5.1]{Die10} & \cite[Prop.\ 3.11]{RG18}
      & \cite[Thm.\ 7.1]{HH95}\\
      \midrule
      \multirow{3}{*}{equal char.\ $0$} & small equational tight
      (${*\mathsf{eq}}$)
      & \cite[Ex.\ 5.7]{Die10}
      & \textbf{open}\tnotex{tn:char0weakalg}
      & \cite[Thm.\ 4.1.5]{HHchar0}\\
      & big equational tight (${*\mathsf{EQ}}$)
      & \cite[Ex.\ 5.7]{Die10}
      & \textbf{open}\tnotex{tn:char0weakalg}
      & \cite[Thm.\ 4.1.5]{HHchar0}\\
      & $\mathfrak{B}$- ($+$) & \cite[Thm.\ 4.2]{Die10} & \cite[Prop.\
      3.11]{RG18} & \cite[Thm.\ 7.14.3]{AS07}\\
      \midrule
      \multirow{3}{*}{mixed char.}
      & full extended plus ($\epf$) & \textbf{open}
      & \textbf{open}\tnotex{tn:char0weakalg} & \cite[Thm.\ 4.2]{Hei01}\\
      & weak $\epf$ ($\wepf$) &
      \multicolumn{2}{c}{\cite[Thm.\ 4.8]{Jia21}\tnotex{tn:ffinite}}
      & \cite[Thm.\ 4.2]{Hei01}\\
      & full rank 1 ($\ronef$) & \multicolumn{2}{c}{\cite[Thm.\ 4.8 and Rem.\
      4.7]{Jia21}\tnotex{tn:ffinite}} & \cite[Thm.\ 4.2]{Hei01}\\
      \bottomrule
      \insertTableNotes\\
      \caption{Some closure operations on Noetherian complete local domains.}
      \label{tab:dietzclosures}
    \end{longtable}}
  \end{ThreePartTable}
\end{table}
\par We also introduce an axiom asserting that a
closure-theoretic version of the Brian\c{c}on--Skoda theorem holds.
\begin{axiom}
  Fix notation as in Notation \ref{notation:closure}.
  If $\cl$ satisfies axiom $(\ref{axioms:dietz1})$ above, we consider the
  following axiom:
  \begin{enumerate}[start=9,label=$(\arabic*)$,ref=\arabic*]
    \item\label{axiom:brianconskoda}
      (Brian\c{c}on--Skoda-type theorem)
      Let $I \subseteq R$ be an ideal generated by at most $h$ elements.
      Then, $\overline{I^{h+k}} \subseteq (I^{k+1})^\cl$ for every integer $k
      \ge 0$.
  \end{enumerate}
\end{axiom}
\subsection{Algebra closures}
The key result we need about Dietz closures satisfying R.G.'s algebra axiom
$(\ref{axiom:rgalg})$ is
that they are related to algebra closures, which are defined as follows.
\begin{citeddef}[{\cite[Definition 2.3 and Remark 2.4]{RG16}}]
  Let $R$ be a ring, and let $S$ be an $R$-algebra.
  Let $Q \subseteq M$ be an inclusion of finitely generated $R$-modules.
  We then set
  \[
    Q_M^{\cl_S} \coloneqq 
    \Set[\big]{u \in M \given 1 \otimes u \in \im(S \otimes_R Q
    \to S \otimes_R M)}.
  \]
  We call the operation $\cl_S$ an \textsl{algebra closure}.
\end{citeddef}
\begin{examples}
  If $B$ is a big Cohen--Macaulay algebra over a Noetherian local domain, then
  the algebra closure $\cl_B$ is a Dietz closure \cite[Theorem 4.2]{Die10} and
  satisfies R.G.'s algebra axiom $(\ref{axiom:rgalg})$
  \cite[Proposition 3.11]{RG18}.
  We list two examples of such algebra closures appearing in
  Table \ref{tab:dietzclosures} for which a
  Brian\c{c}on--Skoda-type theorem $(\ref{axiom:brianconskoda})$ holds.
  \begin{enumerate}[label=$(\roman*)$,ref=\roman*]
    \item Let $R$ be an excellent biequidimensional local domain of equal
      characteristic $p > 0$, or a Noetherian local domain of equal
      characteristic $p > 0$ that is a homomorphic image of a Gorenstein local
      ring.
      Then, $R^+$ is a big Cohen--Macaulay algebra over $R$ by
      \citeleft\citen{HH92}\citemid Main Theorem 5.15\citepunct
      \citen{HL07}\citemid Corollary 2.3$(b)$\citeright.
      The associated algebra closure $\cl_{R^+}$ is called \textsl{plus
      closure} \cite[Definition 2.13]{Smi94}, and
      is denoted by $+$.
    \item Let $R$ be a Noetherian local ring of equal characteristic zero.
      Using ultraproducts, Aschenbrenner and Schoutens construct a big
      Cohen--Macaulay algebra $\mathfrak{B}(R)$ over $R$ \cite[(7.7)]{AS07}.
      The associated algebra closure $\cl_{\mathfrak{B}(R)}$ is called
      \textsl{$\mathfrak{B}$-closure} \cite[(7.13)]{AS07}, and is also denoted
      by $+$.
  \end{enumerate}
\end{examples}
We now state the following result due to R.G., which says that Dietz
closures satisfying R.G.'s algebra axiom $(\ref{axiom:rgalg})$
are contained in an algebra
closure defined by a big Cohen--Macaulay algebra.
\begin{citedprop}[{\cite[Proposition 4.1]{RG18}}]\label{prop:rg41}
  Let $R$ be a Noetherian local domain and let $\cl$ be a Dietz closure on $R$
  that satisfies R.G.'s algebra axiom $(\ref{axiom:rgalg})$.
  Then, there exists a big Cohen--Macaulay algebra $B$ over $R$ such that
  $Q^\cl_M \subseteq Q^{\cl_B}_M$
  for every inclusion $Q \subseteq M$ of finitely generated $R$-modules.
\end{citedprop}

\section{Proof of main theorems via closure theory}\label{sect:proofs}
We are now ready to prove 
Theorems \ref{thm:mainelshhms}, \ref{thm:mainjohnson}, and
\ref{thm:mainjohnsonhhty} using closure operations.
Theorem \ref{thm:mainelshhms} follows from
Theorem \ref{thm:mainjohnson} by setting $s_i = 0$ for all $i$.
It therefore suffices to show Theorems \ref{thm:mainjohnson} and
\ref{thm:mainjohnsonhhty}.\medskip
\par We start with the following result.
\begin{proposition}\label{prop:elshhmsmd}
  Let $(R,\fm)$ be a complete local domain.
  Consider an ideal $I \subseteq R$.
  Let $u \in I^{(M)}$ for an integer $M > 0$, and fix non-negative integers
  $s_1,s_2,\ldots,s_n$.
  \begin{enumerate}[label=$(\roman*)$,ref=\roman*]
    \item\label{prop:elshhmsmdrprime}
      Denote by $R'$ the integral closure of
      \[
        R[u^{1/M}] \subseteq \Frac(R)[u^{1/M}]
      \]
      in $\Frac(R)[u^{1/M}]$,
      where $u^{1/M}$ is a choice of $M$-th root of $u$ in an algebraic closure
      of $\Frac(R)$.
      Then, $R'$ is a complete normal local domain $(R',\fm')$.
      Moreover, if $R/\fm$ is $F$-finite and of characteristic $p > 0$, then
      $R'/\fm'$ is $F$-finite and of characteristic $p > 0$.
  \end{enumerate}
  Now assume that the localizations of $R$ at
  the associated primes of $R/I$ have infinite residue fields.
  Let $h$ be the largest analytic spread of $IR_\fp$, where $\fp$ ranges
  over all associated primes $\fp$ of $R/I$.
  \begin{enumerate}[resume,label=$(\roman*)$,ref=\roman*]
    \item\label{prop:elshhmsmdepf}
      Fix a closure operation $\cl$ on $R$ for which the
      Brian\c{c}on--Skoda-type theorem $(\ref{axiom:brianconskoda})$ holds.
      Let $\fp_\ell \in \Ass_R(R/I)$.
      Then, there exists $c_\ell \in R - \fp_\ell$ such that for every $i$, we
      have
      \[
        c_\ell u^{\frac{s_i+h}{M}} \in (I^{s_i+1}R')^\cl.
      \]
    \item\label{prop:elshhmsmdbcm}
      Suppose that if $R$ is of mixed characteristic, then $R/\fm$ is
      $F$-finite and of characteristic $p > 0$.
      Then, there exists a big Cohen--Macaulay algebra $B$ over $R'$
      such that the following holds:
      For every $\fp_\ell \in \Ass_R(R/I)$, there exists $c_\ell \in R -
      \fp_\ell$ such that for every $i$, we have
      \[
        c_\ell u^{\frac{s_i+h}{M}} \in I^{s_i+1}B.
      \]
      Moreover, if $R$ is regular, then for every $i$, we have
      \[
        u^{\frac{s_i+h}{M}} \in I^{(s_i+1)}B.
      \]
  \end{enumerate}
\end{proposition}
\begin{proof}
  We first prove $(\ref{prop:elshhmsmdrprime})$.
  The ring $R'$ is module-finite over $R$ and is complete local by
  \cite[Chapitre 0, Th\'eor\`eme 23.1.5 and Corollaire 23.1.6]{EGAIV1}.
  If $R/\fm$ is $F$-finite and of characteristic $p > 0$, then $R'/\fm'$ is
  $F$-finite and of characteristic $p > 0$
  because the field extension $R/\fm \hookrightarrow R'/\fm'$ is a finite field
  extension.\smallskip
  \par Next, we prove $(\ref{prop:elshhmsmdepf})$.
  Fix $\fp_\ell \in \Ass_R(R/I)$.
  Since the residue field of $R_{\fp_\ell}$ is infinite, there exists an ideal
  $J_\ell \subseteq I$ with at most $h$ generators such that $J_\ell R_{\fp_\ell}$ is a
  reduction of $IR_{\fp_\ell}$ by \cite[Proposition 8.3.7]{SH06}.
  By \cite[Lemma 8.1.3(1)]{SH06}, we also know that $J_\ell R'_{\fp_\ell}$ is a
  reduction of $I_\ell R'_{\fp_\ell}$, and
  by \cite[Proposition 8.1.5]{SH06}, we know that
  $J^{s_i+h}_\ell R_{\fp_\ell}'$ is a reduction of $I^{s_i+h}R'_{\fp_\ell}$ for all
  $i$.
  \par Since $u \in I^{(M)}$, there exists an element $x_\ell \in R - \fp_\ell$
  such that $x_\ell u \in I^M$.
  We then have
  \begin{align*}
    \bigl(x_\ell^{s_i+h}u^{\frac{s_i+h}{M}}\bigr)^{M} =
    x_\ell^{M(s_i+h)}u^{s_i+h}
    &= (x_\ell^M u)^{s_i+h}
    \in (I^{s_i+h}R')^{M},
  \intertext{and hence}
    x_\ell^{s_i+h}u^{\frac{s_i+h}{M}} &\in
    \overline{I^{s_i+h}R'}.
  \intertext{Since $\overline{J^{s_i+h}_\ell R_{\fp_\ell}'} =
  \overline{I^{s_i+h}R'_{\fp_\ell}}$, there exists an element $y_{\ell,i}
  \in R-
  \fp_\ell$ such that}
  y_{\ell,i} x_\ell^{s_i+h}u^{\frac{s_i+h}{M}} &\in
    \overline{J^{s_i+h}_\ell R'}.
  \end{align*}
  By the Brian\c{c}on--Skoda-type theorem $(\ref{axiom:brianconskoda})$,
  we have
  \[
    y_{\ell,i} x_\ell^{s_i+h}u^{\frac{s_i+h}{M}} \in
    \bigl(J^{s_i+1}_\ell R'\bigr)^\cl \subseteq (I^{s_i+1}R')^\cl,
  \]
  where the inclusion on the right holds by the order-preservation axiom
  $(\ref{axioms:dietz3})$.
  Setting
  \[
    y_\ell = y_{\ell,1}y_{\ell,2}\cdots y_{\ell,n},
  \]
  we therefore have
  \[
    y_\ell x_\ell^{s+h}u^{\frac{s_i+h}{M}} \in (I^{s_i+1}R')^\cl
  \]
  for all $i$.
  Setting $c_\ell = y_\ell x_\ell^{s+h}$, we obtain
  $(\ref{prop:elshhmsmdepf})$.\smallskip
  \par Finally, we show $(\ref{prop:elshhmsmdbcm})$.
  Fix a Dietz closure $\cl$ on $R$ satisfying R.G.'s algebra axiom and the
  Brian\c{c}on--Skoda-type theorem $(\ref{axiom:brianconskoda})$.
  Note that such a closure operation exists in all characteristics by Table
  \ref{tab:dietzclosures}.
  By $(\ref{prop:elshhmsmdepf})$, there exist $c_\ell \in R - \fp_\ell$ such
  that
  \begin{align*}
    c_\ell u^{\frac{s_i+h}{M}} &\in (I^{s_i+1}R')^\cl
    \intertext{for all $i$.
    By R.G.'s result stating that Dietz closures satisfying R.G.'s algebra
    axiom are contained in a big Cohen--Macaulay algebra closure
    (Proposition \ref{prop:rg41}),
    there exists a big Cohen--Macaulay algebra $B$
    over $R'$ such that}
    c_\ell u^{\frac{s_i+h}{M}} &\in I^{s_i+1}B.
  \end{align*}
  This proves the first statement in $(\ref{prop:elshhmsmdbcm})$.
  \par For the second statement in $(\ref{prop:elshhmsmdbcm})$,
  for each $\ell$ such that $\fp_\ell$ is maximal in $\Ass_R(R/I)$ with respect
  to inclusion, choose $d_\ell \in \fp_\ell - \bigl(\bigcup_{j \ne \ell}
  \fp_j \bigr)$, which is possible by prime avoidance.
  As in the proof of Lemma \ref{lem:symbolicdefs}, we have
  \[
    c = \sum_{\ell} \biggl(c_\ell \prod_{j \ne \ell}d_j\biggr) \notin
    \bigcup_\ell \fp_\ell.
  \]
  We also have
  \[
    cu^{\frac{s_i+h}{M}} \in I^{s_i+1}B \subseteq I^{(s_i+1)}B.
  \]
  We now note that $B$ is a big Cohen--Macaulay algebra over $R$, since
  every system of parameters in $R$ maps to a system of parameters in $R'$.
  Since $R$ is regular, $R \to B$ is therefore faithfully flat by
  \cite[Lemma 5.5]{Hoc75queens} (see also \citeleft\citen{HH92}\citemid
  (6.7)\citepunct \citen{HH95}\citemid Lemma 2.1$(d)$\citeright).
  Thus, since $c$ is a nonzerodivisor on $R/I^{(s_i+1)}$ by \cite[Theorem
  6.1$(ii)$]{Mat89}, it is also a
  nonzerodivisor on $B/I^{(s_i+1)}B$.
  We therefore see that for every $i$, we have
  \[
    u^{\frac{s_i+h}{M}} \in I^{(s_i+1)}B.\qedhere
  \]
\end{proof}
\begin{remark}\label{rem:weakalgaxiom}
  The proof of Proposition \ref{prop:elshhmsmd}$(\ref{prop:elshhmsmdbcm})$ in
  fact applies to any closure operation satisfying a Brian\c{c}on--Skoda-type
  theorem $(\ref{axiom:brianconskoda})$ and for which the
  following property holds for all inclusions $Q\subseteq M$ of finitely
  generated $R$-modules:
  \begin{enumerate}[label=$(\star)$,ref=\star]
    \item\label{axiom:weakalg}
      If $u \in Q^\cl_M$, then there exists a big Cohen--Macaulay
      algebra $B$ over $R$ such that $1 \otimes u \in \im(B \otimes_R Q \to
      B \otimes_R M)$.
  \end{enumerate}
  The property $(\ref{axiom:weakalg})$
  holds for all Dietz closures satisfying R.G.'s algebra axiom
  $(\ref{axiom:rgalg})$ by 
  R.G.'s result stating that Dietz closures satisfying R.G.'s algebra
  axiom are contained in a big Cohen--Macaulay algebra closure
  (Proposition \ref{prop:rg41}).
  However, $(\ref{axiom:weakalg})$ holds for the other closure
  operations listed in Table \ref{tab:dietzclosures} as well.
  \begin{enumerate}[label=$(\roman*)$]
    \item The property $(\ref{axiom:weakalg})$ holds for big (and hence also
      small) equational tight closure on all
      Noetherian rings of equal characteristic zero \cite[Theorem 11.4]{Hoc94}.
      See \cite[Definition 3.4.3$(b)$ and (4.6.1)]{HHchar0} for definitions of
      these closure operations.
    \item The property $(\ref{axiom:weakalg})$ holds for $\epf$ closure for
      complete local domains of mixed characteristic with $F$-finite residue
      field.
      This follows from the fact that $\wepf$ closure is a Dietz closure
      satsifying R.G.'s algebra axiom $(\ref{axiom:rgalg})$
      \cite[Theorem 4.8]{Jia21} since $Q_M^\epf \subseteq Q_M^\wepf$.
    \item A stronger version of $(\ref{axiom:weakalg})$ holds for tight
      closure on all analytically irreducible excellent local domains of
      equal characteristic $p > 0$ \cite[Theorem 11.1]{Hoc94}.
  \end{enumerate}
  We can therefore use these closure operations and
  property $(\ref{axiom:weakalg})$ instead of Proposition
  \ref{prop:rg41} to prove Theorems \ref{thm:mainjohnson} and
  \ref{thm:mainjohnsonhhty}
  below.
\end{remark}
Finally, we can prove Theorems \ref{thm:mainjohnson} and
\ref{thm:mainjohnsonhhty} using closure operations.
\begin{proof}[Proof of Theorems \ref{thm:mainjohnson} and
\ref{thm:mainjohnsonhhty} via closure theory]
  By Lemma \ref{lem:hh02reductions}, we may assume that $R$ is a complete regular local
  ring $(R,\fm)$ with a perfect residue field,
  and that the localizations of $R$ at the associated primes of $R/I$ have
  infinite residue fields.\smallskip
  \par We start with Theorem \ref{thm:mainjohnson}.
  Setting $M = s+nh$ in 
  Proposition \ref{prop:elshhmsmd}$(\ref{prop:elshhmsmdbcm})$, for $u \in
  I^{(s+nh)}$, we have
  \[
    u^{\frac{s_i+h}{s+nh}} \in I^{(s_i+1)}B
  \]
  for some big Cohen--Macaulay algebra $B$ over $R$.
  Multiplying together these inclusions for every $i$, we have
  \[
    u \in \prod_{i=1}^n I^{(s_i+1)}B.
  \]
  Since $R$ is regular, $R \to B$ is faithfully flat by
  \cite[Lemma 5.5]{Hoc75queens} (see also \citeleft\citen{HH92}\citemid
  (6.7)\citepunct \citen{HH95}\citemid Lemma 2.1$(d)$\citeright), and hence we have
  \[
    u \in \prod_{i=1}^n I^{(s_i+1)}.\smallskip
  \]
  \par We now prove Theorem \ref{thm:mainjohnsonhhty}.
  Setting $M = s+nh+1$ in 
  Proposition \ref{prop:elshhmsmd}$(\ref{prop:elshhmsmdbcm})$, for $u \in
  I^{(s+nh+1)}$, we have
  \[
    u^{\frac{s_i+h}{s+nh+1}} \in I^{(s_i+1)}B
  \]
  for some big Cohen--Macaulay algebra $B$ over $R$.
  Multiplying together these inclusions for every $i$, we have
  \[
    u^{\frac{s+nh}{s+nh+1}} \in \prod_{i=1}^n I^{(s_i+1)}B.
  \]
  \par Let $l$ be the largest integer for which the composition
  \begin{equation}\label{eq:whatsplits}
    R \longrightarrow B \xrightarrow{u^{\frac{l}{s+nh+1}} \cdot -} B
  \end{equation}
  splits as a map of $R$-modules.
  We claim such an $l$ exists.
  When $l = 0$, the map $R \to B$ is faithfully
  flat, hence pure \cite[Theorem 7.5$(i)$]{Mat89}.
  Thus, the map $R \to B$ splits by an argument of Auslander \cite[p.\
  59]{HH90}.
  \par We now have
  \[
    u^{\frac{l}{s+nh+1}}u =
    u^{\frac{l+1}{s+nh+1}}u^{\frac{s+nh}{s+nh+1}} \in u^{\frac{l+1}{s+nh+1}}
    \cdot \prod_{i=1}^n I^{(s_i+1)}B.
  \]
  Applying $\varphi \in \Hom_R(B,R)$, we have
  \[
    \varphi\Bigl(u^{\frac{l}{s+nh+1}}\Bigr) \cdot u =
    \varphi\Bigl(u^{\frac{l+1}{s+nh+1}}u^{\frac{s+nh}{s+nh+1}}\Bigr)
    \in
    \varphi\Bigl(u^{\frac{l+1}{s+nh+1}}\Bigr) \cdot \prod_{i=1}^n
    I^{(s_i+1)}.
  \]
  Setting $\varphi$ to be a splitting of the map \eqref{eq:whatsplits}, we
  obtain
  \[
    u \in \fm \cdot \prod_{i=1}^n I^{(s_i+1)}
  \]
  as claimed, since $\varphi(u^{\frac{l+1}{s+nh+1}}) \in \fm$ for all
  $\varphi \in \Hom_R(B,R)$ by the assumption on $l$.
\end{proof}

\begingroup
\makeatletter
\renewcommand{\@secnumfont}{\bfseries}
\part{Proof via multiplier/test ideals}\label{part:elshhms}
\makeatother
\endgroup
\section{Big Cohen--Macaulay test ideals, +-test ideals, and multiplier ideals}\label{sect:hls}
In this section, we
review the
theory of big Cohen--Macaulay test ideals developed by Ma and Schwede \cite{MS}
and the theory of $+$-test ideals developed by Hacon, Lamarche, and Schwede
\cite{HLS}.
We also define multiplier ideals for excellent schemes of equal
characteristic zero, following \cite{dFM09,ST}.\medskip
\par When working with big Cohen--Macaulay test ideals throughout Part
\ref{part:elshhms} of this paper,
we will often use the following notation.
\begin{notation}[see {\cite[Setting 6.1]{MS}}]\label{notation:ms21like}
  Let $(R,\fm)$ be a normal complete Noetherian local ring of dimension $d$.
  Fix a dualizing complex $\omega_R^\bullet$ on $R$ with associated canonical
  module $\omega_R$, and fix an embedding $R \subseteq \omega_R \subseteq K(R)$.
  This yields a choice of an effective canonical divisor $K_R$ on $\Spec(R)$.
  \par Let $\Delta$ be an effective $\QQ$-Weil divisor on $\Spec(R)$ such that
  $K_R+\Delta$ is $\QQ$-Cartier.
  Since $K_R+\Delta$ is effective, there is an integer $N > 0$ and an element $h
  \in R$ such that
  \[
    K_R+\Delta = \frac{1}{N} \prdiv_R(h).
  \]
\end{notation}
\subsection{Big Cohen--Macaulay test ideals of divisor pairs}
We define the big Cohen--Macaulay test ideals from \cite{MS}.
See also \cite[Definition 3.1]{PRG} for a related definition, which applies when
$\Delta = 0$ but does not assume that $K_R$ is $\QQ$-Cartier.
\begin{citeddef}[{\cite[Definitions 6.2 and 6.9]{MS}}]\label{def:ms21def}
  With notation as in Notation \ref{notation:ms21like},
  let $B$ be an $R^+$-algebra that is big Cohen--Macaulay over $R$.
  We set
  \[
    0^{B,K_R+\Delta}_{H^d_\fm(R)} \coloneqq \ker\Bigl( H^d_\fm(R)
    \xrightarrow{h^{1/N}} H^d_\fm(B) \Bigr),
  \]
  which does not depend on the choice of $h^{1/N}$.
  The \textsl{big Cohen--Macaulay test ideal of $(R,\Delta)$ with respect to
  $B$} is
  \[
    \tau_B\bigl(R,\Delta\bigr) \coloneqq
    \Ann_{\omega_R}\Bigl(0^{B,K_R+\Delta}_{H^d_\fm(R)}\Bigr).
  \]
  This is an $R$-submodule of $\omega_R$, and is in fact an
  ideal in $R$ by \cite[Lemma 6.8]{MS}.
\end{citeddef}
\subsection{+-test ideals}
Next, we define Hacon, Lamarche, and Schwede's notion of $+$-test
ideals \cite{HLS}.
\subsubsection{$+$-stable global sections}
To define $+$-test ideals,
we first need to define Hacon, Lamarche, and Schwede's version of the space of
$+$-stable global sections first introduced in \cite[Definition 3.10]{TY} and
\cite[Definition 4.2]{BMPSTWW}.
\begin{citeddef}[{\cite[Notation 3.1 and Definition 3.2]{HLS}}]\label{def:b0hls}
  Let $(R,\fm)$ be a complete Noetherian local ring of residue characteristic $p
  > 0$, and consider a proper morphism $\pi\colon X \to \Spec(R)$ from a normal
  integral scheme of dimension $d$.
  Fixing a dualizing complex $\omega_R^\bullet$ on $R$, we have an induced
  dualizing complex $\omega_X^\bullet = \pi^!\omega_R^\bullet$ on $X$ with
  associated canonical sheaf $\omega_X$.
  We can then fix a canonical divisor $K_X$ on $X$, which satisfies $\omega_X =
  \cO_X(K_X)$.
  \par Fix an algebraic closure $\overline{K(X)}$ of the function field $K(X)$
  of $X$, and consider the absolute integral closure $\nu\colon X^+ \to X$.
  For a $\QQ$-Weil divisor $B$ on $X$, we set
  \[
    \nu_*\cO_{X^+}(\nu^*B) \coloneqq \varinjlim_{f \colon Y \to X} f_*
    \cO_Y\bigl(\lfloor f^*B \rfloor\bigr),
  \]
  where the colimit is taken over finite surjective morphisms $f\colon Y \to X$
  as in Definition \ref{def:xplus}.
  This yields a sheaf $\cO_{X^+}(\nu^*B)$.
  \par Now consider a Weil divisor $M$ and an effective $\QQ$-Weil divisor $B$
  on $X$.
  The Matlis dual of the $R$-module
  \[
    \im\biggl(H^d_\fm\Bigl( \RR\Gamma\bigl(X,\cO_X(-M)\bigr)\Bigr)
    \longrightarrow H^d_\fm\Bigl( \RR\Gamma\bigl(X^+,\cO_{X^+}(-\nu^*M + 
    \nu^*B)\bigr)\Bigr)\biggr)
  \]
  yields an $R$-submodule
  \[
    \BB^0\bigl(X,B,\cO_X(K_X+M)\bigr) \subseteq H^0\bigl(X,\cO_X(K_X+M)\bigr)
  \]
  by Lipman's local-global duality \cite[Theorem on p.\ 188]{Lip78} (see
  \cite[Lemma 2.2]{BMPSTWW}).
  This submodule $\BB^0(X,B,\cO_X(K_X+M))$ is
  called the module of \textsl{$+$-stable global sections}.
\end{citeddef}
\subsubsection{$+$-test ideals for divisor pairs}
We can now define $+$-test ideals for pairs
$(X,\Delta)$, where $\Delta$ is a $\QQ$-Weil divisor, following \cite{HLS}.
\begin{citeddef}[{\cite[Definition 4.3, Notation 4.6, Definition 4.14, and
  Definition 4.15]{HLS}}]
  With notation as in Definition \ref{def:b0hls},
  fix an effective $\QQ$-Weil divisor $\Delta$ on $X$ and
  a Cartier divisor $H$ on $X$ such that $K_X+\Delta+H
  \ge 0$.
  The \textsl{$+$-test ideal of the pair $(X,\Delta)$} is the coherent subsheaf
  \[
    \tau_+\bigl(\cO_X,\Delta\bigr) \subseteq \omega_X \otimes \cO_X(H)
  \]
  such that
  $\tau_+(\cO_X,\Delta) \otimes \cO_X(-H) \otimes \sL \subseteq \omega_X \otimes
  \sL$ is globally generated by
  \[
    \BB^0(X,K_X+\Delta+H,\omega_X \otimes \sL) \subseteq H^0(X,\omega_X
    \otimes \sL),
  \]
  where $\sL = \cO_X(L)$ is a sufficiently ample invertible sheaf.
  The subsheaf $\tau_+(\cO_X,\Delta)$ is independent of the choice of
  $\sL$ and $H$ by \cite[Proposition 4.5 and Lemma 4.8$(b)$]{HLS}, and is an
  ideal sheaf by \cite[Lemma 4.18]{HLS}.
\end{citeddef}
Big Cohen--Macaulay test ideals and $+$-test ideals are related in the following
manner.
\begin{remark}\label{rem:hls512}
  With notation as in Notation \ref{notation:ms21like}, assume that $R$ is of
  residue characteristic $p > 0$, and set $X = \Spec(R)$.
  By \cite[Remark 5.12]{HLS}, we have
  \[
    \tau_{\widehat{R^+}}\bigl(R,\Delta\bigr) = \tau_+\bigl(\cO_X,\Delta\bigr),
  \]
  where we recall from Remark
  \ref{rem:bigcmalgexist}$(\ref{rem:bigcmalgexistrplusmixedchar})$
  that $\widehat{R^+}$ is a
  big Cohen--Macaulay $R^+$-algebra.
\end{remark}
This definition also extends to quasi-projective morphisms $U \to \Spec(R)$
under $\QQ$-Cartier assumptions.
\begin{citeddef}[{\cite[Definition 5.10]{HLS}}]\label{def:hlsqproj}
  Let $(R,\fm)$ be a complete Noetherian local ring of residue characteristic $p
  > 0$, and consider a quasi-projective morphism $U \to \Spec(R)$ from a
  normal integral scheme.
  As in Definition \ref{def:b0hls}, by fixing a choice of dualizing complex
  $\omega_R^\bullet$ on $R$, we have an induced canonical sheaf $\omega_U$ on
  $U$, and we can fix a canonical divisor $K_U$ on $U$ that satisfies $\omega_U
  = \cO_U(K_U)$.
  \par Fix an effective $\QQ$-Weil divisor $\Delta_U$ on $U$ such that
  $K_U+\Delta_U$ is $\QQ$-Cartier, and fix a Cartier divisor $G_U$ on $U$ such
  that $K_U+\Delta_U+G_U \ge 0$.
  The \textsl{$+$-test ideal of the pair $(U,\Delta_U)$} is the coherent ideal
  sheaf
  \[
    \tau_+\bigl(\cO_U,\Delta_U\bigr) \coloneqq \tau_+\bigl(\cO_X,\Delta\bigr)
    \bigr\rvert_U
    \subseteq \cO_U,
  \]
  where $\tau_+(\cO_X,\Delta)$ is defined by finding a normal integral
  projective compactification $X \to \Spec(R)$ of $U \to \Spec(R)$, together
  with compactifications $G$ for $G_U$ and $K_X+\Delta+G$ for
  $K_U+\Delta_U+G_U$ such that $G \ge 0$ and $K_X+\Delta+G \ge 0$.
  This definition does not depend on the choice of compactifications by
  \cite[Proposition 5.7]{HLS}.
\end{citeddef}
\subsubsection{$+$-test ideals for ideal pairs}
We also define $+$-test ideals for ideal pairs, following \cite{HLS}.
\begin{citeddef}[{\cite[p.\ 24]{HLS}}]\label{def:plusidealpairs}
  With notation as in Definition \ref{def:b0hls}, suppose that $K_X$ is
  $\QQ$-Cartier.
  Let $\fa \subseteq \cO_X$ be a coherent ideal sheaf, and let $\mu\colon X' \to
  X$ be the normalized blowup of $\fa$ in which case $\mu^{-1}\fa\cdot
  \cO_{X'} = \cO_{X'}(-F)$ for an effective Cartier divisor $F$ on $X'$.
  For every rational number $t \ge 0$, the \textsl{$+$-test ideal of the
  pair $(X,\fa^t)$} is the coherent ideal sheaf
  \[
    \tau_+\bigl(\cO_X,\fa^t\bigr) \subseteq \cO_X
  \]
  such
  that $\tau_+(\cO_X,\fa^t) \otimes \sL \subseteq \sL$ is globally
  generated by
  \begin{align*}
    \MoveEqLeft[5]\BB^0\bigl(X',tF,\cO_{X'}(K_{X'}-\mu^*K_X+L')\bigr)\\
    &\subseteq
    \BB^0\Bigl(X',\{tF\},\cO_{X'}\bigl(K_{X'}-\mu^*K_X-\lfloor tF
    \rfloor+L'\bigr)\Bigr)\\
    &\subseteq H^0\bigl(X',\cO_{X'}(K_{X'} - \mu^*K_X+L')\bigr)\\
    &\subseteq H^0\bigl(X,\cO_X(L)\bigr),
  \end{align*}
  where $\sL = \cO_X(L)$ is a sufficiently ample invertible sheaf and
  $L' = \mu^*L$.
  The ideal sheaf $\tau_+(\cO_X,\fa^t)$ is independent of the choice of $\sL$ by
  \cite[p.\ 24]{HLS}.
\end{citeddef}
We compare the definitions of $+$-test ideals for divisor pairs and for ideal
pairs.
We note that the ideal on the left-hand side of \eqref{rem:hls64cont} below
is related to Robinson's notion of the big Cohen--Macaulay test
ideal for a triple $(R,\Delta,\fa^t)$ \cite{Rob} defined in Definition
\ref{def:ms18like}.
See Lemma \ref{lem:testidealasasum}$(\ref{lem:hlscomparison})$.
\begin{lemma}[{cf.\ \cite[Remark 6.4]{HLS}}]\label{lem:hls64}
  With notation as in Definition \ref{def:b0hls}, if $K_X$ is
  $\QQ$-Cartier, we have
  \begin{align}
    \adjustlimits{\sum^\infty}_{m=1}
    {\sum}_{\Gamma \in \lvert \sL^{\otimes m} \otimes \fa^m \rvert}
    \tau_+\biggl(\cO_X,\frac{t}{m} \Gamma \biggr) &\subseteq
    \tau_+\bigl(\cO_X,\fa^t\bigr).\label{rem:hls64cont}
  \intertext{In particular, if $X = \Spec(R)$, we have}
    \sum_{m=1}^\infty\sum_{g \in \fa^m}
    \tau_+\biggl(\cO_X,\frac{t}{m} \prdiv_R(g) \biggr) &\subseteq
    \tau_+\bigl(\cO_X,\fa^t\bigr).\nonumber
  \end{align}
\end{lemma}
\begin{proof}
  We adapt the proof of \cite[Remark 6.4]{HLS}, which is a version of our lemma
  for test modules $\tau_+(\omega_X,\Delta)$ and $\tau_+(\omega_X,\fa^t)$.
  \par Let $\sL = \cO_X(L)$ be a sufficiently ample invertible sheaf on $X$ such
  that $\fa \otimes \sL$ is globally generated.
  Let $\mu\colon X' \to X$ be the normalized blowup of $\fa$ in which case
  $\mu^{-1}\fa\cdot\cO_{X'} = \cO_{X'}(-F)$ for an effective Cartier divisor $F$
  on $X'$.
  Set $L' = \mu^*L$.
  First, we note that for every integer $m \ge 1$, we have
  $\mu^*\frac{1}{m} \Gamma \ge F$ for all $\Gamma \in \lvert
  \sL^{\otimes m} \otimes \fa^m \rvert$.
  Thus, by \cite[Lemma 3.3$(a)$ and Proposition 3.9]{HLS}, we see that
  $\BB^0(X',tF,\cO_{X'}(K_{X'} - \mu^*K_X +L'))$ contains
  \[
    \BB^0\biggl(X',t \cdot \mu^* \frac{1}{m} \Gamma,\cO_{X'}(K_{X'} - \mu^*K_X +
    L')\biggr) \cong \BB^0\biggl(X,\frac{t}{m} \Gamma,\cO_X(L)\biggr).
  \]
  By the definitions of $\tau_+(\cO_X,\frac{t}{m} \Gamma)$ and of
  $\tau_+(\cO_X,\fa^t)$, we therefore have \eqref{rem:hls64cont}.
\end{proof}
\subsection{Multiplier ideals for excellent schemes of equal
characteristic zero}
Finally, we define multiplier ideals for excellent schemes of equal
characteristic zero, following \citeleft\citen{dFM09}\citemid
\S2\citepunct \citen{ST}\citemid \S2.1\citeright.
\begin{definition}[cf.\ {\citeleft\citen{Laz04b}\citemid Generalization
  9.2.8\citepunct \citen{dFM09}\citemid p.\ 495\citepunct \citen{ST}\citemid
  Definition 2.4\citeright}]
  \label{def:multiplierideal}
  Let $X$ be an excellent normal integral scheme of equal characteristic zero
  with a dualizing complex $\omega_X^\bullet$ and associated choice of
  canonical divisor $K_X$.
  \par Let $\Delta$ be an effective $\QQ$-Weil divisor on $X$ such that
  $K_X+\Delta$ is $\QQ$-Cartier.
  Let $\fa_1,\fa_2,\ldots,\fa_n \subseteq \cO_X$ be coherent ideal sheaves,
  and let $\mu\colon X' \to X$ be a log resolution for the triple
  $(X,\Delta,\fa_1\fa_2\cdots\fa_n)$, which exists by \cite[Theorem 1.1]{Tem08}.
  For each $i$, write $\mu^{-1}\fa_i\cdot\cO_{X'} = \cO_{X'}(-F_i)$ for
  effective Cartier divisors $F_i$ on $X'$.
  For real numbers $t_1,t_2,\ldots,t_n \ge 0$, the \textsl{multiplier ideal
  of the triple $(X,\Delta,\fa_1^{t_1}\fa_2^{t_2} \cdots \fa_n^{t_n})$} is the
  coherent ideal sheaf
  \[
    \cJ\bigl(X,\Delta,\fa_1^{t_1}\fa_2^{t_2} \cdots \fa_n^{t_n}\bigr)
    \coloneqq \mu_*\biggl(\cO_{X'}\biggl(K_{X'}-\biggl\lfloor\mu^*(K_X+\Delta)
    + \sum_{i=1}^n t_iF_i \biggr\rfloor\biggr)\biggr)
    \subseteq \cO_X.
  \]
  This definition does not depend on the choice of $\mu$ by
  \citeleft\citen{dFM09}\citemid Proposition 2.2\citepunct
  \citen{ST}\citemid Remark 2.5$(ii)$\citeright.
\end{definition}

\section{Big Cohen--Macaulay test ideals of triples}
\label{sect:bcmtestfixedgens}
In this section, we define a slight variation of
Robinson's version of test ideals from \cite{Rob}, which combines
aspects of the big Cohen--Macaulay test ideals defined in \cite{MS} and
\cite{PRG} together with aspects of the perfectoid test ideals defined in
\cite{MS18}.
We then prove fundamental results on Robinson's test ideals necessary to prove
Theorems
\ref{thm:mainelshhms},
\ref{thm:mainjohnson}, \ref{thm:mainjohnsonhhty}, and \ref{thm:ty23like}.
Throughout this section, we will use the notation from Notation
\ref{notation:ms21like}.
\subsection{Definition and preliminaries}
We start with slight variations of Robinson's definitions of big Cohen--Macaulay
test ideals for triples from \cite{Rob} (see Remark \ref{rem:difffromrobinson}).
There are two versions of test ideals in \cite{Rob}:
the ideals
$\tau_B(R,\Delta,[\underline{f_1}]^{t_1}[\underline{f_2}]^{t_2} \cdots
[\underline{f_n}]^{t_n})$ that fix sets of generators for ideals, and the ideals
$\tau_B(R,\Delta,\fa_1^{t_1}\fa_2^{t_2} \cdots
\fa_n^{t_n})$ that do not fix sets of generators.
While our main results in this paper can be shown using only the ideals
$\tau_B(R,\Delta,[\underline{f_1}]^{t_1}[\underline{f_2}]^{t_2} \cdots
[\underline{f_n}]^{t_n})$, we will prove many of our results for both versions
of Robinson's test ideals.
We note that it is necessary to consider the ideals
$\tau_B(R,\Delta,[\underline{f_1}]^{t_1}[\underline{f_2}]^{t_2} \cdots
[\underline{f_n}]^{t_n})$ since we do not know whether the subadditivity theorem
(Theorem \ref{thm:subadditivity}) and
the strongest version of our unambiguity
statement (Proposition \ref{prop:notambig}) hold for
$\tau_B(R,\Delta,\fa_1^{t_1}\fa_2^{t_2} \cdots \fa_n^{t_n})$.
\par The definition with a fixed set of generators
combines the definitions in \cite{MS18} and in \cite{MS}.
Compared to \cite{MS18},
this definition works in arbitrary characteristic, and
does not include the factor $p^{1/p^\infty}$ coming from almost
mathematics.
One also ranges over all integers $m_1,m_2,\ldots,m_n$
instead of just $p$-th powers, and there is no perturbation present
in contrast to
\cite[Definition 3.5]{MS18}.
See \cite[Remark 3.7]{Rob}.
\par A similar definition that does not fix a set of generators also appeared in
the work of Sato and Takagi \cite[Definition 3.2$(i)$]{ST}.
See Lemma \ref{lem:testidealasasum}$(\ref{lem:taubasasum})$.
\begin{definition}[see {\cite[Definition 3.6]{Rob}}]
  \label{def:ms18like}
  With notation as in Notation \ref{notation:ms21like}, fix sets
  \[
    \{f_{i,1},f_{i,2},\ldots,f_{i,r_i}\}
  \]
  of elements in $R$ for $i \in \{1,2,\ldots,n\}$.
  For fixed $i$, we denote these data by $[\underline{f_i}]$, and we denote by
  $\fa_i$ the ideal $(f_{i,1},f_{i,2},\ldots,f_{i,r_i})$.
  Fix real numbers $t_1,t_2,\ldots,t_n \ge 0$.
  For an $R^+$-algebra $B$ that is big Cohen--Macaulay over $R$, we set
  \begin{align*}
    0^{B,K_R+\Delta,\fa_1^{t_1}\fa_2^{t_2}\cdots
    \fa_n^{t_n}}_{H^d_\fm(R)}
    &\coloneqq
    \Set*{\eta \in H^d_\fm(R) \given \begin{tabular}{@{}c@{}}
        for all integers $m_1,m_2,\ldots,m_n > 0$, we have\\
        $h^{1/N}g_1^{1/m_1}g_2^{1/m_2}\cdots g_n^{1/m_n} \eta = 0$ in
        $H^d_\fm(B)$\\
        for all $g_i \in \fa_i^{\lceil m_it_i \rceil}$
    \end{tabular}}\\
    0^{B,K_R+\Delta,[\underline{f_1}]^{t_1}[\underline{f_2}]^{t_2}\cdots
    [\underline{f_n}]^{t_n}}_{H^d_\fm(R)}
    &\coloneqq
    \Set*{\eta \in H^d_\fm(R) \given \begin{tabular}{@{}c@{}}
        for all integers $m_1,m_2,\ldots,m_n > 0$, we have\\
        $h^{1/N}g_1g_2\cdots g_n \eta = 0$ in $H^d_\fm(B)$\\
        for all $g_i = \prod_{k=1}^{a_i} f_{i,j_k}^{1/m_i}$ where $a_i \ge m_it_i$
    \end{tabular}}
  \end{align*}
  where $\eta$ in the equations $h^{1/N}g_1^{1/m_1}g_2^{1/m_2}\cdots g_n^{1/m_n}
  \eta = 0$ and $f^{1/N}g_1g_2\cdots g_n \eta =
  0$ is interpreted as the image of $\eta$ under the map $H^d_\fm(R) \to
  H^d_\fm(B)$.
  We then set
  \begin{align*}
    \tau_B\bigl(R,\Delta,\fa_1^{t_1}\fa_2^{t_2} \cdots
    \fa_n^{t_n}\bigr) &\coloneqq \Ann_{\omega_R}\biggl(
      0^{B,K_R+\Delta,\fa_1^{t_1}\fa_2^{t_2} \cdots
    \fa_n^{t_n}}_{H^d_\fm(R)} \biggr)\\
    \tau_B\bigl(R,\Delta,[\underline{f_1}]^{t_1}[\underline{f_2}]^{t_2} \cdots
    [\underline{f_n}]^{t_n}\bigr) &\coloneqq \Ann_{\omega_R}\biggl(
      0^{B,K_R+\Delta,[\underline{f_1}]^{t_1}[\underline{f_2}]^{t_2}\cdots
    [\underline{f_n}]^{t_n}}_{H^d_\fm(R)} \biggr)
  \end{align*}
  which are ideals in $R$ since they are $R$-submodules of
  $\tau_B(R,\Delta)$.
  For both $0^B$ and $\tau_B$, we omit the divisor $\Delta$ (resp.\ $\fa_i$ or
  $[\underline{f_i}]$) from our notation if $\Delta = 0$ (resp.\ if $n=0$).
\end{definition}
\begin{remark}\label{rem:difffromrobinson}
  In addition to allowing for multiple tuples of generators $[\underline{f_i}]$
  or ideals $\fa_i$,
  Definition \ref{def:ms18like} differs slightly from Robinson's definition in
  \cite[Definition 3.6]{Rob} in that Robinson ranges over all sufficiently large
  $m_i$ instead of all $m_i$.
  However, the ideal defined in
  Definition \ref{def:ms18like} and Robinson's original definition from
  \cite[Definition 3.6]{Rob} coincide
  by Lemma \ref{lem:ms18lem36} below.
\end{remark}
\begin{remark}
  Robinson's definitions in Definition \ref{def:ms18like} allow
  the exponents $t_1,t_2,\ldots,t_n$ to be real numbers.
  Some of our results below (Lemma
  \ref{lem:testidealasasum}$(\ref{lem:hlscomparison})$, Proposition
  \ref{prop:notambig}, and Proposition \ref{prop:notambignogens})
  require rational exponents $t_1,t_2,\ldots,t_n$, at least in their strongest
  forms.
  In \cite{MS18}, Ma and Schwede are able to prove the analogues of these
  results for their perfectoid test ideals with real exponents
  $t_1,t_2,\ldots,t_n$ because their definitions include
  perturbations.
  According to \cite[Remark 3.7]{Rob}, it is expected that these perturbations
  are unnecessary for sufficiently large choices of the big Cohen--Macaulay
  algebra $B$.
\end{remark}
The two definitions in Definition \ref{def:ms18like} are related in the
following manner.
This comparison appears in \cite[Proposition
3.3$(a)$]{MS18} for perfectoid test ideals when $\Delta = 0$ and in
\cite[Proposition 3.8]{Rob} for big Cohen--Macaulay test ideals of triples
when $n = 1$.
\begin{proposition}[cf.\ {\citeleft\citen{MS18}\citemid Proposition
  3.3$(a)$\citepunct \citen{Rob}\citemid Proposition 3.8\citeright}]
  \label{prop:ms18comparison}
  Fix notation as in Notation \ref{notation:ms21like} and Definition
  \ref{def:ms18like}.
  We have
  \[
    \tau_B\bigl(R,\Delta,[\underline{f_1}]^{t_1}[\underline{f_2}]^{t_2}
    \cdots [\underline{f_n}]^{t_n}\bigr) \subseteq
    \tau_B\bigl(R,\Delta,\fa_1^{t_1}\fa_2^{t_2}\cdots \fa_n^{t_n}\bigr).
  \]
\end{proposition}
\begin{proof}
  It suffices to show that
  \[
    0^{B,K_R+\Delta,[\underline{f_1}]^{t_1}[\underline{f_2}]^{t_2}\cdots
    [\underline{f_n}]^{t_n}}_{H^d_\fm(R)} \supseteq
    0^{B,K_R+\Delta,\fa_1^{t_1}\fa_2^{t_2}\cdots
    \fa_n^{t_n}}_{H^d_\fm(R)}.
  \]
  Let $\eta$ be an element in the module on the right-hand side, and consider
  elements
  \[
    g_i = \prod_{k=1}^{a_i} f_{i,j_k}^{1/m_i} =
    \biggl(\prod_{k=1}^{a_i} f_{i,j_k}\biggr)^{1/m_i}
  \]
  for each $i$ where $a_i \ge m_it_i$.
  Since each $f_{i,j_k}$ is a generator of $\fa_i$, we have
  \[
    \prod_{k=1}^{a_i} f_{i,j_k} \in \fa_i^{a_i} \subseteq \fa_i^{\lceil m_it_i
    \rceil}.
  \]
  Thus, each $g_i$ can be written as an $m_i$-th root of an element in
  $\fa_i^{\lceil m_it_i \rceil}$, and $h^{1/N}g_1g_2\cdots g_n\eta = 0$.
\end{proof}
We also note that in Definition \ref{def:ms18like}, we may restrict to $m_i$
sufficiently large and divisible.
\begin{lemma}[cf.\ {\cite[Lemma 3.6]{MS18}}]\label{lem:ms18lem36}
  Fix notation as in Notation \ref{notation:ms21like} and Definition
  \ref{def:ms18like}.
  In the definitions of
  $0^{B,K_R+\Delta,\fa_1^{t_1}\fa_2^{t_2}\cdots \fa_n^{t_n}}_{H^d_\fm(R)}$ and
  $0^{B,K_R+\Delta,[\underline{f_1}]^{t_1}[\underline{f_2}]^{t_2}\cdots
  [\underline{f_n}]^{t_n}}_{H^d_\fm(R)}$,
  we may restrict to $m_1,m_2,\ldots,m_n > 0$ sufficiently large and
  divisible.
\end{lemma}
\begin{proof}
  Fix integers $m_{i,0} > 0$ for each $i$.
  We want to show that
  \begin{align*}
    0^{B,K_R+\Delta,\fa_1^{t_1}\fa_2^{t_2}\cdots \fa_n^{t_n}}_{H^d_\fm(R)}
    &=
    \Set*{\eta \in H^d_\fm(R) \given \begin{tabular}{@{}c@{}}
        for all integers $m_i$ divisible by $m_{i,0}$, we have\\
        $h^{1/N}g_1^{1/m_1}g_2^{1/m_2}\cdots g_n^{1/m_n} \eta = 0$ in
        $H^d_\fm(B)$\\
        for all $g_i \in \fa_i^{\lceil m_it_i \rceil}$
    \end{tabular}},\\
    0^{B,K_R+\Delta,[\underline{f_1}]^{t_1}[\underline{f_2}]^{t_2}\cdots
    [\underline{f_n}]^{t_n}}_{H^d_\fm(R)}
    &=
    \Set*{\eta \in H^d_\fm(R) \given \begin{tabular}{@{}c@{}}
        for all integers $m_i$ divisible by $m_{i,0}$, we have\\
        $h^{1/N}g_1g_2\cdots g_n \eta = 0$ in $H^d_\fm(B)$\\
        for all $g_i = \prod_{k=1}^{a_i} f_{i,j_k}^{1/m_i}$ where $a_i \ge m_it_i$
    \end{tabular}}.
  \end{align*}
  The inclusion $\subseteq$ holds in both cases
  since restricting to $m_i$ divisible by
  $m_{i,0}$ results in fewer conditions.
  \par For the inclusion $\supseteq$, we first consider the case for the
  $\fa_i$.
  Let $\eta$ be an element in the module on the
  right-hand side, and consider elements $g_i \in \fa_i^{\lceil m_it_i \rceil}$
  for each $i$.
  We then have
  \[
    g_i^{m_{i,0}} \in \fa_i^{m_{i,0}\lceil m_it_i \rceil} \subseteq
    \fa_i^{\lceil m_{i,0}m_it_i \rceil}
  \]
  for each $i$, and hence
  \[
    h^{1/N}g_1^{1/m_1}g_2^{1/m_2}\cdots g_n^{1/m_n} \eta =
    h^{1/N}(g_1^{m_{1,0}})^{1/(m_{1,0}m_1)}a
    (g_2^{m_{2,0}})^{1/(m_{2,0}m_2)} \cdots
    (g_n^{m_{n,0}})^{1/(m_{n,0}m_n)}
    \eta = 0.
  \]
  \par It remains to show the inclusion $\supseteq$ for the
  $[\underline{f_i}]$.
  Let $\eta$ be an element in the module on the
  right-hand side, and consider elements
  $g_i = \prod_{k=1}^{a_i} f_{i,j_k}^{1/m_i}$ for each $i$ where $a_i
  \ge m_it_i$.
  We can then write
  \[
    g_i = \prod_{k=1}^{a_i} \bigl(f_{i,j_k}^{m_{i,0}}\bigr)^{1/(m_{i,0}m_i)} =
    \prod_{k'=1}^{m_{i,0}a_i} f_{i,j_{k'}}^{1/(m_{i,0}m_i)}
  \]
  for some sequence $(j_{k'})_{k'=1}^{m_{i,0}a_i}$,
  and hence $h^{1/N}g_1g_2\cdots g_n \eta = 0$.
\end{proof}
We connect Definition \ref{def:ms18like}
to the big Cohen--Macaulay test ideals reviewed
in \S\ref{sect:hls}.
In $(\ref{lem:taubasasum})$ below, \eqref{eq:taubasasumnogens} connects
Robinson's definition (Definition \ref{def:ms18like}) without a fixed choice of
generators to Sato and Takagi's definition in \cite[Definition 3.2$(i)$]{ST}.
A version of \eqref{eq:taubasasum} when $n = 1$
appears in the proof of \cite[Theorem 3.11]{Rob}.
\begin{lemma}\label{lem:testidealasasum}
  Fix notation as in Notation \ref{notation:ms21like} and Definition
  \ref{def:ms18like}.
  \begin{enumerate}[label=$(\roman*)$,ref=\roman*]
    \item\label{lem:closureasintersect}
      We have
      \begin{align*}
        0^{B,K_R+\Delta,\fa_1^{t_1}\fa_2^{t_2}\cdots \fa_n^{t_n}}_{H^d_\fm(R)}
        &= \bigcap_{m_1,m_2,\ldots,m_n \in \ZZ_{>0}}
        \bigcap_{g_i \in \fa_i^{\lceil m_it_i \rceil}} \ker\biggl( H^d_\fm(R)
        \xrightarrow{h^{1/N}g_1^{1/m_1}g_2^{1/m_2}\cdots g_n^{1/m_n}}
        H^d_\fm(B) \biggr),\\
        0^{B,K_R+\Delta,[\underline{f_1}]^{t_1}[\underline{f_2}]^{t_2}\cdots
        [\underline{f_n}]^{t_n}}_{H^d_\fm(R)}
        &= \bigcap_{m_1,m_2,\ldots,m_n \in \ZZ_{>0}}
        \bigcap_{\substack{g_i = \prod_{k=1}^{a_i} f_{i,j_k}^{1/m_i}
        \\\text{where}\ a_i \ge m_it_i}} \ker\biggl( H^d_\fm(R)
        \xrightarrow{h^{1/N}g_1g_2\cdots g_n} H^d_\fm(B) \biggr).
      \end{align*}
      Moreover, on the right-hand sides
      we may restrict to $m_1,m_2,\ldots,m_n > 0$
      sufficiently large and divisible.
    \item\label{lem:taubasasum}
      We have
      \begin{align}
        \tau_B\bigl(R,\Delta,\fa_1^{t_1}\fa_2^{t_2}\cdots \fa_n^{t_n}\bigr) &=
        \sum_{m_1,m_2,\ldots,m_n \in
        \ZZ_{>0}} \sum_{g_i \in \fa_i^{\lceil m_it_i \rceil}}
        \tau_B\biggl(R,\Delta+\sum_{i=1}^n \frac{1}{m_i}\prdiv_R(g_i) \biggr),
        \label{eq:taubasasumnogens}\\
        \tau_B\bigl(R,\Delta,[\underline{f_1}]^{t_1}[\underline{f_2}]^{t_2}
        \cdots [\underline{f_n}]^{t_n}\bigr) &=
        \sum_{m_1,m_2,\ldots,m_n \in
        \ZZ_{>0}} \sum_{\substack{g_i = \prod_{k=1}^{a_i} f_{i,j_k}^{1/m_i}
        \\\text{where}\ a_i \ge m_it_i}}
        \tau_B\biggl(R,\Delta+\sum_{i=1}^n \prdiv_R(g_i) \biggr),
        \label{eq:taubasasum}
      \end{align}
      where $\prdiv_R(g_i) \coloneqq \frac{1}{m_i} \sum_{k=1}^{a_i}
      \prdiv_R(f_{i,j_k})$.
      Moreover, on the right-hand sides
      we may restrict to $m_1,m_2,\ldots,m_n > 0$
      sufficiently large and divisible.
    \item\label{lem:hlscomparison}
      If the $t_i$ are rational numbers, then we have
      \begin{equation}
        \tau_B\bigl(R,\Delta,\fa_1^{t_1}\fa_2^{t_2}\cdots \fa_n^{t_n}\bigr)
        =
        \sum_{m_1,m_2,\ldots,m_n \in \ZZ_{>0}} \sum_{g_i \in \fa_i^{m_i}}
        \tau_B\biggl(R,\Delta+\sum_{i=1}^n
        \frac{t_i}{m_i} \prdiv_R(g_i)\biggr).
        \label{eq:hlscomparisonrational}
      \end{equation}
      If, moreover, $r_i = 1$ for all $i$, then we have
      \[
        \tau_B\bigl(R,\Delta,[\underline{f_1}]^{t_1}[\underline{f_2}]^{t_2}
        \cdots [\underline{f_n}]^{t_n}\bigr)
        = \tau_B\bigl(R,\Delta,\fa_1^{t_1}\fa_2^{t_2}\cdots \fa_n^{t_n}\bigr)
        = \tau_B\biggl(R,\Delta+\sum_{i=1}^n t_i \prdiv_R(f_{i,1})\biggr).
      \]
  \end{enumerate}
\end{lemma}
\begin{proof}
  $(\ref{lem:closureasintersect})$ follows from Definition \ref{def:ms18like}.
  The last statement about taking $m_i > 0$ sufficiently large and divisible
  follows from Lemma \ref{lem:ms18lem36}.\smallskip
  \par For $(\ref{lem:taubasasum})$, we note that the
  sums on the right-hand side of \eqref{eq:taubasasumnogens} and
  \eqref{eq:taubasasum}
  are finite sums since $R$ is Noetherian,
  and that the intersection in $(\ref{lem:closureasintersect})$ is a finite
  intersection since $H^d_\fm(R)$ is Artinian \citeleft\citen{Gro67}\citemid
  Proposition 6.4(2)\citepunct \citen{Har68}\citemid
  \S2, Example 1\citepunct \citen{Har70}\citemid Proposition 1.1\citeright.
  Applying $\Ann_{\omega_R}(-)$ to both sides of
  $(\ref{lem:closureasintersect})$, we
  therefore obtain \eqref{eq:taubasasum}.
  The penultimate statement about taking $m_i > 0$ sufficiently large and
  divisible follows from Lemma \ref{lem:ms18lem36} by taking $m_i > 0$
  sufficiently large and divisible in
  $(\ref{lem:closureasintersect})$.\smallskip
  \par For $(\ref{lem:hlscomparison})$, we first show the inclusion $\subseteq$
  in \eqref{eq:hlscomparisonrational}.
  By $(\ref{lem:taubasasum})$, it suffices to show that for $m_i$ sufficiently
  large and divisible, the ideal
  \begin{equation}\label{eq:doesthisappear}
    \tau_B\biggl(R,\Delta+\sum_{i=1}^n \frac{1}{m_i}\prdiv_R(g_i) \biggr)
  \end{equation}
  for $g_i \in \fa_i^{\lceil m_it_i \rceil}$
  appears on the right-hand side of \eqref{eq:hlscomparisonrational}.
  Since $m_i$ is sufficiently large and divisible, we have $m_it_i \in \ZZ$, and
  hence we can write $m_it_i = s_i$ for some integer $s_i \ge 0$.
  We then have
  \[
    \frac{1}{m_i} \prdiv_R(g_i) =
    \frac{t_i}{m_it_i} \prdiv_R(g_i) = \frac{t_i}{s_i} \prdiv_R(g_i),
  \]
  where $g_i \in \fa_i^{\lceil m_it_i \rceil} = \fa_i^{s_i}$.
  Thus, the ideal \eqref{eq:doesthisappear} appears on the right-hand side of
  \eqref{eq:hlscomparisonrational}.
  \par We now show the inclusion $\supseteq$ in
  \eqref{eq:hlscomparisonrational}.
  It suffices to show that the ideal
  \begin{equation}\label{eq:doesthisappear2}
    \tau_B\biggl(R,\Delta+\sum_{i=1}^n \frac{t_i}{m_i} \prdiv_R(g_i)\biggr)
  \end{equation}
  for $g_i = \prod_{k=1}^{a_i} f_{i,j_k}^{1/m_i}$ where $a_i \ge m_it_i$ appears
  on the right-hand side of \eqref{eq:taubasasumnogens}.
  Write $t_i = p_i/q_i$ for integers $p_i,q_i \ge 0$.
  We then have
  \begin{align*}
    \tau_B\biggl(R,\Delta+\sum_{i=1}^n \frac{t_i}{m_i} \prdiv_R(g_i)\biggr)
    &= \tau_B\biggl(R,\Delta+\sum_{i=1}^n \frac{p_i}{m_iq_i}
    \prdiv_R(g_i)\biggr)\\
    &= \tau_B\biggl(R,\Delta+\sum_{i=1}^n \frac{1}{m_iq_i}
    \prdiv_R(g_i^{p_i})\biggr).
  \end{align*}
  Since
  $g_i^{p_i} \in \fa_i^{m_ip_i} =
  \fa_i^{m_iq_it_i} = \fa_i^{\lceil m_iq_it_i \rceil}$,
  we see that the ideal \eqref{eq:doesthisappear2}
  appears as an ideal on the right-hand side of
  \eqref{eq:hlscomparisonrational}.
  \par The last statement in $(\ref{lem:hlscomparison})$
  follows by looking at the monomials appearing in the
  sum on the right-hand sides of \eqref{eq:taubasasum} and
  \eqref{eq:hlscomparisonrational}.
\end{proof}
We use Lemma \ref{lem:testidealasasum} to show the following result.
\begin{proposition}[cf.\ {\citeleft\citen{MS18}\citemid Proposition
  3.9\citepunct \citen{MS}\citemid Lemma 6.6\citeright}]
  \label{prop:taubmultiplied}
  Fix notation as in Notation \ref{notation:ms21like} and Definition
  \ref{def:ms18like}.
  Fix another set of elements $\{f_1,f_2,\ldots,f_r\}$, which we abbreviate by
  $[\underline{f}]$, and denote by $\fa$ the ideal $(f_1,f_2,\ldots,f_r)$.
  We have
  \begin{align*}
    \tau_B\bigl(R,\Delta,\fa_1^{t_1}\fa_2^{t_2}\cdots \fa_n^{t_n}\bigr) \cdot
    (f_{1},f_{2},\ldots,f_{r})
    &\subseteq
    \tau_B\bigl(R,\Delta,\fa_1^{t_1}\fa_2^{t_2}\cdots \fa_n^{t_n}
    \fa^1\bigr),\\
    \tau_B\bigl(R,\Delta,[\underline{f_1}]^{t_1}[\underline{f_2}]^{t_2}\cdots
    [\underline{f_{n}}]^{t_{n}}\bigr) \cdot
    (f_{1},f_{2},\ldots,f_{r})
    &\subseteq
    \tau_B\bigl(R,\Delta,[\underline{f_1}]^{t_1}[\underline{f_2}]^{t_2}\cdots
    [\underline{f_{n}}]^{t_{n}}[\underline{f}]^1\bigr).
  \end{align*}
  In particular, if $(R,\Delta)$ is big Cohen--Macaulay-regular with respect to
  $B$ (which holds for example when $R$ is regular and $\Delta = 0$), then
  \[
    (f_{1},f_{2},\ldots,f_{r}) \subseteq
    \tau_B\bigl(R,\Delta,[\underline{f}]^1\bigr) \subseteq
    \tau_B\bigl(R,\Delta,\fa^1\bigr).
  \]
\end{proposition}
We recall from \cite[Definition 8.9]{MS} that a pair $(R,\Delta)$ as in Notation
\ref{notation:ms21like} is \textsl{big Cohen--Macaulay-regular} with respect to
a big Cohen--Macaulay $R^+$-algebra $B$ if $\tau_B(R,\Delta) = R$.
\begin{proof}
  By Lemma \ref{lem:testidealasasum}$(\ref{lem:taubasasum})$, we have
  \begin{align*}
    \MoveEqLeft[5]
    \tau_B\bigl(R,\Delta,\fa_1^{t_1}\fa_2^{t_2}\cdots \fa_n^{t_n}\bigr) \cdot
    (f_{1},f_{2},\ldots,f_{r})\\
    &= \sum_{j=1}^{r}
    \tau_B\bigl(R,\Delta,\fa_1^{t_1}\fa_2^{t_2}\cdots \fa_n^{t_n}\bigr) \cdot
    f_j\\
    &= \sum_{j=1}^{r} 
    \sum_{m_1,m_2,\ldots,m_n \in
    \ZZ_{>0}} \sum_{g_i \in \fa_i^{\lceil m_it_i \rceil}}
    \tau_B\biggl(R,\Delta+\sum_{i=1}^n \frac{1}{m_i}
    \prdiv_R(g_i) \biggr) \cdot f_j\\
    &= \sum_{j=1}^{r} 
    \sum_{m_1,m_2,\ldots,m_n \in
    \ZZ_{>0}} \sum_{g_i \in \fa_i^{\lceil m_it_i \rceil}}
    \tau_B\biggl(R,\Delta+\sum_{i=1}^n \frac{1}{m_i}
    \prdiv_R(g_i) + \prdiv_R(f_j)\biggr)\displaybreak[1]\\
    \MoveEqLeft[5]
    \tau_B\bigl(R,\Delta,[\underline{f_1}]^{t_1}[\underline{f_2}]^{t_2}\cdots
    [\underline{f_{n}}]^{t_{n}}\bigr) \cdot
    (f_{1},f_{2},\ldots,f_{r})\\
    &= \sum_{j=1}^{r}
    \tau_B\bigl(R,\Delta,[\underline{f_1}]^{t_1}[\underline{f_2}]^{t_2}\cdots
    [\underline{f_{n}}]^{t_{n}}\bigr) \cdot f_j\\
    &= \sum_{j=1}^{r} 
    \sum_{m_1,m_2,\ldots,m_n \in
    \ZZ_{>0}} \sum_{\substack{g_i = \prod_{k=1}^{a_i} f_{i,j_k}^{1/m_i}
    \\\text{where}\ a_i \ge m_it_i}}
    \tau_B\biggl(R,\Delta+\sum_{i=1}^n \prdiv_R(g_i) \biggr) \cdot f_j\\
    &= \sum_{j=1}^{r} 
    \sum_{m_1,m_2,\ldots,m_n \in
    \ZZ_{>0}} \sum_{\substack{g_i = \prod_{k=1}^{a_i} f_{i,j_k}^{1/m_i}
    \\\text{where}\ a_i \ge m_it_i}}
    \tau_B\biggl(R,\Delta+\sum_{i=1}^n \prdiv_R(g_i) + \prdiv_R(f_j) \biggr)
  \end{align*}
  where the last equality in each case holds by
  the Skoda-type result in \cite[Lemma 6.6]{MS}.
  Applying Lemma \ref{lem:testidealasasum}$(\ref{lem:taubasasum})$ again, we see
  that these ideals are contained in
  $\tau_B(R,\Delta,\fa_1^{t_1}\fa_2^{t_2}\cdots \fa_n^{t_n}
  \fa^1)$ and
  $\tau_B(R,\Delta,[\underline{f_1}]^{t_1}[\underline{f_2}]^{t_2}\cdots
  [\underline{f_{n}}]^{t_{n}}[\underline{f}]^1)$, respectively.
  \par The assertion for big Cohen--Macaulay-regular pairs $(R,\Delta)$ now
  follows because $\tau_B(R,\Delta) = R$ by definition in \cite[Definition
  8.9]{MS}.
  Pairs $(R,\Delta)$ where $R$ is regular and $\Delta = 0$
  are big
  Cohen--Macaulay-regular with respect to all big Cohen--Macaulay $R^+$-algebras
  $B$ since
  $R \to B$ is faithfully flat
  \cite[Lemma 5.5]{Hoc75queens} (see also \citeleft\citen{HH92}\citemid
  (6.7)\citepunct \citen{HH95}\citemid Lemma 2.1$(d)$\citeright), and hence
  $H^d_\fm(R) \to H^d_\fm(B)$ is injective.
\end{proof}
\subsection{Unambiguity of exponents}
We next prove the following unambiguity-type statement for exponents for the
test ideals $\tau_B(R,\Delta,[\underline{f_1}]^{t_1}[\underline{f_2}]^{t_2} 
\cdots [\underline{f_n}]^{t_n})$, which is
stronger than that proved for the perfectoid test ideals in \cite{MS18}.
This statement does not need the perturbations present in \cite[Definition 3.5
and Proposition 3.8]{MS18} because we do not have the extra factor
$p^{1/p^\infty}$ coming from almost mathematics in \cite[Definitions 3.1 and
4.1]{MS18}, and because
we can clear the denominators in $s_1,s_2$ after passing to sufficiently large
and divisible $m_1,m_2$ using Lemma \ref{lem:ms18lem36}, at least when $s_1,s_2$
are rational.
\begin{proposition}[{cf.\ \cite[Proposition 3.8]{MS18}}]
  \label{prop:notambig}
  Fix notation as in Notation \ref{notation:ms21like} and Definition
  \ref{def:ms18like}.
  Fix another set of elements $\{f_1,f_2,\ldots,f_r\}$, which we abbreviate by
  $[\underline{f}]$, and denote by $\fa$ the ideal $(f_1,f_2,\ldots,f_r)$.
  For all real numbers $s_1,s_2 \ge 0$, we have
  \begin{align}
    \tau_B\bigl(R,\Delta,\fa_1^{t_1}\fa_2^{t_2}\cdots \fa_n^{t_n}
    \fa^{s_1}\fa^{s_2}\bigr)
    &\subseteq
    \tau_B\bigl(R,\Delta,\fa_1^{t_1}\fa_2^{t_2}\cdots \fa_n^{t_n}
    \fa^{s_1+s_2}\bigr)\nonumber\\
    \tau_B\bigl(R,\Delta,[\underline{f_1}]^{t_1}[\underline{f_2}]^{t_2} \cdots
    [\underline{f_n}]^{t_n}[\underline{f}]^{s_1}[\underline{f}]^{s_2}\bigr)
    &\subseteq
    \tau_B\bigl(R,\Delta,[\underline{f_1}]^{t_1}[\underline{f_2}]^{t_2} \cdots
    [\underline{f_n}]^{t_n}[\underline{f}]^{s_1+s_2}\bigr)
    \label{eq:notambiggens}
  \end{align}
  with equality in \eqref{eq:notambiggens} if $s_1$ and $s_2$ are rational.
\end{proposition}
\begin{proof}
  Fix $m_1,m_2,\ldots,m_{n}$ and $g_1,g_2,\ldots,g_{n}$
  as in Lemma \ref{lem:testidealasasum}$(\ref{lem:taubasasum})$,
  and set
  \[
    \Delta' = \sum_{i=1}^{n} \prdiv_R(g_i).
  \]
  By Lemma \ref{lem:testidealasasum}$(\ref{lem:taubasasum})$,
  it suffices to show that
  \begin{align}
    \tau_B\bigl(R,\Delta+\Delta',\fa^{s_{1}}
    \fa^{s_{2}}\bigr)
    &\subseteq
    \tau_B\bigl(R,\Delta+\Delta',\fa^{s_1+s_2}\bigr)\nonumber\\
    \tau_B\bigl(R,\Delta+\Delta',[\underline{f}]^{s_{1}}
    [\underline{f}]^{s_{2}}\bigr)
    &\subseteq
    \tau_B\bigl(R,\Delta+\Delta',[\underline{f}]^{s_1+s_2}\bigr)\nonumber
  \intertext{with equality if $s_1$ and $s_2$ are rational.
  Replacing $\Delta$ with $\Delta+\Delta'$,
  it therefore suffices to prove the case when there are no
  $[\underline{f_i}]$.
  By Definition \ref{def:ms18like}, it moreover suffices to show that}
    0_{H^d_\fm(R)}^{B,K_R+\Delta,\fa^{s_{1}} \fa^{s_{2}}}
    &\supseteq 0_{H^d_\fm(R)}^{B,K_R+\Delta,\fa^{s_1+s_2}}
    \label{eq:notambigclosuresnogens}
    \\
    0_{H^d_\fm(R)}^{B,K_R+\Delta,[\underline{f}]^{s_{1}}
    [\underline{f}]^{s_{2}}}
    &\supseteq 0_{H^d_\fm(R)}^{B,K_R+\Delta,[\underline{f}]^{s_1+s_2}}
    \label{eq:notambigclosures}
  \end{align}
  with equality in \eqref{eq:notambigclosures} if $s_1$ and $s_2$ are rational.
  \par For the inclusion $\supseteq$ in \eqref{eq:notambigclosuresnogens},
  suppose that $\eta \in 0_{H^d_\fm(R)}^{B,K_R+\Delta,\fa^{s_1+s_2}}$.
  It suffices to show that for all $m_1,m_2 > 0$, we have
  $h^{1/N}g_1^{1/m_1}g_2^{1/m_2}\eta = 0$ in $H^d_\fm(B)$ for all choices of
  elements $g_i \in \fa_i^{\lceil m_is_i \rceil}$.
  Writing
  \[
    g_1^{1/m_1}g_2^{1/m_2} = g_1^{m_2/(m_1m_2)}g_2^{m_1/(m_1m_2)}
    = (g_1^{m_2}g_2^{m_1})^{1/(m_1m_2)},
  \]
  we see that $g_1^{1/m_1}g_2^{1/m_2}$ can be written as the $(m_1m_2)$-th root
  of the element
  \[
    g = g_1^{m_2}g_2^{m_1} \in
    \fa^{m_2\lceil m_1s_1 \rceil+m_1\lceil m_2s_2 \rceil}
    \subseteq \fa^{\lceil m_1m_2s_1 \rceil+\lceil m_1m_2s_2 \rceil} \subseteq
    \fa^{\lceil m_1m_2(s_1 + s_2) \rceil}.
  \]
  We therefore see that $h^{1/N}g_1^{1/m_1}g_2^{1/m_2}\eta =
  h^{1/N}g^{1/(m_1m_2)} = 0$.
  \par For the inclusion $\supseteq$ in \eqref{eq:notambigclosures},
  suppose that $\eta \in
  0_{H^d_\fm(R)}^{B,K_R+\Delta,[\underline{f}]^{s_1+s_2}}$.
  It suffices to show that for all $m_1,m_2 > 0$,
  we have $f^{1/N}g_1g_2\eta = 0$ in $H^d_\fm(B)$ for all choices of
  elements $g_i = \prod_{k_i=1}^{a_i} f_{i,j_k}^{1/m_i}$ where $a_i \ge
  m_is_i$.
  Writing
  \[
    g_1g_2 = \biggl(\prod_{k_1=1}^{a_1}
    \bigl(f_{j_{k_1}}^{m_2}\bigr)^{1/(m_1m_2)}
    \biggr)\biggl( \prod_{k_2=1}^{a_2} \bigl(f_{j_{k_2}}^{m_1}
    \bigr)^{1/(m_1m_2)}
    \biggr),
  \]
  we see that $g_1g_2$ can be written as a product of $a_1m_2+a_2m_1$ choices of
  $(m_1m_2)$-th roots of elements corresponding to $[\underline{f}]$, where
  \[
    a_1m_2+a_2m_1 \ge m_1s_1m_2 + m_2s_2m_1 = m_1m_2(s_1+s_2).
  \]
  Writing this product as
  \[
    g = g_1g_2 = \prod_{k=1}^{a_1m_2+a_2m_1} f_{j_k}^{1/(m_1m_2)},
  \]
  we see that $h^{1/N}g_1g_2\eta = h^{1/N}g\eta = 0$.
  \par We now show the inclusion $\subseteq$ in \eqref{eq:notambigclosures}
  assuming that $s_1$ and $s_2$ are rational.
  Suppose that $\eta \in
  0_{H^d_\fm(R)}^{B,K_R+\Delta,[\underline{f_{1}}]^{s_{1}}
  [\underline{f_{2}}]^{s_{2}}}$.
  By Lemma \ref{lem:ms18lem36},
  it suffices to show that for all sufficiently large and divisible
  $m > 0$, we have
  $h^{1/N}g \eta = 0$
  in $H^d_\fm(B)$ for all choices of elements $g \in R$ which are expressed as a
  product of $a$ choices of $m$-th roots of
  elements corresponding to $[\underline{f}]$, where $a \ge m(s_1+s_2)$.
  By ranging over $m$ divisible enough, we may assume that $ms_1$ and $ms_2$ are
  integers, in which case we may write
  $a = a_1+a_2$ where $a_1 \ge ms_1$ and
  $a_2 \ge ms_2$.
  Write
  \begin{align*}
    g = \biggl( \prod_{k_1=1}^{a_1} f_{j_{k_1}}^{1/m} \biggr) {}&{}
    \biggl( \prod_{k_2=1}^{a_2} f_{j_{k_2}}^{1/m} \biggr).
  \intertext{Thus, we have}
    h^{1/N}g \eta = h^{1/N}
    \biggl( \prod_{k_1=1}^{a_1} f_{j_{k_1}}^{1/m} \biggr) {}&{}
    \biggl( \prod_{k_2=1}^{a_2} f_{j_{k_2}}^{1/m} \biggr)\eta = 0.\qedhere
  \end{align*}
\end{proof}
We note that the unambiguity statement in \cite[Proposition 3.7]{MS18} for
perfectoid test ideals holds for
the test ideal $\tau_B(R,\Delta,\fa_1^{t_1}\fa_2^{t_2}\cdots\fa_n^{t_n})$ as
well.
\begin{proposition}[cf.\ {\cite[Proposition
  3.7]{MS18}}]\label{prop:notambignogens}
  Fix notation as in Notation \ref{notation:ms21like} and Definition
  \ref{def:ms18like}.
  Fix another set of elements $\{f_1,f_2,\ldots,f_r\}$ and a real number
  $s \ge 0$, and denote by $\fa$ the ideal $(f_1,f_2,\ldots,f_r)$.
  We have
  \[
    \tau_B\bigl(R,\Delta,\fa_1^{t_1}\fa_2^{t_2}\cdots\fa_n^{t_n}(\fa^k)^s\bigr)
    \subseteq
    \tau_B\bigl(R,\Delta,\fa_1^{t_1}\fa_2^{t_2}\cdots\fa_n^{t_n}\fa^{ks}\bigr)
  \]
  for every integer $k \ge 0$, with equality if $s$ is rational.
\end{proposition}
\begin{proof}
  It suffices to show
  \[
    0^{B,K_R+\Delta,\fa_1^{t_1}\fa_2^{t_2}\cdots
    \fa_n^{t_n}(\fa^k)^s}_{H^d_\fm(R)} \supseteq
    0^{B,K_R+\Delta,\fa_1^{t_1}\fa_2^{t_2}\cdots \fa_n^{t_n}
    \fa^{ks}}_{H^d_\fm(R)}
  \]
  with equality when $s$ is rational.
  The inclusion $\supseteq$ holds since
  $(\fa^k)^{\lceil ms \rceil} = \fa^{k\lceil ms \rceil} \subseteq
  \fa^{\lceil mks \rceil}$
  for every $m$.
  When $s$ is rational, by Lemma \ref{lem:ms18lem36}, we may restrict to all $m$
  sufficiently large and divisible such that $ms$ is an integer.
  In this case, the inclusion $\fa^{k\lceil ms \rceil} \subseteq
  \fa^{\lceil mks \rceil}$ is an equality.
\end{proof}

\subsection{Subadditivity}
We now use the
strategy in \cite[Theorem 4.4]{MS18} (which is in turn based on the strategies
in \cite[Theorem 2.4]{Tak06} and \cite[Proposition 2.11$(iv)$]{BMS08})
to prove a stronger version of the subadditivity theorem for the big
Cohen--Macaulay test ideals
$\tau_B(R,[\underline{f}]^t)$ than is proved for perfectoid test ideals in
\cite{MS18}.
\begin{theorem}[{cf.\ \cite[Theorem 4.4]{MS18}}]\label{thm:subadditivity}
  Let $(R,\fm)$ be a complete regular local ring.
  Fix sets
  \[
    \{f_{i,1},f_{i,2},\ldots,f_{i,r_i}\}
  \]
  of elements in $R$ for $i \in \{1,2,\ldots,n\}$,
  and fix real numbers $t_1,t_2,\ldots,t_n \ge 0$.
  For every big Cohen--Macaulay $R^+$-algebra $B$, we have
  \begin{equation}\label{eq:mainsubadditivity}
    \begin{aligned}
      \MoveEqLeft[4]
      \tau_B\bigl(R,[\underline{f_1}]^{t_1}[\underline{f_2}]^{t_2}\cdots
      [\underline{f_n}]^{t_n}\bigr)\\
      &\subseteq
      \tau_B\bigl(R,[\underline{f_1}]^{t_1}[\underline{f_2}]^{t_2}\cdots
      [\underline{f_{i_0}}]^{t_{i_0}}\bigr) \cdot
      \tau_B\bigl(R,[\underline{f_{i_0+1}}]^{t_{i_0+1}}
        [\underline{f_{i_0+2}}]^{t_{i_0+2}}\cdots
      [\underline{f_n}]^{t_n}\bigr)
    \end{aligned}
  \end{equation}
  for every $i_0 \in \{1,2,\ldots,n-1\}$, where we define $\tau_B$
  using the choice of canonical divisor $K_R = 0$.
  In particular, if $\Delta_1$ and $\Delta_2$ are effective
  $\QQ$-divisors on $X = \Spec(R)$, we have
  \begin{align*}
    \tau_B(R,\Delta_1+\Delta_2)
    &\subseteq
    \tau_B(R,\Delta_1) \cdot
    \tau_B(R,\Delta_2)\\
    \tau_+(\cO_X,\Delta_1+\Delta_2)
    &\subseteq
    \tau_+(\cO_X,\Delta_1) \cdot
    \tau_+(\cO_X,\Delta_2).
  \end{align*}
\end{theorem}
We note that the proof of the special case $\tau_+(\cO_X,\Delta_1+\Delta_2)
\subseteq
\tau_+(\cO_X,\Delta_1) \cdot \tau_+(\cO_X,\Delta_2)$ (still assuming $X =
\Spec(R)$) was suggested
in \cite[Bottom of p.\ 32]{HLS}.
\begin{proof}
  Let $d = \dim(R)$.
  The two ``in particular'' statements follow from \eqref{eq:mainsubadditivity}
  by Lemma \ref{lem:testidealasasum}$(\ref{lem:hlscomparison})$ and Remark
  \ref{rem:hls512}, respectively.
  It therefore suffices to show \eqref{eq:mainsubadditivity}.
  \par For \eqref{eq:mainsubadditivity}, we want to show the containment
  \begin{align}
    \begin{split}
      \MoveEqLeft[2]
      0^{B,K_R,[\underline{f_1}]^{t_1}[\underline{f_2}]^{t_2} \cdots
      [\underline{f_n}]^{t_n}}_{H^d_\fm(R)}\\
      &\supseteq \Ann_{H^d_\fm(R)}\Bigl(
        \tau_B\bigl(R,[\underline{f_1}]^{t_1}[\underline{f_2}]^{t_2}\cdots
        [\underline{f_{i_0}}]^{t_{i_0}}\bigr) \cdot
        \tau_B\bigl(R,[\underline{f_{i_0+1}}]^{t_{i_0+1}}
          [\underline{f_{i_0+2}}]^{t_{i_0+2}}\cdots
      [\underline{f_n}]^{t_n}\bigr)\Bigr)
    \end{split}\label{eq:dualofsubadd}
    \intertext{by \cite[Lemma 2.1$(iii)$]{Smi95}.
    We claim it suffices to show that}
    \begin{split}
      \MoveEqLeft[2]
      0^{B,K_R,[\underline{f_1}]^{t_1}[\underline{f_2}]^{t_2} \cdots
      [\underline{f_n}]^{t_n}}_{H^d_\fm(R)}\\
      &\supseteq
      \Set[\bigg]{\eta \in H^d_\fm(R)
        \given \tau_{B}\bigl(R,[\underline{f_1}]^{t_1}[\underline{f_2}]^{t_2}
        \cdots [\underline{f_{i_0}}]^{t_{i_0}}\bigr) \cdot \eta \subseteq
        0^{B,K_R,[\underline{f_{i_0+1}}]^{t_{i_0+1}}
          [\underline{f_{i_0+2}}]^{t_{i_0+2}}\cdots
      [\underline{f_n}]^{t_n}}_{H^d_\fm(R)}}.
    \end{split}\label{eq:ms183}
  \end{align}
  If $\eta$ is in the module on the right-hand side of
  \eqref{eq:dualofsubadd}, then
  \begin{align*}
      \MoveEqLeft[4]
      \tau_B\bigl(R,[\underline{f_1}]^{t_1}[\underline{f_2}]^{t_2}\cdots
      [\underline{f_{i_0}}]^{t_{i_0}}\bigr)\cdot \eta\\
      &\subseteq
      \Ann_{H^d_\fm(R)}\Bigl(\tau_B\bigl(R,[\underline{f_{i_0+1}}]^{t_{i_0+1}}
          [\underline{f_{i_0+2}}]^{t_{i_0+2}}\cdots
      [\underline{f_n}]^{t_n}\bigr)\Bigr)\\
      &= 0^{B,K_R,[\underline{f_{i_0+1}}]^{t_{i_0+1}}
        [\underline{f_{i_0+2}}]^{t_{i_0+2}}\cdots
      [\underline{f_n}]^{t_n}}_{H^d_\fm(R)}
  \end{align*}
  again by \cite[Lemma 2.1$(iii)$]{Smi95}, and hence
  $\eta$ lies in the module on the right-hand side of \eqref{eq:ms183}.
  By \eqref{eq:ms183}, we see that $\eta \in
  0^{B,K_R,[\underline{f_1}]^{t_1}[\underline{f_2}]^{t_2} \cdots
  [\underline{f_n}]^{t_n}}_{H^d_\fm(R)}$, showing \eqref{eq:dualofsubadd} as
  claimed.
  \par It remains to show \eqref{eq:ms183}.
  Suppose $\eta$ is in the module on the right-hand side of \eqref{eq:ms183}.
  By definition, we know that
  \[
    g_{i_0+1}g_{i_0+2}\cdots g_n \eta \cdot
    \tau_{B}\bigl(R,[\underline{f_1}]^{t_1}[\underline{f_2}]^{t_2}
    \cdots [\underline{f_{i_0}}]^{t_{i_0}}\bigr) = 0 \subseteq H^d_\fm(B)
  \]
  for all $g_i = \prod_{k=1}^{a_i} f_{i,j_k}^{1/m_i}$ where $a_i \ge m_it_i$ and
  $i > i_0$.
  We therefore have
  \begin{align*}
    g_{i_0+1}g_{i_0+2}\cdots g_n \eta &\in \Ann_{H^d_\fm(B)}\Bigl(
    \tau_{B}\bigl(R,[\underline{f_1}]^{t_1}[\underline{f_2}]^{t_2}
    \cdots [\underline{f_{i_0}}]^{t_{i_0}}\bigr) \cdot B \Bigr).
  \intertext{Since $R$ is regular and $B$ is a big Cohen--Macaulay algebra over
  $R$, we
  know that $R \to B$ is faithfully flat
  \cite[Lemma 5.5]{Hoc75queens} (see also \citeleft\citen{HH92}\citemid
  (6.7)\citepunct \citen{HH95}\citemid Lemma 2.1$(d)$\citeright), and hence
  taking annihilators commutes with extending scalars to $B$
  \cite[Chapter I, \S2, no.\ 10, Proposition 12 and Remark on
  p.\ 24]{Bou72}.
  Thus, we have}
    g_{i_0+1}g_{i_0+2}\cdots g_n \eta &\in
    B \otimes_R \Ann_{H^d_\fm(R)}\Bigl(
      \tau_{B}\bigl(R,[\underline{f_1}]^{t_1}[\underline{f_2}]^{t_2}
    \cdots [\underline{f_{i_0}}]^{t_{i_0}}\bigr) \Bigr)\\
    \MoveEqLeft[-4] =
    B \otimes_R 0_{H^d_\fm(R)}^{B,K_R,[\underline{f_1}]^{t_1}[\underline{f_2}]^{t_2}
    \cdots [\underline{f_{i_0}}]^{t_{i_0}}}
  \end{align*}
  again by \cite[Lemma 2.1$(iii)$]{Smi95}, and we can write
  \[
    g_{i_0+1}g_{i_0+2}\cdots g_n \eta = b_1\eta_1 + b_2\eta_2 + \cdots + b_\ell
    \eta_\ell \in H^d_\fm(B)
  \]
  where $b_1,b_2,\ldots,b_\ell \in B$ and $\eta_1,\eta_2,\ldots,\eta_\ell \in
  0_{H^d_\fm(R)}^{B,K_R,[\underline{f_1}]^{t_1}[\underline{f_2}]^{t_2}
  \cdots [\underline{f_{i_0}}]^{t_{i_0}}}$.
  Finally, multiplying this sum
  by any product of elements $g_1g_2\cdots g_{i_0}$ where
  $g_i = \prod_{k=1}^{a_i} f_{i,j_k}^{1/m_i}$ where $a_i \ge m_it_i$ and $i \le
  i_0$, we see that
  \begin{align*}
    \MoveEqLeft[3]
    g_1g_2\cdots g_{i_0}g_{i_0+1}g_{i_0+2}\cdots g_n \eta\\
    &=
    b_1(g_1g_2\cdots g_{i_0}\eta_1) + b_2(g_1g_2\cdots g_{i_0}\eta_2)
    + \cdots + b_\ell(g_1g_2\cdots g_{i_0}\eta_\ell)\\
    &= 0
  \end{align*}
  in $H^d_\fm(B)$,
  and hence $\eta \in 0^{B,K_R,[\underline{f_1}]^{t_1}[\underline{f_2}]^{t_2} \cdots
  [\underline{f_n}]^{t_n}}_{H^d_\fm(R)}$ as claimed in \eqref{eq:ms183}.
\end{proof}
\subsection{Comparison with multiplier ideals}
Finally, we show that Robinson's big Cohen--Macaulay test ideals are contained in
multiplier ideals.
While we will only use this comparison in the proof of Theorem
\ref{thm:ty23like} in equal characteristic zero, the proof yields a statement
valid in arbitrary characteristic.
Ma and Schwede showed a version of this statement when $\Delta = 0$ for
perfectoid test ideals in mixed characteristic $p > 0$ in \cite[Theorem
6.3]{MS18}.
The statement for big Cohen--Macaulay test ideals of pairs $(R,\Delta)$
in residue characteristic $p > 0$ is also due to Ma and Schwede
\cite[Proposition 3.7 and Theorem 6.21]{MS} (see also \cite[Theorem
5.1]{MSTWW}).
The statement below is due to Robinson in residue characteristic $p > 0$ when $n
= 1$ \cite[Theorem 3.9]{Rob}.
A version of this statement for pairs $(R,\Delta)$ appears in the arXiv version
of \cite{MSTWW} (see footnote \ref{footnote:mstww} on page
\pageref{footnote:mstww} of this paper), although in equal characteristic zero
the big Cohen--Macaulay
algebra constructed therein is not necessarily an $R^+$-algebra.
\par We note that recently, Yamaguchi showed that various versions of big
Cohen--Macaulay test ideals for normal local rings $(R,\fm)$ essentially of
finite type over $\CC$ coincide with the multiplier ideal
\cite[Theorem 6.4 and Proposition 7.7]{Yam} when $B$ is chosen to be the
$\fm$-adic completion of the big Cohen--Macaulay algebra constructed by
Schoutens for this class of rings in \cite[Theorem A]{Sch04}.
\begin{theorem}[cf.\ {\citeleft\citen{MS18}\citemid Theorem 6.3\citepunct
  \citen{MS}\citemid Proposition 3.7 and Theorem 6.21\citepunct
  \citen{MSTWW}\citemid Theorem 5.1\citepunct
  \citen{Rob}\citemid Theorem 3.9\citeright}]\label{thm:compmultideals}
  Fix notation as in Notation \ref{notation:ms21like} and Definition
  \ref{def:ms18like}.
  Let $\mu\colon Y \to \Spec(R)$ be a proper birational morphism with $Y$
  normal such that for every $i$, there exists a Weil divisor $F_i$ on $Y$ for
  which
  \[
    \frac{1}{m_i} \mu^*\prdiv_R(g_i) \ge t_iF_i
  \]
  for every $g_i \in \fa_i^{\lceil m_it_i \rceil}$ as $m_i$ ranges over all
  positive integers.
  Then, there exists an $R^+$-algebra $B$ that is big Cohen--Macaulay over $R$
  such that
  \[
    \tau_B\bigl(R,\Delta,\fa_1^{t_1}\fa_2^{t_2}\cdots\fa_n^{t_n}\bigr) \subseteq
    \mu_*\cO_Y\biggl(K_Y - \biggl\lfloor \mu^*(K_R+\Delta) + \sum_{i=1}^n t_iF_i
    \biggr\rfloor \biggr).
  \]
  In particular, if $R$ is of equal characteristic zero, there exists
  an $R^+$-algebra $B$ that is big Cohen--Macaulay over $R$
  such that for all choices of $t_1,t_2,\ldots,t_n \ge 0$, we have
  \[
    \tau_B\bigl(R,\Delta,\fa_1^{t_1}\fa_2^{t_2}\cdots\fa_n^{t_n}\bigr) \subseteq
    \cJ\bigl(R,\Delta,\fa_1^{t_1}\fa_2^{t_2}\cdots\fa_n^{t_n}\bigr)
  \]
\end{theorem}
We note that the hypothesis that the $F_i$ exist holds in particular when
$\mu^{-1}\fa_i \cdot \cO_Y = \cO_Y(-F_i)$ for some Cartier divisors $F_i$ on
$Y$.
\begin{proof}
  First, we claim that we may assume that $\mu$ is projective, and that if we
  write $Y = \Proj_R(R[Jt])$ for an ideal $J \subseteq R$, we may assume that
  $R[Jt]$ is normal.
  We adapt the proof of \cite[Theorem 5.1]{MSTWW}.
  By Chow's lemma, there exists a projective birational morphism $\mu'\colon
  Y' \to Y$ such that $\mu\circ\mu'$ is projective.
  Now writing $Y' = \Proj_R(R[J't])$, we may replace $R[J't]$ by its
  normalization to assume that $R[J't]$ is normal, in which case
  $\Proj_R(R[J't])$ is also normal.
  Finally, we note that
  \begin{align*}
    \MoveEqLeft[5]
    (\mu\circ\mu')_*\cO_{Y'}\biggl(K_{Y'} - \biggl\lfloor (\mu \circ \mu')^*
    (K_R+\Delta) + \sum_{i=1}^n t_i\mu^{\prime*}F_i \biggr\rfloor \biggr)\\
    &\subseteq
    \mu_*\cO_Y\biggl(K_Y - \biggl\lfloor \mu^*(K_R+\Delta) + \sum_{i=1}^n t_iF_i
    \biggr\rfloor \biggr)
  \end{align*}
  by the projection formula since $\mu'_*\cO_{Y'}(K_{Y'}) \subseteq \cO_Y(K_Y)$.
  We may therefore replace $Y$ by $Y'$ to assume that $\mu$ is projective, and
  that writing $Y = \Proj_R(R[Jt])$, the Rees algebra $R[Jt]$ is normal.
  \par We now construct the big Cohen--Macaulay algebra $B$, following the
  proof of \cite[Proposition 3.7]{MS}.
  Let $S = R[Jt]$ and let $\fn$ denote the maximal ideal $\fm + Jt$, and
  consider the surjective ring map $\widehat{S_\fn} \to R$.
  Since $R$ is a domain, this map factors through $\widehat{S_\fn}/\fp$, where
  $\fp$ is a minimal prime of $\widehat{S_\fn}$, and $\dim(\widehat{S_\fn}/\fp)
  = \dim(R) + 1$.
  We now note that $\widehat{S_\fn}/\fp$ and $R$ have the same characteristic.
  Thus, by \citeleft\citen{HH92}\citemid Theorem 1.1\citepunct
  \citen{HH95}\citemid Proposition 1.2\citeright\ in equal characteristic $p >
  0$,
  \cite[Theorem 1.2.1]{And20} in mixed characteristic, and Theorem
  \ref{thm:weaklyfunctorialequalchar0} in equal characteristic zero, there
  exists a commutative diagram
  \[
    \begin{tikzcd}
      B' \rar & B\\
      \bigl(\widehat{S_\fn}/\fp\bigr)^+ \rar\uar & R^+\uar\\
      \widehat{S_\fn}/\fp \rar\uar & R\uar
    \end{tikzcd}
  \]
  where $B'$ is big Cohen--Macaulay over $\widehat{S_\fn}/\fp$ and $B$ is big
  Cohen--Macaulay over $R$.
  \par We now prove the theorem for the choice of $B$ constructed in the
  previous paragraph.
  By Lemma \ref{lem:testidealasasum}$(\ref{lem:taubasasum})$, we have
  \[
    \tau_B\bigl(R,\Delta,\fa_1^{t_1}\fa_2^{t_2}\cdots\fa_n^{t_n}\bigr) =
    \sum_{m_1,m_2,\ldots,m_n \in
    \ZZ_{>0}} \sum_{g_i \in \fa_i^{\lceil m_it_i \rceil}}
    \tau_B\biggl(R,\Delta+\sum_{i=1}^n \frac{1}{m_i}\prdiv_R(g_i) \biggr).
  \]
  By the proof of \cite[Theorem 5.1]{MSTWW}\footnote{See also Remark 5.1.1 in
  the arXiv version of \cite{MSTWW} available at
  \url{https://arxiv.org/abs/1910.14665v5} in the equal characteristic zero
  case.\label{footnote:mstww}}, we see that
  \[
    \tau_B\biggl(R,\Delta+\sum_{i=1}^n \frac{1}{m_i}\prdiv_R(g_i) \biggr)
    \subseteq
    \mu_*\cO_Y\biggl(K_Y - \biggl\lfloor \mu^*(K_R+\Delta) + \sum_{i=1}^n
    \frac{1}{m_i} \mu^*\prdiv_R(g_i) \biggr\rfloor \biggr)
  \]
  for every $m_i$ and $g_i$.
  Since
  $\frac{1}{m_i} \mu^*\prdiv_R(g_i)
  \ge t_iF_i$,
  we see that
  \[
    \tau_B\biggl(R,\Delta+\sum_{i=1}^n \frac{1}{m_i}\prdiv_R(g_i) \biggr)
    \subseteq 
    \mu_*\cO_Y\biggl(K_Y - \biggl\lfloor \mu^*(K_R+\Delta) + \sum_{i=1}^n t_iF_i
    \biggr\rfloor \biggr).
  \]
  \par Finally, the last statement when $R$ is of equal characteristic zero
  follows by choosing $\mu$ to be a log resolution for the triple
  $(R,\Delta,\fa_1\fa_2\cdots\fa_n)$, since in this case
  \[
    \frac{1}{m_i} \mu^*\prdiv_R(g_i)
    \ge \frac{\lceil m_it_i \rceil}{m_i} F_i
    \ge t_iF_i
  \]
  for all choices of $t_1,t_2,\ldots,t_n \ge 0$.
\end{proof}

\section{Proof of main theorems via multiplier/test ideals}\label{sect:mainproofs}
Our goal in this section is to prove Theorems \ref{thm:mainelshhms},
\ref{thm:mainjohnson}, and \ref{thm:mainjohnsonhhty}
using Robinson's version of big Cohen--Macaulay
test ideals and our results for these test ideals
from \S\ref{sect:bcmtestfixedgens}.
In equal characteristic zero, we also use multiplier ideals.
The idea is to use our results for the test ideals
$\tau_B(R,\Delta,[\underline{f}]^t)$ and
$\tau_B(R,\Delta,\fa^t)$
proved above, instead of those for
multiplier ideals or for existing versions of test ideals used in
\cite[Theorem 2.2 and Variant
on p.\ 251]{ELS01}, \cite[Theorem 2.12]{Har05}, \cite[Theorem 3.1]{TY08}, and
\cite[Theorem 7.4]{MS18} (see also \cite[Theorem 6.23]{ST12})
to prove previously known cases of Theorem \ref{thm:mainelshhms}.
\par To prove Theorems \ref{thm:mainelshhms}, \ref{thm:mainjohnson}, and
\ref{thm:mainjohnsonhhty},
we prove Theorem \ref{thm:ty23like} using
the test ideals $\tau_B(R,\Delta,[\underline{f}]^t)$ and
$\tau_B(R,\Delta,\fa^t)$
and the big Cohen--Macaulay test ideals from
\cite{MS,PRG,ST}, together with
the $+$-test
ideals from \cite{HLS} in residue characteristic $p > 0$ and with
multiplier ideals in equal characteristic zero as developed in
\cite{dFM09,JM12,ST}.
As mentioned in \S\ref{sect:intropartii}, we also give a proof of Theorems
\ref{thm:mainjohnson} and \ref{thm:mainjohnsonhhty} in equal characteristic zero
using only multiplier ideals.
We present these proofs in equal characteristic zero using multiplier ideals in
\S\ref{subsection:proofsusingmult}, followed by the proofs in arbitrary
characteristic in \S\ref{subsection:generalproofs}.
\par Since Theorem \ref{thm:mainelshhms} follows
from Theorem \ref{thm:mainjohnson} by setting
$s_i = 0$ for all $i$, it suffices to show Theorems \ref{thm:mainjohnson},
\ref{thm:mainjohnsonhhty}, and \ref{thm:ty23like}.
\subsection{Proof via multiplier ideals in equal characteristic zero}
\label{subsection:proofsusingmult}
We prove Theorems \ref{thm:mainjohnson} and \ref{thm:mainjohnsonhhty} in
equal characteristic zero using multiplier ideals.
The proof below uses
multiplier ideals (see Definition \ref{def:multiplierideal}) and
illustrates the strategy we will want to use for Theorems \ref{thm:mainjohnson},
\ref{thm:mainjohnsonhhty}, and \ref{thm:ty23like} in arbitrary characteristic.
Our proof also yields the first proof of Theorem \ref{thm:mainelshhms} for
all regular rings of equal characteristic zero that does not rely on the
N\'eron-type desingularization theorem due to Artin and Rotthaus \cite[Theorem
1]{AR88} or stronger results.
This proof therefore answers a question of Schoutens \cite[p.\ 179 and p.\
187]{Sch03}, who asked whether one could show Theorem \ref{thm:mainelshhms} in
equal characteristic zero without using \cite[Theorem 1]{AR88}.\medskip
\par We first prove the analogue of Theorem \ref{thm:ty23like} for multiplier
ideals.
\begin{proposition}[cf.\ {\citeleft\citen{ELS01}\citemid Proof of Variant on p.\
  251\citepunct \citen{TY08}\citemid Proof of Theorem 4.1\citeright}]
  \label{prop:ty23likechar0}
  Let $R$ be an excellent normal domain of equal characteristic zero
  with a dualizing complex $\omega_R^\bullet$ and associated choice of canonical
  divisor $K_R$.
  Suppose that $K_R$ is $\QQ$-Cartier, and set $X = \Spec(R)$.
  \par Consider an ideal $I\subseteq R$, and let $h$ be the largest analytic
  spread of $IR_\fp$, where $\fp$ ranges over all associated primes of
  $R/I$.
  Then, for every integer $M > 0$, we have
  \begin{equation}\label{eq:lastinclusionwithmultiplier}
    \cJ\Bigl(X,\bigl(I^{(M)}\bigr)^{\frac{s_i+h}{M}}\Bigr) \subseteq
    I^{(s_i+1)}
  \end{equation}
  for every integer $s_i \ge 0$.
\end{proposition}
\begin{proof}
  \par Let $\{\fp_\ell\}$ denote the associated primes of $R/I$.
  It suffices to show that \eqref{eq:lastinclusionwithmultiplier} holds after
  localizing at every $\fp_\ell$.
  Since the residue field of $R_{\fp_\ell}$ contains $\QQ$, it is infinite,
  and hence for every $\ell$
  there exists an ideal $J_\ell \subseteq I$ with at most $h$ generators such
  that $J_\ell R_{\fp_\ell}$ is a reduction for
  $IR_{\fp_\ell}$
  by
  \cite[Proposition 8.3.7 and Corollary 1.2.5]{SH06}.
  Note that after localizing at $\fp_\ell$, the ordinary and symbolic powers of
  $I$ coincide.
  Setting $X_{\fp_\ell} = \Spec(R_{\fp_\ell})$, which is excellent by
  \cite[Scholie 7.8.3$(ii)$]{EGAIV2}, we have
  \begin{align*}
    \cJ\Bigl(X,\bigl(I^{(M)}\bigr)^{\frac{s_i+h}{M}}\Bigr) \cdot R_{\fp_\ell}
    &= \cJ\Bigl(X_{\fp_\ell},\bigl(I^{(M)}R_{\fp_\ell}
    \bigr)^{\frac{s_i+h}{M}}\Bigr)\\
    &= \cJ\Bigl(X_{\fp_\ell},\bigl(I^{M}R_{\fp_\ell}
    \bigr)^{\frac{s_i+h}{M}}\Bigr)\\
    &= \cJ(X_{\fp_\ell},I^{s_i+h}R_{\fp_\ell})
  \end{align*}
  since the formation of multiplier ideals is compatible with localization by
  \citeleft\citen{dFM09}\citemid Proposition 2.2\citepunct
  \citen{ST}\citemid Remark 2.5$(ii)$\citeright.
  Finally, by the Skoda-type result in \cite[Theorem 9.6.36]{Laz04b} (whose
  proof holds in this generality using \cite[Theorem A]{Mur} instead of the
  local vanishing theorem \cite[Variant 9.4.4]{Laz04b}), we have
  \[
    \cJ(X_{\fp_\ell},I^{s_i+h}R_{\fp_\ell}) = J_\ell^{s_i+1} R_{\fp_\ell} \cdot
    \cJ(X_{\fp_\ell},I^{h-1}R_{\fp_\ell}) \subseteq I^{s_i+1}
    R_{\fp_\ell},
  \]
  where the last inclusion holds since $J_\ell \subseteq I$ and
  $\cJ(X_{\fp_\ell},I^{h-1}R_{\fp_\ell}) \subseteq R_{\fp_\ell}$.
\end{proof}
We now show Theorems \ref{thm:mainjohnson} and
\ref{thm:mainjohnsonhhty} in equal characteristic zero.
\begin{proof}[Proof of Theorems \ref{thm:mainjohnson} and
  \ref{thm:mainjohnsonhhty} in equal characteristic zero via multiplier ideals]
  By Lemma \ref{lem:hh02reductions}, we may assume that $R$ is a complete local
  ring, which is excellent by \cite[Scholie 7.8.3$(iii)$]{EGAIV2}.
  Note that Step \ref{step:infiniteresidues} of Lemma
  \ref{lem:hh02reductions} is unnecessary in this case, since the complete local
  ring produced in Step \ref{step:completelocal} contains $\QQ$.
  \par Set $X = \Spec(R)$.
  We start with Theorem \ref{thm:mainjohnson}.
  We have
  \begin{align*}
    I^{(s+nh)} &= I^{(s+nh)} \cdot \cJ(X,R^1)
    \subseteq \cJ\Bigl(X,\bigl(I^{(s+nh)}\bigr)^1\Bigr),
  \intertext{where the first equality holds by \cite[Proposition
  2.3(4)]{dFM09}, and the second
  inclusion holds by Definition \ref{def:multiplierideal}.
  Next, by Definition \ref{def:multiplierideal}, we have}
    \cJ\Bigl(X,\bigl(I^{(s+nh)}\bigr)^1\Bigr) &=
    \cJ\Bigl(X,\bigl(I^{(s+nh)}\bigr)^{\frac{s+nh}{s+nh}}\Bigr)\\
    &= \cJ\Bigl(X,\bigl(I^{(s+nh)}\bigr)^{\frac{s_1+h}{s+nh}}
    \bigl(I^{(s+nh)}\bigr)^{\frac{s_2+h}{s+nh}} \cdots
    \bigl(I^{(s+nh)}\bigr)^{\frac{s_n+h}{s+nh}}\Bigr)\\
    &\subseteq \prod_{i=1}^n
    \cJ\Bigl(X,\bigl(I^{(s+nh)}\bigr)^{\frac{s_i+h}{s+nh}}\Bigr),
  \end{align*}
  where the last inclusion holds by applying
  the subadditivity theorem \cite[Theorem A.2]{JM12} $(n-1)$ times.
  We can therefore apply Proposition \ref{prop:ty23likechar0} when $M = s+nh$
  and combine these inclusions to show \eqref{eq:mainjohnsonincl}.
  \par We now prove Theorem \ref{thm:mainjohnsonhhty}.
  Let $l$ be the largest integer
  such that
  \[
    \cJ\Bigl(X,\bigl(I^{(s+nh+1)}\bigr)^{\frac{l}{s+nh+1}}\Bigr) = R.
  \]
  Note that such an $l$ exists since $l = 0$ satisfies this equality by
  \cite[Proposition 2.3(4)]{dFM09}.
  We then have
  \begin{align*}
    I^{(s+nh+1)} &= I^{(s+nh+1)} \cdot
    \cJ\Bigl(X,\bigl(I^{(s+nh+1)}\bigr)^{\frac{l}{s+nh+1}}\Bigr)\\
    &\subseteq
    \cJ\Bigl(X,\bigl(I^{(s+nh+1)}\bigr)^{\frac{l}{s+nh+1}}
    \bigl(I^{(s+nh+1)}\bigr)^{\frac{s+nh+1}{s+nh+1}}\Bigr)\\
    &= \cJ\Bigl(X,\bigl(I^{(s+nh+1)}\bigr)^{\frac{s+nh+l+1}{s+nh+1}}\Bigr)\\
    &= \cJ\Bigl(X,\bigl(I^{(s+nh+1)}\bigr)^{\frac{l+1}{s+nh+1}}
    \bigl(I^{(s+nh+1)}\bigr)^{\frac{s_1+h}{s+nh+1}} \cdots
    \bigl(I^{(s+nh+1)}\bigr)^{\frac{s_n+h}{s+nh+1}}\Bigr)
  \intertext{by Definition \ref{def:multiplierideal}.
  Applying the subadditivity theorem \cite[Theorem
  A.2]{JM12} $n$ times, we have}
    \MoveEqLeft[5]
    \cJ\Bigl(X,\bigl(I^{(s+nh+1)}\bigr)^{\frac{l+1}{s+nh+1}}
    \bigl(I^{(s+nh+1)}\bigr)^{\frac{s_1+h}{s+nh+1}} \cdots
    \bigl(I^{(s+nh+1)}\bigr)^{\frac{s_n+h}{s+nh+1}}\Bigr)\\
    &\subseteq \cJ\Bigl(X,\bigl(I^{(s+nh+1)}\bigr)^{\frac{l+1}{s+nh+1}}\Bigr)
    \cdot \prod_{i=1}^n
    \cJ\Bigl(X,\bigl(I^{(s+nh+1)}\bigr)^{\frac{s_i+h}{s+nh+1}}\Bigr)\\
    &\subseteq \fm \cdot \prod_{i=1}^n I^{(s_i+1)},
  \end{align*}
  where
  the last inclusion holds since
  $\cJ(X,(I^{(s+nh+1)})^{\frac{l+1}{s+nh+1}}) \subseteq \fm$
  by our assumption on $l$ and by applying Proposition \ref{prop:ty23likechar0}
  when $M = s+nh+1$.
  We can therefore combine these inclusions to show
  \eqref{eq:mainjohnsonhhtyincl}.
\end{proof}
\subsection{Proof via big Cohen--Macaulay test ideals in all characteristics}
\label{subsection:generalproofs}
We now prove Theorems \ref{thm:mainjohnson}, \ref{thm:mainjohnsonhhty}, and
\ref{thm:ty23like}
in all characteristics using Robinson's version of big Cohen--Macaulay test
ideals from \cite{Rob} and the results we proved for them in
\S\ref{sect:bcmtestfixedgens}.\medskip
\par We start by proving Theorem \ref{thm:ty23like}.
To prove Theorem \ref{thm:ty23like}, we use Robinson's test ideals from
\cite{Rob} and the
big Cohen--Macaulay test ideals
from \cite{MS,PRG,ST}, together with the $+$-test ideals from \cite{HLS} in
residue characteristic $p > 0$ and multiplier ideals in equal
characteristic zero.
In the proof below, $\widehat{R^+}$ denotes
the $p$-adic completion of the absolute integral
closure $R^+$ of a complete local ring $R$ of residue characteristic $p > 0$.
As in Remark \ref{rem:bigcmalgexist}$(\ref{rem:bigcmalgexistrplusmixedchar})$,
$\widehat{R^+}$
is a big Cohen--Macaulay $R^+$-algebra by \citeleft\citen{HH92}\citemid Theorem
8.1\citepunct \citen{Bha}\citemid Corollary 5.17\citeright.
\begin{proof}[Proof of Theorem \ref{thm:ty23like}]
  We first consider the case when $R$ is of residue characteristic $p > 0$, in
  which case we will show that setting $B = \widehat{R^+}$ suffices.
  Let $\{\fp_\ell\}$ denote the associated primes of $R/I$.
  It suffices to show that \eqref{eq:lastinclusionwithtaub} holds after
  localizing at every $\fp_\ell$.
  Since the residue field of $R_{\fp_\ell}$ is infinite, for every $\ell$
  there exists an ideal $J_\ell \subseteq I$ with at most $h$ generators such
  that $J_\ell R_{\fp_\ell}$ is a reduction of
  $IR_{\fp_\ell}$
  by
  \cite[Proposition 8.3.7 and Corollary 1.2.5]{SH06}.
  Note that after localizing at $\fp_\ell$, the ordinary and symbolic powers of
  $I$ coincide, and the integral closures of $I$ and $J_\ell$ coincide.
  Since $R$ is Noetherian, for every $\ell$, we can then
  choose an element $x_\ell \notin \fp_\ell$ such that
  \begin{equation}\label{eq:compareideals}
    I^{(M)}R_{x_\ell} = I^{M}R_{x_\ell} \qquad \text{and} \qquad
    \overline{I}R_{x_\ell} = \overline{J_\ell}R_{x_\ell}.
  \end{equation}
  Since $x_\ell \notin \fp_\ell$, it suffices to show that
  \eqref{eq:lastinclusionwithtaub} holds after inverting each $x_\ell$.
  \par By Lemma \ref{lem:testidealasasum}$(\ref{lem:hlscomparison})$,
  we have
  \begin{align*}
    \tau_{\widehat{R^+}}\Bigl(R,\bigl(I^{(M)}\bigr)^{\frac{s_i+h}{M}}\Bigr)
    &\subseteq \sum_{m=1}^\infty
    \sum_{g \in (I^{(M)})^m}
    \tau_{\widehat{R^+}}\biggl(R,\frac{s_i+h}{mM}
    \prdiv_R(g) \biggr).
  \intertext{We now invert $x_\ell$ and use the test ideal theory from
  \cite{HLS} (see \S\ref{sect:hls}).
  Setting $X = \Spec(R)$ and $U = \Spec(R_{x_\ell})$, we have
  $\tau_{\widehat{R}}(R,-) = \tau_+(\cO_X,-)$ and $\tau_{\widehat{R}}(R,-) \cdot
  R_{x_\ell} = \tau_+(\cO_U,-\rvert_U)$ by Remark \ref{rem:hls512} and
  Definition \ref{def:hlsqproj}.
  For fixed $m$, we have}
    \MoveEqLeft[8]\sum_{g \in (I^{(M)})^m}
    \tau_{\widehat{R^+}}\biggl(R,\frac{s_i+h}{mM}
    \prdiv_R(g) \biggr) \cdot R_{x_\ell}\\
    &= \sum_{g \in (I^{(M)})^m}
    \tau_+\biggl(\cO_U,\frac{s_i+h}{mM}
    \prdiv_U(g) \biggr)\\
    &= \sum_{\makebox[\widthof{$\scriptstyle g \in
    (I^{(M)})^m$}][c]{$\scriptstyle g \in
    (I^{M})^m$}}
    \tau_+\biggl(\cO_U,\frac{s_i+h}{mM}
    \prdiv_U(g) \biggr)\\
    &\subseteq 
    \sum_{\makebox[\widthof{$\scriptstyle g \in
    (I^{(M)})^m$}][c]{$\scriptstyle g \in
    (\overline{J_\ell}^{M})^m$}}
    \tau_+\biggl(\cO_U,\frac{s_i+h}{mM}
    \prdiv_U(g) \biggr)
  \intertext{by \cite[Corollary 5.8]{HLS}, since the ideals
  $(I^{(M)})^m$ and $(I^{M})^m$ coincide
  after inverting $x_\ell$ by \eqref{eq:compareideals}, as is also the case for
  $(\overline{J_\ell}^{M})^m$ and $(\overline{I}^{M})^m$.
  We therefore have}
  \tau_{\widehat{R^+}}\Bigl(R,\bigl(I^{(M)}\bigr)^{\frac{s_i+h}{M}}\Bigr)
    \cdot R_{x_\ell}
    &\subseteq \sum_{m=1}^\infty
    \sum_{g \in (\overline{J_\ell}^{M})^m}
    \tau_+\biggl(\cO_X,\frac{s_i+h}{mM}
    \prdiv_R(g) \biggr) \cdot R_{x_\ell}\\
    &\subseteq \sum_{m=1}^\infty
    \sum_{g \in \overline{J_\ell}^m}
    \tau_+\biggl(\cO_X,\frac{s_i+h}{m}
    \prdiv_R(g) \biggr) \cdot R_{x_\ell}.
  \end{align*}
  Next, by Lemma \ref{lem:hls64}, we have
  \begin{align*}
    \sum_{m=1}^\infty
    \sum_{g \in \overline{J_\ell}^m}
    \tau_+\biggl(\cO_X,\frac{s_i+h}{m}
    \prdiv_R(g) \biggr)
    &\subseteq \tau_+\bigl(\cO_X,\overline{J_\ell}^{s_i+h}\bigr)\\
    &= \tau_+(\cO_X,J_\ell^{s_i+h})\\
    &= J_\ell^{s_i+1} \cdot \tau_+(\cO_X,J_\ell^{h-1}),
  \end{align*}
  where the middle equality is by definition of $\tau_+$ (see the last
  paragraph in \cite[Remark 6.4]{HLS}), and the last equality is by
  the Skoda-type result in \cite[Theorem 6.6]{HLS} (see also \cite[Footnote on
  p.\ 27]{HLS}).
  Here, we use the fact that the analytic spread of $J_\ell$ is bounded above by
  the number of generators of $J_\ell$ \cite[Corollary 8.2.5]{SH06}, which is at
  most $h$.
  Inverting $x_\ell$ in this last inclusion and combining the containments so
  far, we obtain
  \[
    \tau_{\widehat{R^+}}\Bigl(R,\bigl(I^{(M)}\bigr)^{\frac{s_i+h}{M}}\Bigr)
    \cdot R_{x_\ell}
    \subseteq
    J_\ell^{s_i+1} \cdot \tau_+(X,J_\ell^{h-1}) \cdot R_{x_\ell}
    \subseteq I^{s_i+1}R_{x_\ell}
  \]
  where the last inclusion holds since $J_\ell \subseteq I$ and
  $\tau_+(X,J_\ell^{h-1}) \subseteq R$.\smallskip
  \par We now consider the case when $R$ is of equal characteristic zero.
  Set $X = \Spec(R)$, which is excellent by \cite[Scholie
  7.8.3$(iii)$]{EGAIV2}.
  Let $B_\natural$ be the $R^+$-algebra that is big Cohen--Macaulay over $R$ as
  constructed in Theorem \ref{thm:compmultideals}.
  By Lemma \ref{lem:testidealasasum}$(\ref{lem:hlscomparison})$ and
  Theorem \ref{thm:compmultideals}, we then have
  \[
    \tau_{B_\natural}\Bigl(R,\bigl(I^{(M)}\bigr)^{\frac{s_i+h}{M}}\Bigr)
    \subseteq \cJ\Bigl(X,\bigl(I^{(M)}
    \bigr)^{\frac{s_i+h}{M}}\Bigr)
    \subseteq I^{(s_i+1)},
  \]
  where the last inclusion holds by Proposition \ref{prop:ty23likechar0}.
\end{proof}
Finally, we prove Theorems \ref{thm:mainjohnson} and
\ref{thm:mainjohnsonhhty}.
\begin{proof}[Proof of Theorems \ref{thm:mainjohnson} and
  \ref{thm:mainjohnsonhhty} via big Cohen--Macaulay test ideals]
  By Lemma \ref{lem:hh02reductions}, we may assume that $R$ is a complete local
  ring and that the localizations of $R$ at the associated primes of $R/I$ have
  infinite residue fields.
  Throughout the proof, we define $\tau_B$ using the choice of canonical divisor
  $K_R = 0$ and the big Cohen--Macaulay algebra $B$ that satisfies the
  conclusion of Theorem \ref{thm:ty23like}.
  \par We start with Theorem \ref{thm:mainjohnson}.
  Set $M = s+nh$, and choose generators $f_1,f_2,\ldots,f_r$ for $I^{(s+nh)}$.
  We have
  \begin{align*}
    I^{(s+nh)}
    &= (f_1,f_2,\ldots,f_r)
    \subseteq \tau_{B}\bigl(R,[\underline{f}]^1\bigr)
  \intertext{by Proposition \ref{prop:taubmultiplied}
  (with notation as in Definition \ref{def:ms18like}).
  We then have}
    \tau_{B}\bigl(R,[\underline{f}]^1\bigr) &=
    \tau_{B}\Bigl(R,[\underline{f}]^{\frac{s+nh}{s+nh}}\Bigr)\\
    &= \tau_{B}\Bigl(R,[\underline{f}]^{\frac{s_1+h}{s+nh}}
    [\underline{f}]^{\frac{s_2+h}{s+nh}}
    \cdots
    [\underline{f}]^{\frac{s_n+h}{s+nh}}\Bigr)\\
    &\subseteq \prod_{i=1}^n
    \tau_{B}\Bigl(R,[\underline{f}]^{\frac{s_i+h}{s+nh}}\Bigr),
  \end{align*}
  where the second equality holds by the unambiguity-type statement in
  Proposition \ref{prop:notambig}, and the last inclusion holds by applying
  our version
  of the subadditivity theorem (Theorem \ref{thm:subadditivity}) $(n-1)$ times.
  We can therefore apply Proposition \ref{prop:ms18comparison}
  and Theorem \ref{thm:ty23like} when $M = s+nh$ and
  combine these inclusions to show \eqref{eq:mainjohnsonincl}.
  \par We now prove Theorem \ref{thm:mainjohnsonhhty}.
  Set $M = s+nh+1$, and choose generators $f_1,f_2,\ldots,f_r$ for
  $I^{(s+nh+1)}$.
  Let $l$ be the largest integer such that
  \[
    \tau_{B}\Bigl(R,[\underline{f}]^{\frac{l}{s+nh+1}}\Bigr) = R.
  \]
  Note that such an $l$ exists since $l = 0$ satisfies this equality by the
  faithful flatness of $R \to B$ \cite[Lemma 5.5]{Hoc75queens}
  (see also \citeleft\citen{HH92}\citemid
  (6.7)\citepunct \citen{HH95}\citemid Lemma 2.1$(d)$\citeright), and hence
  $H^d_\fm(R) \to H^d_\fm(B)$ is injective.
  We then have
  \begin{align*}
    I^{(s+nh+1)}
    &= (f_1,f_2,\ldots,f_r) \cdot
    \tau_{B}\Bigl(R,[\underline{f}]^{\frac{l}{s+nh+1}}\Bigr)\\
    &\subseteq
    \tau_{B}\Bigl(R,[\underline{f}]^{\frac{l}{s+nh+1}}
    [\underline{f}]^{\frac{s+nh+1}{s+nh+1}}\Bigr)\\
    &=
    \tau_{B}\Bigl(R,[\underline{f}]^{\frac{s+nh+l+1}{s+nh+1}}\Bigr)\\
    &= \tau_{B}\Bigl(R,[\underline{f}]^{\frac{l+1}{s+nh+1}}
    [\underline{f}]^{\frac{s_1+h}{s+nh+1}} \cdots
  [\underline{f}]^{\frac{s_n+h}{s+nh+1}}\Bigr),
  \intertext{where the inclusion holds by Proposition \ref{prop:taubmultiplied}
  (with notation as in Definition \ref{def:ms18like}), and the last two
  equalities hold by the unambiguity-type statement in Proposition
  \ref{prop:notambig}.
  By applying our version of the subadditivity theorem (Theorem
  \ref{thm:subadditivity}) $n$ times, we have}
    \MoveEqLeft[5]
    \tau_{B}\Bigl(R,[\underline{f}]^{\frac{l+1}{s+nh+1}}
    [\underline{f}]^{\frac{s_1+h}{s+nh+1}} \cdots
    [\underline{f}]^{\frac{s_n+h}{s+nh+1}}\Bigr)\\
    &\subseteq \tau_{B}\Bigl(R,[\underline{f}]^{\frac{l+1}{s+nh+1}}\Bigr) \cdot
    \prod_{i=1}^n \tau_{B}\Bigl(R,[\underline{f}]^{\frac{s_i+h}{s+nh+1}}\Bigr)\\
    &\subseteq \fm \cdot \prod_{i=1}^n I^{(s_i+1)},
  \end{align*}
  where the last inclusion holds since
  $\tau_{B}(R,[\underline{f}]^{\frac{l+1}{s+nh+1}}) \subseteq \fm$
  by our assumption on $l$ and by applying Proposition \ref{prop:ms18comparison}
  and Theorem \ref{thm:ty23like} when $M = s+nh+1$.
  We can therefore combine these inclusions to show
  \eqref{eq:mainjohnsonhhtyincl}.
\end{proof}

\addtocontents{toc}{\protect\medskip}
\bookmarksetup{startatroot}

\end{document}